\documentclass[11pt,a4paper]{article}

\usepackage[T1]{fontenc}
\usepackage[utf8]{inputenc}
\usepackage[margin=2.7cm]{geometry}
\usepackage{amsmath,amssymb,amsthm}
\usepackage{enumitem}
\usepackage{booktabs}
\usepackage{longtable}

\usepackage{array}
\usepackage{xcolor}
\usepackage{url}
\usepackage[colorlinks=true,linkcolor=blue!45!black,citecolor=blue!45!black, bookmarks=true,bookmarksnumbered=true,bookmarksopen=true]{hyperref}

\allowdisplaybreaks

\theoremstyle{definition}
\newtheorem{definition}{Definition}

\newtheorem{assumption}[definition]{Assumption}
\newtheorem{property}[definition]{Property}
\newtheorem{remark}[definition]{Remark}
\newtheorem{example}[definition]{Example}

\theoremstyle{plain}
\newtheorem{introthm}{Theorem}

\newtheorem{introcor}[introthm]{Corollary}
\newtheorem{lemma}[definition]{Lemma}
\newtheorem{proposition}[definition]{Proposition}
\newtheorem{theorem}[definition]{Theorem}
\newtheorem{corollary}[definition]{Corollary}

\newcommand{\XX}{\mathcal{X}}                      
\newcommand{\XXA}{\mathcal{B}_{\XX}}               
\newcommand{\Xp}{\XX_{+}}                          
\newcommand{\YY}{C}                                
\newcommand{\R}{\mathbb{R}}
\newcommand{\N}{\mathbb{N}}
\newcommand{\Nz}{\mathbb{N}_{0}}
\newcommand{\Prob}{\mathcal{M}_{1}}                
\newcommand{\Probac}{\mathcal{M}_{1}^{\ll\pi}}      
\newcommand{\Sign}{\mathcal{M}_{\pm}}              
\newcommand{\Fplus}{\mathcal{F}^{+}}               
\newcommand{\Dens}{\mathcal{D}}                    
\newcommand{\dd}{\mathop{}\!\mathrm{d}}
\newcommand{\tv}[1]{\lVert #1 \rVert}
\newcommand{\Lone}[1]{\lVert #1 \rVert_{L^{1}(\pi)}}
\newcommand{\ind}{\mathbf{1}}
\newcommand{\Ex}{\mathbb{E}}                        
\newcommand{\Prb}{\mathbb{P}}                       
\DeclareRobustCommand{\hyptgt}[2]{\hypertarget{hyp:#1}{\textup{(#2)}}}
\DeclareRobustCommand{\hyp}[2]{\hyperlink{hyp:#1}{\textup{(#2)}}}
\newcommand{\sing}{\operatorname{sing}}             
\newcommand{\aeevery}{\pi\text{-almost every}}
\newcommand{\aewhere}{\pi\text{-almost everywhere}}
\newcommand{\contig}{\mathbin{\vartriangleleft}}   
\newcommand{\acp}[1]{\ll #1}                
\newcommand{\sgp}[1]{\setminus #1}          

\title{\bfseries A Direct Route to Markov Chain Convergence\\[2mm]
\Large via Asymptotic Equivalence with the Target}
\author{Patrick Forr\'e\\[+10pt]
    \small{AI4Science Lab}\\[0pt]
 \small{Korteweg-de Vries Institute for Mathematics}\\[0pt]
    \small{University of Amsterdam}\\[0pt] \small{\texttt{p.d.forre@uva.nl}}}
\date{}

\begin{document}
\maketitle

\begin{abstract}
\noindent
For a Markov kernel $T$ with an invariant probability measure $\pi$, we give a
self-contained proof of the \emph{Markov chain convergence} theorem via a criterion
called \emph{asymptotic equivalence with the target}. It assumes two parts about the
Lebesgue decompositions of $T^{n}_{x}$ and $\pi$ for every starting point $x$:
\emph{asymptotic absolute continuity}: the singular mass
$\sing(T^{n}_{x}\mid\pi)$ tends to $0$; \emph{asymptotic domination of the target}:
the singular mass $\sing(\pi\mid T^{n}_{x})$ tends to $0$, as $n \to \infty$.
Assuming a jointly measurable density for the absolutely continuous part of each
iterate, this criterion is sufficient and necessary for convergence. A density version
of it is verified on general measurable spaces in three cases:
i.) $T$ has a positive transition density w.r.t.\ $\pi$;
ii.) $T$ consists of an absolutely continuous part with positive transition density
together with an atom at the starting point, which covers the Metropolis--Hastings
algorithm;
iii.) the transition density is positive only after a finite number of steps that may
depend on the starting point $x$.
To demonstrate our general criterion, we investigate the Gibbs sampler with random scan
and the parallel tempering algorithm. Furthermore, we show that in all mentioned
settings Birkhoff's ergodic theorem applies, so as to obtain the \emph{strong law of
large numbers}. Throughout this paper, neither irreducibility, nor aperiodicity, nor
recurrence, nor couplings, nor splitting constructions, nor small sets are used. In all
results, the state space is a general measurable space with no structure beyond a
$\sigma$-algebra. That joint measurability is assumed of the kernel, not of the space;
countable generation supplies it. None of the theorems proved here is new; what is
offered is a short route to a single, widely applicable Markov chain convergence
criterion.
\end{abstract}

{
\small
\tableofcontents
}
\newpage

\section{Introduction}\label{sec:intro}

\subsection{Motivation}

Markov chain Monte Carlo is used, in statistics, in machine learning and in
computational physics, overwhelmingly on continuous state spaces. The target is a
probability measure $\pi$ on a space $\XX = \R^{d}$, or on a manifold, specified by a
density $p$ with respect to Lebesgue measure $\lambda$,
\[
  \pi(A) = \int_{A} p \dd\lambda ,
\]
and known only up to its normalising constant: what one can evaluate is a measurable
map $\tilde p : \XX \to (0,\infty)$ with $p = \tilde p/Z$, the number
$Z = \int_{\XX}\tilde p \dd\lambda \in (0,\infty)$ being unavailable. One then
constructs a Markov kernel $T$ leaving $\pi$ invariant and hopes that the law
$T^{n}\circ\mu$ of the chain after $n$ steps, started in an initial distribution
$\mu$, approaches $\pi$. The convergence theory ordinarily invoked to justify this is
the general theory of $\psi$-irreducible aperiodic Harris chains
\cite{Nummelin84,MeynTweedie09,Douc18}, whose development --- irreducibility, cycles
and aperiodicity, small sets, recurrence, the splitting construction --- is long, and
is designed to cover state spaces far more general than the ones on which the method
is actually run.

What has to be true for the chain to converge is best seen through the Lebesgue
decomposition of its law. At each time $n$ and from each starting point $x$ the law
$T^{n}_{x}$ splits, uniquely, into a part possessing a density with respect to the
target and a part carried by a set the target ignores,
\[
  T^{n}_{x} \;=\; T^{n}_{x,\acp{\pi}} \;+\; T^{n}_{x,\sgp{\pi}},
  \qquad T^{n}_{x,\acp{\pi}} \ll \pi,
  \qquad T^{n}_{x,\sgp{\pi}}(\XX\setminus S) = 0 \ \text{ for some } S \text{ with } \pi(S)=0 .
\]
Two things must happen along this splitting, and between them they are enough. The
singular part must become negligible, from every starting point: whatever mass sits on
sets of $\pi$-measure zero must, in the limit, be given up. And the absolutely
continuous part must become rich enough to see the whole target: from $\pi$-almost
every starting point, the part of $\pi$ on which its density still vanishes must, in
the limit, carry no mass, so that no region of positive $\pi$-measure stays invisible.
The two conditions are the two halves of one relation, mutual absolute continuity,
each asked in the limit rather than at any finite time. These two conditions --- \emph{asymptotic absolute
continuity} and \emph{asymptotic domination of the target} --- are the criterion this
note proves, and they are Theorem \ref{thm:crit} below. Neither refers to a second
starting point, to an initial distribution, or to a density that anyone has to
exhibit.

The condition under which both are easiest to verify, and the one from which the
subject historically starts, is that the kernel $T$ itself possess a density
$\tau(y \mid x)$ with respect to $\lambda$, so that
$T(A \mid x) = \int_{A}\tau(y \mid x)\,\lambda(\dd y)$, and that this density be
strictly positive. On a continuous state space that is, in practice, a mild
assumption. Three observations make the point.

\begin{itemize}[leftmargin=1.4em,topsep=3pt]
  \item A Metropolis--Hastings chain whose proposal density $q(y \mid x)$ is
        strictly positive satisfies it, and only ratios $\tilde p(y)/\tilde p(x)$ of
        the unnormalised target enter the algorithm, so that the unknown constant $Z$
        cancels. The Gaussian random walk proposal
        $q(y \mid x) = \varphi_{\sigma}(y-x)$, with $\varphi_{\sigma}$ the centred
        Gaussian density of variance $\sigma^{2}$, is of this kind, as are
        independence samplers with a Gaussian or heavy-tailed proposal.
  \item If the proposal is local or degenerate, positivity can be enforced at
        negligible cost. Convolving with a Gaussian of arbitrarily small variance
        $\sigma^{2}$ turns any proposal kernel $Q_{x}$ into one with density
        $q_{\sigma}(y \mid x) = \int_{\R^{d}}\varphi_{\sigma}(y-z)\,Q(\dd z \mid x)$,
        strictly positive on all of $\R^{d}$ whatever $Q$ was. Mixing in a global
        proposal with small probability $\epsilon \in (0,1)$, that is replacing $q$ by
        $q' := (1-\epsilon)\,q + \epsilon\,\varphi_{\sigma}$, has the same effect.
  \item What positivity excludes --- periodic structure, and regions the chain cannot
        reach in one step --- is a genuine phenomenon on discrete or lattice-like
        state spaces, and is what forces the general theory to introduce cycles and
        small sets. With a diffuse proposal on a continuous space it does not occur.
\end{itemize}

Strict positivity is, however, more than the criterion asks, and in two separate
respects. It makes the singular part vanish \emph{outright} after a single step,
where asymptotic absolute continuity requires only that its mass tends to zero; and it
makes the density of the absolutely continuous part positive \emph{immediately},
where asymptotic domination asks only that the invisible part of the target shrink to
nothing, at a rate that may depend on the starting point and without ever reaching
zero. Both slacks are used. The Metropolis--Hastings kernel carries an atom at
the point it starts from, of mass the rejection probability there; when $\pi$ is
atomless and that probability is positive, no iterate is absolutely continuous, and
$T^{n}_{x,\sgp{\pi}} \ne 0$ for every $n$. What is true is that the mass of the atom decays
geometrically in $n$, which is asymptotic absolute continuity and nothing stronger. The Gibbs sampler with random
scan, when the coordinate measures are atomless, is singular with respect to $\pi$ at
every step, and no iterate of it is absolutely continuous either; there an absolutely
continuous minorant appears only after a full sweep of the coordinates has become
possible, that is after $d$ steps, and only then is the target dominated. Neither algorithm has a
strictly positive transition density, and the criterion covers both.

Two caveats should be stated at once. The criterion yields \emph{qualitative}
convergence in total variation and no rate, and no rate valid uniformly in the
starting point is available even under strict positivity: Example \ref{ex:norate}
exhibits a chain with an everywhere strictly positive transition density for which
$\sup_{x \in \XX} \tv{T^{n}_{x} - \pi} = 1$ for \emph{every} $n \in \N$,
although $\tv{T^{n}_{x} - \pi} \to 0$ for each fixed $x$.
Moreover the convolution device above, while it restores positivity, does not restore
a usable rate: on $\R^{d}$ the uniform lower bound $\inf_{x,y} q_{\sigma}(y \mid x)$
it supplies is $0$, and over a region of diameter $R$ it is only of the order of
$\varphi_{\sigma}(R)$. Rates belong to the theory of geometric ergodicity,
for which we refer to \cite{Gallegos24} and \cite{MeynTweedie09}.

The purpose of this note is to prove the two-part criterion directly, and to verify it
in the settings in which Markov chain Monte Carlo is actually run --- strictly
positive transition densities among them, but also the two algorithms just mentioned,
for which no such density exists.
The argument rests on a handful of short lemmas
and uses no irreducibility, no aperiodicity, no recurrence, no small sets, no couplings
and no splitting construction. It is carried out on an arbitrary measurable space; no
topological hypothesis of any kind is needed (Remark \ref{rem:nostructure}).

\subsection{The criterion, and the results in their applied form}\label{ssec:applied}

The results of these notes are all instances of a single criterion, and the criterion
is short enough to be stated first. It is a statement about the Lebesgue decomposition
of the law of the chain: the singular half must eventually become small, and the
absolutely continuous half must eventually become large enough to see all of $\pi$.

\begin{introthm}[Convergence via asymptotic equivalence with the target]\label{thm:crit}
Let $(\XX,\XXA)$ be a measurable space and  $T$ be a Markov kernel with invariant probability measure $\pi$ on
$(\XX,\XXA)$: $T \circ \pi = \pi$.
For $x \in \XX$ and $n \in \N$ let
\[
  T^{n}_{x}
  = T^{n}_{x,\acp{\pi}} + T^{n}_{x,\sgp{\pi}}
\]
be the Lebesgue decomposition of $T^{n}_{x}$ with respect to $\pi$, and assume that
for each $n$ the absolutely continuous part admits a density
$t_{n}(y \mid x)$ with respect to $\pi$ that is jointly measurable in $(x,y)$ --- as
it does, for instance, whenever $\XXA$ is countably generated.
Assume further:
\begin{enumerate}[label=\textup{(\roman*)},nosep,topsep=3pt,leftmargin=2.4em,itemsep=3pt]
  \item \emph{(asymptotic absolute continuity)} for every $x \in \XX$ the mass of the
        part of $T^{n}_{x}$ that $\pi$ does not see vanishes:
        \[
          T^{n}_{x,\sgp{\pi}}(\XX) \;=\; \sing\bigl( T^{n}_{x} \bigm| \pi \bigr)
          \;\longrightarrow\; 0
          \qquad\text{as } n \to \infty ;
        \]
  \item \emph{(asymptotic domination of the target)} for $\pi$-almost every
        $x \in \XX$ the mass of the part of $\pi$ that $T^{n}_{x}$ does not see
        vanishes:
        \[
          \sing\bigl( \pi \bigm| T^{n}_{x} \bigr) \;\longrightarrow\; 0
          \qquad\text{as } n \to \infty ,
        \]
        where $\sing(\pi\mid\alpha)$ denotes the mass of the part of $\pi$ that is
        singular with respect to $\alpha$.
\end{enumerate}
Then $\pi$ is the unique invariant probability measure of $T$, and
\[
  \sup_{A \in \XXA}\bigl| (T^{n}\circ\mu)(A) - \pi(A) \bigr| \;\longrightarrow\; 0
  \qquad\text{as } n \to \infty,
\]
for every probability measure $\mu$ on $(\XX,\XXA)$.

Conversely, each of the two hypotheses is implied by that conclusion, so the two
together characterise it.
\end{introthm}

\noindent
The two hypotheses have names, and the names are what the rest of these notes are
about. Hypothesis (i) is \emph{asymptotic absolute continuity}: the part of the law
that is singular with respect to $\pi$ carries, in the limit, no mass. It does not say
that $T^{n}_{x}$ becomes absolutely continuous --- $T^{n}_{x,\sgp{\pi}}$ may well be
nonzero for every $n$ --- only that what remains of it is asymptotically negligible.
Hypothesis (ii) is \emph{asymptotic domination of the target}: from $\pi$-almost every
starting point, the part of $\pi$ that the chain has not yet learned to see becomes
negligible. Together they say that $T^{n}_{x}$ and $\pi$ become mutually absolutely
continuous as $n$ grows --- \emph{asymptotic equivalence with the target}, the two
halves of the relation $T^{n}_{x} \sim \pi$ of \ref{conv:order}, each asked only in
the limit. Neither is required to hold exactly at any finite time; in particular
neither $T^{n}_{x} \ll \pi$ nor $\pi \ll T^{n}_{x}$ need ever hold. Neither hypothesis mentions a second starting point, an initial
distribution, a coupling, a small set, an irreducibility measure, a period, or a
topology on $\XX$; both are conditions on the chain started at one point, and both are
conditions on one and the same decomposition. Neither may be dropped
(Examples \ref{ex:RnotU} and \ref{ex:UnotR}).

\medskip

\noindent
\emph{What the jointly measurable density is doing, and why it is not a hypothesis on
the space.} It is used at exactly one point, and for one purpose. On a bare measurable
space a family of Radon--Nikodym derivatives indexed by the starting point need not
admit a version measurable in $(x,y)$ jointly, and asymptotic domination --- which
asserts the existence of no density at all --- gives no way to select one; the selection
is what is being assumed. In these notes that assumption is isolated as the hypothesis
\hyp{J}{J} of Section \ref{sec:core}, and it is a hypothesis on the \emph{kernel}, of
exactly the kind already made when a transition density is assumed measurable in both
variables: \hyp{J}{J} stands to the Lebesgue density of a general iterate as \hyp{D}{D}
does to the transition density itself. Countable generation of $\XXA$ is one sufficient
condition for it, and the one that is checked in practice
(Proposition \ref{prop:jointdensity}); it holds on $\R^{d}$, on a manifold, on any
standard Borel space and on countable products of these. It is not, however, the
hypothesis: \emph{no convergence theorem in these notes assumes anything of the state
space} (Remark \ref{rem:nostructure}).

Everything after the introduction in fact takes a second road, on which even \hyp{J}{J}
is not needed: it \emph{assumes the conclusion} of the martingale construction one step
further on, in the form of the hypothesis \hyp{P}{P} of Section \ref{sec:core}, which
asks for a jointly measurable minorant density directly.
Theorem \ref{thm:criterion} is the criterion in that form. Its hypotheses are implied
by those of Theorem \ref{thm:crit} under \hyp{J}{J}
(Proposition \ref{prop:LimpliesP}) and always imply them
(Lemma \ref{lem:PimpliesL}), so the two are equivalent as soon as \hyp{J}{J} holds; and
it is the version from which
Theorems \ref{thm:posdens}--\ref{thm:gibbsintro} below are deduced.

\medskip

\noindent
The four statements that follow are the criterion in the settings where it is applied.
They are self-contained, and are the form in which the results are likely to be used;
they are proved below as Corollaries \ref{cor:reference}, \ref{cor:mh},
\ref{cor:referencegen} and \ref{cor:gibbs} respectively; Theorem \ref{thm:crit} itself
is proved as Theorem \ref{thm:asympequiv}.

\begin{introthm}[Markov chains with a strictly positive transition density]\label{thm:posdens}
Let $\XX$ be a set, $\XXA$ a $\sigma$-algebra on $\XX$ and $\lambda$ a
$\sigma$-finite measure on $(\XX,\XXA)$. Let $p : \XX \to (0,\infty)$ be measurable
with $\int_{\XX} p \dd\lambda = 1$ and let $\pi(A) := \int_{A} p \dd\lambda$ be the
target. Let
\[
  \tau : \XX\times\XX \to [0,\infty), \qquad (x,y) \longmapsto \tau(y \mid x),
\]
be $\XXA\otimes\XXA$-measurable with $\int_{\XX}\tau(y \mid x)\,\lambda(\dd y) = 1$
for every $x \in \XX$, and let $T$ be the Markov kernel
$T(A \mid x) := \int_{A}\tau(y \mid x)\,\lambda(\dd y)$. Assume
\begin{enumerate}[label=(\roman*),nosep,topsep=3pt]
  \item \emph{(invariance)} $\displaystyle\int_{\XX}\tau(y \mid x)\,p(x)\,\lambda(\dd x) = p(y)$ for $\lambda$-almost every $y \in \XX$;
  \item \emph{(positivity)} $\tau(y \mid x) > 0$ for $(\lambda\otimes\lambda)$-almost every $(x,y) \in \XX\times\XX$.
\end{enumerate}
Then $\pi$ is the unique invariant probability measure of $T$, and for every initial
distribution $\mu$ on $(\XX,\XXA)$,
\[
  \lim_{n\to\infty}\ \sup_{A \in \XXA}\bigl| (T^{n}\circ\mu)(A) - \pi(A) \bigr| = 0 .
\]
\end{introthm}

\begin{introthm}[The Metropolis--Hastings algorithm]\label{thm:mhintro}
Let $\XX,\XXA,\lambda,p,\pi$ be as in Theorem \ref{thm:posdens}. Let the proposal be the
Markov kernel $Q(A \mid x) = \int_{A} q(y \mid x)\,\lambda(\dd y)$, given by an
$\XXA\otimes\XXA$-measurable map $q : \XX\times\XX \to [0,\infty)$ with
$\int_{\XX} q(y \mid x)\,\lambda(\dd y) = 1$ for every $x \in \XX$, and let $T$ be the
Metropolis--Hastings kernel
\[
  T(A \mid x) := \int_{A} a(y \mid x)\,q(y \mid x)\,\lambda(\dd y) + r(x)\,\ind_{A}(x),
  \qquad
  a(y \mid x) := \min\Bigl\{ 1,\ \frac{p(y)\,q(x \mid y)}{p(x)\,q(y \mid x)} \Bigr\},
\]
with $r(x)$ the resulting total rejection probability. If
\[
  q(y \mid x) > 0 \qquad\text{for every } (x,y) \in \XX\times\XX ,
\]
then $\pi$ is the unique invariant probability measure of $T$ and
$\sup_{A \in \XXA}| (T^{n}\circ\mu)(A) - \pi(A) | \to 0$ for every initial
distribution $\mu$ on $(\XX,\XXA)$. No invariance hypothesis is needed: for the
Metropolis--Hastings kernel it is automatic.
\end{introthm}

\noindent
Strict positivity of $q$ everywhere is the simplest hypothesis that can be checked in
practice, and it is what the devices described above deliver; Corollary
\ref{cor:mh} proves Theorem \ref{thm:mhintro} under a weaker pair of conditions, and
Remark \ref{rem:mhsharp} shows that the weaker pair cannot be weakened further.
Note that the kernel $T$ of Theorem \ref{thm:mhintro} is in general \emph{not} absolutely
continuous with respect to $\pi$: it retains an atom at the starting point whenever
$r(x) > 0$, which is the generic case (Remark \ref{rem:comparison}(iii)). So
Theorem \ref{thm:mhintro} is not a special case of Theorem \ref{thm:posdens};
Section \ref{sec:mh} is devoted to this.

The third result weakens the positivity requirement in a different direction. Write
$\tau_{1} := \tau$ and, recursively,
$\tau_{n+1}(y \mid x) := \int_{\XX}\tau_{n}(y \mid z)\,\tau(z \mid x)\,\lambda(\dd z)$
for the $n$-step transition densities.

\begin{introthm}[Markov chains with a strictly positive transition density after finitely many steps]\label{thm:eventual}
Let $\XX,\XXA,\lambda$, $p,\pi,\tau,T$ be as in Theorem \ref{thm:posdens} and assume
invariance, that is, hypothesis (i) of Theorem \ref{thm:posdens}. Assume further that
\begin{enumerate}[label=(\roman*),nosep,topsep=3pt,start=2]
  \item for every $x \in \XX$ there is a number $n=n(x) \in \N$ such that
        $\tau_{n}(y \mid x) > 0$ for $\lambda$-almost every $y \in \XX$.
\end{enumerate}
Then $\pi$ is the unique invariant probability measure of $T$, and
$\sup_{A \in \XXA}|(T^{n}\circ\mu)(A) - \pi(A)| \to 0$ for every initial distribution
$\mu$ on $(\XX,\XXA)$.
\end{introthm}

\noindent
The number $n$ may depend on the starting point $x$, and no aperiodicity hypothesis
is imposed: it turns out to be a consequence
(Corollary \ref{cor:selfimprovement}). On a \emph{finite} state space, with $\lambda$ the
counting measure, hypothesis (ii) of Theorem \ref{thm:eventual} is \emph{equivalent} to
irreducibility together with aperiodicity, so that Theorem \ref{thm:eventual} then reduces
to the fundamental theorem of Markov chains convergence (Remark \ref{rem:finite}).
Theorems \ref{thm:mhintro} and \ref{thm:eventual} each weaken the hypotheses of
Theorem \ref{thm:posdens}, in two different directions, but neither implies the other:
the kernel of Theorem \ref{thm:mhintro} retains an
atom at the starting point, so that for atomless $\pi$ no iterate of it is absolutely
continuous and Theorem \ref{thm:eventual} does not apply to it; conversely a kernel whose
first step has no density at all, but whose second step has a strictly positive one,
is covered by Theorem \ref{thm:eventual} and not by Theorem \ref{thm:mhintro}.
Remark \ref{rem:comparison} states this precisely.

The fourth result is of a different character. It is not a further weakening of the
hypotheses on the density, but an instance of the general convergence criterion itself
(Theorem \ref{thm:criterion}) applied to a kernel for which no iterate is absolutely
continuous with respect to $\pi$, so that none of Theorems \ref{thm:posdens},
\ref{thm:mhintro} and \ref{thm:eventual} applies. The Gibbs sampler is the example carried out in
full; parallel tempering is carried out in Corollary \ref{cor:tempering}.

\begin{introthm}[The Gibbs sampler with random scan]\label{thm:gibbsintro}
Let $d \in \N$ with $d \ge 2$, let $\mathcal{Y}_{1},\dots,\mathcal{Y}_{d}$ be
measurable spaces carrying $\sigma$-finite measures
$\lambda_{1},\dots,\lambda_{d}$, and let $\XX$, $\XXA$ and $\lambda$ be their
products. Let $p : \XX \to (0,\infty)$ be measurable with
$\int_{\XX} p \dd\lambda = 1$ and let $\pi(A) := \int_{A} p \dd\lambda$ be the
target. Assume that all full conditional distributions exist, that is, that
\[
  \int_{\mathcal{Y}_{i}} p(x^{i:w})\,\lambda_{i}(\dd w) < \infty
  \qquad\text{for every } i \in \{1,\dots,d\} \text{ and every } x \in \XX ,
\]
where $x^{i:w}$ denotes $x$ with its $i$-th coordinate replaced by $w$. Let $P_{i}$
be the Markov kernel that replaces the $i$-th coordinate by a draw from its full
conditional and leaves the others unchanged, and let
$T := \frac{1}{d}\sum_{i=1}^{d}P_{i}$ be the random scan Gibbs kernel. Then $\pi$ is
the unique invariant probability measure of $T$, and
$\sup_{A \in \XXA}|(T^{n}\circ\mu)(A) - \pi(A)| \to 0$ for every initial distribution
$\mu$ on $(\XX,\XXA)$.
\end{introthm}

\noindent
No invariance hypothesis is needed here either: each $P_{i}$ is reversible with
respect to $\pi$ by construction. The words ``for every $x$'' in the displayed
hypothesis are load-bearing and cannot be relaxed to ``for $\lambda$-almost every
$x$'': see Remark \ref{rem:gibbssharp}. If in addition every $\lambda_{i}$ is atomless, the
kernel $T$ is \emph{singular} with respect to $\pi$, and no iterate of it is
absolutely continuous, so that Theorem \ref{thm:gibbsintro} is a special case of none of
Theorems \ref{thm:posdens}, \ref{thm:mhintro} and \ref{thm:eventual}
(Lemma \ref{lem:gibbssingular}). By contrast the systematic scan sampler
$P_{d}\cdots P_{1}$ \emph{does} have a strictly positive transition density and is
covered by Theorem \ref{thm:posdens} (Remark \ref{rem:gibbssystematic}); the two scans, so
close in practice, sit on opposite sides of the absolute continuity divide.

Finally, all five of Theorems \ref{thm:crit}--\ref{thm:gibbsintro} have a common consequence,
which is what is actually used when a chain is run.

\begin{introcor}[Convergence of ergodic averages]\label{cor:llnintro}
In each of Theorems \ref{thm:crit}, \ref{thm:posdens}, \ref{thm:mhintro}, \ref{thm:eventual} and \ref{thm:gibbsintro}, let
$(X_{k})_{k \in \Nz}$ be the Markov chain with kernel $T$ started in an arbitrary
initial distribution $\mu$, and let $\Prb_{\mu}$ denote its law on the space of
trajectories. Then for every measurable map $f : \XX \to \R$ with
$\int_{\XX} |f| \dd\pi < \infty$,
\[
  \frac{1}{n}\sum_{k=0}^{n-1} f(X_{k}) \longrightarrow \int_{\XX} f \dd\pi
  \qquad (n \to \infty), \quad \Prb_{\mu}\text{-almost surely.}
\]
If $f$ is bounded, the convergence holds in addition in $L^{r}(\Prb_{\mu})$ for every
number $r \in [1,\infty)$. If $\mu \le M\pi$ for some number $M \in [1,\infty)$ ---
for instance if $\mu = \pi$ --- then, for every number $r \in [1,\infty)$, the
convergence holds in $L^{r}(\Prb_{\mu})$ for every $f$ with
$\int_{\XX} |f|^{r} \dd\pi < \infty$.

In Theorems \ref{thm:posdens}--\ref{thm:gibbsintro}, where $\pi$ is presented by a
density $p$ with respect to a reference measure $\lambda$, the integrals against $\pi$
read $\int_{\XX} f(x)\,p(x)\dd\lambda(x)$ and the condition $\mu \le M\pi$ says that
$\mu$ has a $\lambda$-density bounded by $Mp$. Theorem \ref{thm:crit} names no
reference measure, and the statement is the one displayed above.
\end{introcor}

\noindent
This is the statement that justifies estimating $\int f \dd\pi$ by an average
along a single trajectory of the Markov chain, and it is what a run of the algorithm actually produces.
It is not a consequence of the convergence of the laws $T^{n}\circ\mu$ alone. Its proof
combines it with Birkhoff's pointwise ergodic
theorem, and is given in Section \ref{sec:ergodic}, where Corollary \ref{cor:llnintro}
appears as Corollary \ref{cor:lln} and Birkhoff's theorem is quoted as
Theorem \ref{thm:birkhoff}. Apart from the standard measure theory listed in
Remark \ref{rem:nostructure}, that section is the only part of the note which relies
on results proved elsewhere.

\subsection{The idea of the proof and the plan}

\paragraph{The four-line proof, and why it is not available.}
It helps to begin with the case in which everything is easy. Suppose the transition density were
bounded below by a constant, $t(y \mid x) \ge \varepsilon_{0} > 0$ for all $x$ and
$y$. Then, whatever the initial law $\mu$, the law after one step satisfies
$T\circ\mu \ge \varepsilon_{0}\pi$: a fixed proportion $\varepsilon_{0}$ of the mass
has been redistributed according to $\pi$ and has forgotten where it came from. Two
chains started differently share that proportion, so their laws can disagree by at
most $1-\varepsilon_{0}$; iterating this argument leads to $\tv{T^{n}\circ\mu - \pi} \le (1-\varepsilon_{0})^{n}$. This is Doeblin's argument
(Remark \ref{rem:uniform}). It is four lines long and it even produces a rate.

On a \emph{finite} state space this is always the situation: a strictly positive
matrix has finitely many entries and therefore a smallest one, so the fundamental
theorem of Markov chains never leaves the easy case. On a general space the
hypothesis ``$t(y\mid x) > 0$ at every point'' gives no such $\varepsilon_{0}$ at
all --- the infimum of a strictly positive function on an infinite set is usually
zero --- and no substitute for it can be manufactured: Example \ref{ex:norate} is a
chain with an everywhere strictly positive density for which
$\sup_{x}\tv{T^{n}_{x} - \pi} = 1$ for every $n$. Everything below is a way
of recovering the \emph{conclusion} of Doeblin's argument, though necessarily not its
rate, from a hypothesis that supplies no $\varepsilon_{0}$.

\paragraph{Three facts.}
Fix two initial laws $\mu$ and $\nu$ and write
$\Delta_{n} := \tv{T^{n}\circ\mu - T^{n}\circ\nu}$ for the discrepancy after $n$
steps. The proof rests on three statements, of which only the second uses eventual
domination of the target.

\begin{enumerate}[label=\textup{(\arabic*)},leftmargin=2.6em,itemsep=4pt,topsep=4pt]

\item \emph{Nothing ever gets worse.} A Markov kernel averages, and averaging cannot
drive two laws apart: $\Delta_{n}$ is non-increasing
(Lemma \ref{lem:contraction}). So the limit $d := \lim_{n}\Delta_{n}$ exists, and the
entire problem is to rule out $d > 0$ --- to exclude that the two laws stop
approaching each other while still a positive distance apart.

\item \emph{On laws that are not too concentrated, one step gains a definite amount.}
Call a probability measure \emph{$M$-flat}\footnote{``$M$-flat'' will in the formal statements be called ``dominated by $M\pi$''.} if $\mu \le M\pi$. An $M$-flat law cannot
hide in a small set: it must put at least half of its mass on a set of $\pi$-measure
at least $1/(2M)$. Domination of the target, in turn, says that from a set of positive
$\pi$-measure a definite amount of mass reaches a definite set of endpoints
(Lemma \ref{lem:positivity}). Combining the two, the images of any two $M$-flat laws
after $N$ steps have a \emph{common minorant} of mass $\gamma > 0$ --- a piece of
mass that both of them carry --- and two measures with a common minorant of mass
$\gamma$ are at distance at most $1-\gamma$ (Lemma \ref{lem:minorant}). This is
Proposition \ref{prop:minorisation}, and it is the exact analogue of Doeblin's
$\varepsilon_{0}$: a lower bound valid not for all laws, but for all laws in a fixed
class. The price is that $\gamma$ shrinks as $M$ grows --- the larger $M$ is, the more
concentrated the laws in the class are allowed to be, and the less can be guaranteed
about them --- so the bound is worthless unless $M$ can be held fixed as the chain
runs.

\item \emph{The class is never left, so that $M$ can indeed be held fixed.}
$\mu \le M\pi$ implies $T\circ\mu \le M\pi$, because $\pi$ is invariant
(Lemma \ref{lem:domination}). The constant $M$, and with it $\gamma$, therefore does
not deteriorate with the number of steps. This is the only place where invariance of
$\pi$ is used in the core argument, and it is what allows the argument to close.

\end{enumerate}

\paragraph{Why the three facts do not simply multiply.}
There is one genuine obstacle left, and it is worth naming, because it is the only
step of the proof that is not a routine transcription of the finite case. Fact (2)
is an \emph{absolute} bound: it says that the distance is at most $1-\gamma$, not
that the distance gets multiplied by $1-\gamma$. Applied directly to the pair
$(T^{n}\circ\mu, T^{n}\circ\nu)$ it yields $\Delta_{n} \le 1-\gamma$ for large $n$
and then says nothing further; once the two laws are closer than $1-\gamma$, the
statement is vacuous.

The device that converts an absolute bound into a proportional one is a
renormalisation. The discrepancy after $n$ steps is a signed measure $h_{n}$ of total
size $\Delta_{n}$; divide it by its own size. Its positive and negative parts then
become \emph{probability} measures, to which fact (2) applies afresh, and scaling
back turns the absolute bound $1-\gamma$ into the proportional bound
$\Delta_{n+N} \le (1-\gamma)\Delta_{n}$.

Dividing by $\Delta_{n}$, however, multiplies the flatness constant by
$1/\Delta_{n}$: the undivided parts are $M$-flat, the divided ones only
$(M/\Delta_{n})$-flat. If $\Delta_{n}$ were allowed to tend to $0$, the constant would
blow up, $\gamma$ would degenerate, and the argument would collapse. This is exactly
where the assumed $d > 0$ earns its keep: it converts the useless bound
$\Delta_{n} > 0$ into the uniform bound $\Delta_{n} \ge d$, so that one single
constant $M/d$, and hence one single $\gamma$, serves all $n$ at once. The distance
then falls by the factor $1-\gamma$ over every block of $N$ steps and therefore tends
to $0$ --- contradicting $d > 0$. Hence $d = 0$. This is
Proposition \ref{thm:dominated}, and it explains why the method yields no rate: the
constant $\gamma$ that drives the decay is chosen only after the unknown limit $d$
has been named.

\paragraph{From flat laws to arbitrary ones.}
What has been proved so far concerns \emph{flat} initial laws, and the law one
actually starts from need not be flat --- $\delta_{x}$ is not, when $\pi$ is
atomless. Here asymptotic absolute continuity enters, and in a completely different
way from the domination hypothesis: it says that after enough steps all
but an arbitrarily small proportion of the law has a density with respect to $\pi$
(under strict positivity this already happens after one step,
Lemma \ref{lem:onestep}), and truncating that density at a high
level turns it into a flat one, at the cost of an arbitrarily small error in total
variation (Lemma \ref{lem:truncation}). Since errors do not grow along the dynamics,
by fact (1), that small error remains small forever, and the convergence for flat
laws transfers to all of them.

These two roles --- supplying the common minorant, and supplying the passage from an
arbitrary initial law to a flat one --- are the only things the argument ever asks
for, and they are exactly the two halves of the criterion. The second role is played
by \hyp{S}{S}, asymptotic absolute continuity: the part of $T^{n}_{x}$ that is
singular with respect to $\pi$ has vanishing mass, for every starting point $x$. The
first is played by \hyp{P}{P}, a strictly positive minorant density after finitely
many steps, which is asymptotic domination of the target together with the requirement
that a witnessing density be exhibited, jointly measurably in $(x,y)$; under \hyp{J}{J}
that requirement is free and the two are the same hypothesis
(Proposition \ref{prop:LimpliesP}). It is \hyp{P}{P} that is assumed below, because it
asks for no measurable selection at all --- in three of the five settings the minorant
it is given is one that no selection would have produced
(Remark \ref{rem:PvsJ}).

\paragraph{In one sentence.}

\begin{quote}
Doeblin's argument for the convergence of the Markov chain needs a lower bound on the
transition density that is valid uniformly on the whole space; asymptotic domination
of the target supplies instead a bound that is uniform only over a restricted class of
initial laws, that class is preserved by the dynamics because $\pi$ is invariant, and
asymptotic absolute continuity is what carries an arbitrary initial law into it.
\end{quote}

The argument is organised around a single convergence theorem, proved once in
Section \ref{sec:core} from \hyp{P}{P} and \hyp{S}{S}
(Theorem \ref{thm:criterion}). That is Theorem \ref{thm:crit} with its second
hypothesis, asymptotic domination of the target, replaced by the jointly measurable minorant density
that a measurable Lebesgue decomposition would otherwise have had to produce; the two
are equivalent under \hyp{J}{J} (Proposition \ref{prop:LimpliesP}), and the form
assumed here is the one that needs no measurable selection at all.
Everything in that section uses only that
$\pi$ is an invariant probability measure of a Markov kernel $T$. Sections
\ref{sec:first}--\ref{sec:general} then verify the two hypotheses under three
successively more general conditions on the transition density, and
Section \ref{sec:applications} verifies them for two algorithms to which none of
those three conditions applies. Nothing after Section \ref{sec:core} does anything
else.

Section \ref{sec:core} is arranged so that the convergence theorem comes as early as
possible. It is first proved, in Subsection \ref{ssec:criterion}, from two properties
of the action of $T$ on measures: \emph{uniform overlap} \eqref{prop:U}, which is
what fact (2) provides, and \emph{regularisation} \eqref{prop:R}, which is what the
last paragraph provides. Only two lemmas are needed for that, and no density occurs
in it. The remaining subsections then trade those two properties for \hyp{P}{P} and
\hyp{S}{S}: the first implies \eqref{prop:U} (Lemma \ref{lem:PimpliesU}) and the
second is \emph{equivalent} to \eqref{prop:R} (Lemma \ref{lem:SiffR}), while both are
far easier to check, being statements about the chain started at a point rather than
about all initial laws. Assumption \hyp{S}{S} is asymptotic absolute continuity, in
the form in which the criterion of Theorem \ref{thm:crit} states it.

Subsection \ref{ssec:singmass} states the Lebesgue decomposition
(Theorem \ref{thm:lebesgue}), defines the singular mass, and proves the handful of
elementary facts about both that are used later; it is what \hyp{S}{S} is stated in terms of, and
where the vocabulary of Theorem \ref{thm:crit} is fixed. Subsection \ref{ssec:L}
replaces \hyp{P}{P} by the hypothesis \hyp{L}{L} that
$\sing(\pi\mid T^{n}_{x}) \to 0$ for $\aeevery$ starting point $x$ --- a statement with
no density in it at all --- and shows the two to be equivalent as soon as the
absolutely continuous parts of the iterates have jointly measurable densities, which is
the hypothesis \hyp{J}{J} stated there (Lemma \ref{lem:PimpliesL} and
Proposition \ref{prop:LimpliesP}); countable generation of $\XXA$, the hypothesis
\hyp{C}{C}, is a sufficient condition for \hyp{J}{J} and enters only through it
(Proposition \ref{prop:jointdensity}). Subsection \ref{ssec:main}, the last of the
section, assembles the pieces into Theorem \ref{thm:asympequiv}, which is
Theorem \ref{thm:crit} above. Only its sufficiency half restates what
Theorem \ref{thm:criterion} already gives; the converse is new, and is one line from
Corollary \ref{cor:singtv}.

\medskip
\noindent
{\small
\begin{tabular}{@{}>{\raggedright\arraybackslash}p{0.44\textwidth}>{\raggedright\arraybackslash}p{0.50\textwidth}@{}}
\toprule
Statement & Depends on \\
\midrule
Lemma \ref{lem:contraction} (the dynamics is a contraction) & $T$ is a Markov kernel \\
Lemma \ref{lem:domination} (domination by $M\pi$ is preserved) & invariance of $\pi$ \\
Proposition \ref{thm:dominated} (convergence for dominated pairs) & Lemmas \ref{lem:contraction}, \ref{lem:domination}, \eqref{prop:U} \\
Lemma \ref{lem:pointwise} (points suffice for laws) & Definition \ref{def:tv}, Fatou \\
Theorem \ref{thm:abstract} (convergence in general) & Lemma \ref{lem:contraction}, Proposition \ref{thm:dominated}, \eqref{prop:R}, Lemma \ref{lem:pointwise} \\
\midrule
Lemma \ref{lem:minorant} (a common minorant bounds the distance) & nothing \\
Lemma \ref{lem:positivity} (quantitative positivity) & nothing (pure measure theory) \\
Proposition \ref{prop:minorisation} (uniform overlap after $N$ steps) & Lemmas \ref{lem:minorant}, \ref{lem:positivity} \\
\midrule
Theorem \ref{thm:lebesgue} (Lebesgue decomposition) & Radon--Nikodym \\
Proposition \ref{prop:singforms} (six descriptions of the singular mass) & Theorem \ref{thm:lebesgue}, Definition \ref{def:tv} \\
Corollary \ref{cor:singtv} (\textbf{necessity}: both singular masses $\le$ the distance) & Definition \ref{def:tv}, \eqref{eq:singsg} \\
Lemma \ref{lem:monotone} (propagation and monotonicity of the singular masses) & Theorem \ref{thm:lebesgue}, Proposition \ref{prop:singforms}, invariance of $\pi$ \\
Definition \ref{def:lebesgue} (the density $g_{\alpha}$, and \eqref{eq:acpositive}) & Radon--Nikodym, Proposition \ref{prop:singforms} \\
\midrule
Lemma \ref{lem:PimpliesU} (\hyp{P}{P} $\Rightarrow$ \eqref{prop:U}) & Proposition \ref{prop:minorisation} \\
Lemma \ref{lem:truncation} (truncation of a density) & Lemma \ref{lem:tvformulas} \\
Lemma \ref{lem:SiffR} (\hyp{S}{S} $\Leftrightarrow$ \eqref{prop:R}) & Lemmas \ref{lem:monotone}(iv), \ref{lem:truncation} \\
\textbf{Theorem \ref{thm:criterion}} (\textbf{the convergence theorem}) & Lemmas \ref{lem:PimpliesU}, \ref{lem:SiffR}, Theorem \ref{thm:abstract} \\
Remark \ref{rem:aposteriori} (\eqref{prop:R} and \hyp{S}{S} for every law) & Theorem \ref{thm:abstract}, Lemma \ref{lem:SiffR} \\
Lemma \ref{lem:uniformminorant} (a uniform minorant gives \hyp{S}{S}) & Lemmas \ref{lem:monotone}(i), \ref{lem:monotone}(iv) \\
\midrule
Lemma \ref{lem:PimpliesL} (\hyp{P}{P} $\Rightarrow$ \hyp{L}{L}) & Tonelli, Borel--Cantelli \\
Proposition \ref{prop:jointdensity} (\hyp{C}{C} $\Rightarrow$ \hyp{J}{J}) & \hyp{C}{C}, L\'evy's upward theorem \\
Proposition \ref{prop:LimpliesP} (\hyp{L}{L} $\Rightarrow$ \hyp{P}{P} under \hyp{J}{J}) & \hyp{J}{J}, Tonelli, Theorem \ref{thm:lebesgue}, \eqref{eq:acpositive}, Lemma \ref{lem:monotone}(v) \\
\textbf{Theorem \ref{thm:asympequiv}} (\textbf{the characterisation}) & \hyp{J}{J}, Proposition \ref{prop:LimpliesP}, Theorem \ref{thm:criterion}, Corollary \ref{cor:singtv} \\
\midrule
Theorem \ref{thm:main} (strictly positive density) & verify \hyp{P}{P}, \hyp{S}{S} with $N=1$ \\
Theorem \ref{thm:mh} (strictly positive density plus an atom) & verify \hyp{P}{P}, \hyp{S}{S} with $N=1$ \\
Theorem \ref{thm:general} (strictly positive density after $n(x)$ steps) & verify \hyp{P}{P}, \hyp{S}{S} with $N=N(\varepsilon)$ \\
Corollary \ref{cor:gibbs} (the Gibbs sampler) & Lemma \ref{lem:uniformminorant}, Theorem \ref{thm:criterion} \\
Corollary \ref{cor:tempering} (parallel tempering) & Lemma \ref{lem:uniformminorant}, Theorem \ref{thm:criterion} \\
\midrule
Corollary \ref{cor:lln} (ergodic averages) & Theorem \ref{thm:criterion}, Remark \ref{rem:aposteriori}, Theorem \ref{thm:birkhoff} \\
\bottomrule
\end{tabular}
}
\medskip

Each of Sections \ref{sec:first}--\ref{sec:general} adds one or two short lemmas of
its own, stated where they are first needed, so that a reader interested only in the
first theorem does not meet them. No topology on the state space appears anywhere after
Section \ref{sec:setting}, and no coupling construction is needed.

\paragraph{Sources, and what is and is not claimed.}
\emph{None of the theorems of Subsection \ref{ssec:applied} is new, and none is
claimed to be.} What is offered is a route to them. This paragraph says which
classical statement each of the four applied theorems is, and states explicitly what the route, rather than the destination, is
supposed to supply. Pointers to the specific places where a result of these notes, or
the machinery it replaces, can be found are given again as they arise.

\smallskip
\noindent\emph{(i) Which theorem is which.}
Theorem \ref{thm:main} is the discrete-time form of what is usually called
\emph{Doob's theorem}, after Doob \cite{Doob48}; see Da Prato and Zabczyk
\cite[Theorem 4.2.1]{DaPratoZabczyk96}, where it is used in exactly this way to
deduce convergence of transition probabilities to the unique invariant measure, and
Kulik and Scheutzow \cite{KulikScheutzow15} for a coupling proof.
Theorem \ref{thm:general} is contained in \cite[Theorem 1]{KulikScheutzow15}: its
hypothesis \hyp{E}{E} gives, through the self-improvement
Corollary \ref{cor:selfimprovement}, a single number $n$ for any prescribed pair
$x,y$ with $T^{n}_{x} \sim \pi \sim T^{n}_{y}$, which is
exactly the hypothesis of that theorem. It is in turn contained in the criterion of
Scheutzow and Schindler \cite{ScheutzowSchindler21}, which is \emph{necessary} as
well as sufficient; Remark \ref{rem:necessary} places \hyp{E}{E} between the two
precisely. Theorem \ref{thm:mh} and Corollary \ref{cor:mh} are the standard
convergence statements for the Metropolis--Hastings algorithm
\cite{Metropolis53,Hastings70}, due in this measure-theoretic generality to Tierney
\cite{Tierney94}; and Corollary \ref{cor:gibbs} is the convergence theorem for the
Gibbs sampler of Roberts and Smith \cite{RobertsSmith94}, Tierney \cite{Tierney94}
and Chan \cite{Chan93}. See the survey of Roberts and Rosenthal
\cite{RobertsRosenthal04}. Theorem \ref{thm:crit}, the criterion of which these four
are instances, is not claimed to be new either. It is stated
in terms of the Lebesgue decomposition of $T^{n}_{x}$ with respect to $\pi$;
Remark \ref{rem:necessary} places the closely related hypothesis \hyp{E}{E} exactly
within the classification of Scheutzow and Schindler \cite{ScheutzowSchindler21},
whose criteria are necessary as well as sufficient but are stated in terms of pairs of
starting points instead. All four are special cases of the general convergence
theory for $\psi$-irreducible aperiodic Harris chains developed by Orey
\cite{Orey71}, Nummelin \cite{Nummelin84}, Meyn and Tweedie \cite{MeynTweedie09} and
Douc, Moulines, Priouret and Soulier \cite{Douc18}, for which see also Kulik
\cite{Kulik18} and Hern\'andez-Lerma and Lasserre \cite{HernandezLermaLasserre03}.

\smallskip
\noindent\emph{(ii) Two neighbouring literatures.}
The $L^{1}$ sector of Theorem \ref{thm:main} is a theorem of the Lasota school: an
integral Markov operator whose kernel is almost everywhere strictly positive and
which possesses a stationary density is \emph{asymptotically stable}. See Lasota and
Mackey \cite{LasotaMackey94}, in the chapters on the asymptotic stability of Markov
operators, and Rudnicki \cite{Rudnicki95}. Restricted to initial laws that
\emph{have} a density, Theorem \ref{thm:main} is that statement; the extension to an
arbitrary initial law is Lemma \ref{lem:onestep}, one line. The relative of
Theorem \ref{thm:mh} is the \emph{partially integral} circle of results: an operator
which merely dominates a nontrivial integral part, and has a nontrivial fixed point,
is asymptotically stable --- proved for continuous-time Markov semigroups on $L^{1}$
by Pich\'or and Rudnicki \cite[Theorem 1]{PichorRudnicki00}. Assumption \hyp{M}{M} is
a discrete-time hypothesis of exactly that shape, the integral part being $k$ and the
remainder the atom. Finally, the renormalisation device of
Proposition \ref{thm:dominated} --- dividing a signed measure by its own total
variation, so that an absolute bound becomes a proportional one --- is the engine of
the \emph{zero-two law} for Markov operators of Derriennic \cite{Derriennic76} and of
the operator-theoretic tradition surveyed by Foguel \cite{Foguel69}, in which the
behaviour of the singular part under a Markov operator (here
Lemma \ref{lem:monotone}(iv)) is a standard tool.

\smallskip
\noindent\emph{(iii) What the route is claimed to offer.} Three things, none of them
a new theorem.
\begin{enumerate}[label=\textup{(\arabic*)},nosep,topsep=3pt,leftmargin=2.4em]
  \item \emph{One criterion, and no structure at all.} Everything is deduced from the
        single pair \hyp{P}{P}, \hyp{S}{S} (Theorem \ref{thm:criterion}) on a bare
        measurable space: no topology, no irreducibility, no aperiodicity, no
        recurrence, no small sets, no splitting and no coupling. This is a genuine and
        not merely stylistic difference from the references above. The sharp criteria
        of \cite{KulikScheutzow15} and \cite{ScheutzowSchindler21} are proved under
        the standing hypothesis that $\XXA$ be countably generated with measurable
        diagonal; the Doob-type proofs in the strong Feller setting, such as the short
        one in Hairer \cite{Hairer08}, use a topology. Moreover those criteria deliver
        convergence from every \emph{point}, and the passage from that to convergence
        from every initial \emph{distribution} runs through the measurability of
        $x \mapsto \tv{T^{n}_{x} - \pi}$, which on a bare measurable space
        can fail; Remark \ref{rem:tvnotmeasurable} is about exactly this point, and
        the argument here never needs that map. What the density-free form of the
        criterion assumes is \hyp{J}{J}, a jointly measurable density for the
        absolutely continuous part of each iterate, which is again a hypothesis on the
        kernel; countable generation enters only as a sufficient condition for it
        (Proposition \ref{prop:jointdensity}), and no convergence theorem here assumes
        it.
  \item \emph{Kernels singular at every step, treated uniformly.} The random scan
        Gibbs sampler and parallel tempering have $\sing(T^{n}_{x}\mid\pi) > 0$ for
        every $n$ and every $x$ (Lemmas \ref{lem:gibbssingular} and
        \ref{lem:temperingsingular}), so the density-based classical statements do not
        reach them directly. They are nevertheless two applications of one lemma
        (Lemma \ref{lem:uniformminorant}) here, and they come with explicit
        quantitative bounds on the singular mass.
  \item \emph{Explicit constants and explicit failure modes.} Every constant is traced
        ($\gamma = \eta\delta/8$ in Proposition \ref{prop:minorisation},
        $c=(1-\omega)(\theta^{-})^{K}$ in Corollary \ref{cor:tempering}), and each
        hypothesis is accompanied by an object showing it cannot be dropped
        (Examples \ref{ex:norate}, \ref{ex:RnotU}, \ref{ex:UnotR} and
        Remarks \ref{rem:mhsharp}, \ref{rem:gibbssharp}, \ref{rem:pcn}).
\end{enumerate}
One further source is close to these notes in spirit: Asmussen and Glynn
\cite{AsmussenGlynn11} give a short proof that an irreducible chain possessing a
transition density and a stationary distribution is automatically positive Harris
recurrent, which is why the density hypothesis used here can replace recurrence
theory. For coupling characterisations of total variation convergence, which is the
other half of the necessary-and-sufficient picture, see Thorisson
\cite{Thorisson00}.

\paragraph{Relation to the finite case.}
These notes may be read as a generalization, to arbitrary measurable spaces, of
the \emph{fundamental theorem of Markov chains} for finite state spaces --- every
irreducible aperiodic chain has a unique stationary distribution, to which it
converges from every initial distribution --- whose various classical proofs are
surveyed by Biswas \cite{Biswas22}. The precise relationship is this. Restricted to a
finite state space, Theorem \ref{thm:general} \emph{is} that theorem, its hypothesis
\hyp{E}{E} being equivalent to irreducibility together with aperiodicity
(Remark \ref{rem:finite}); and the reduction, carried out in \cite{Biswas22} and
elsewhere, of the general finite case to the case of a strictly positive transition
matrix is the finite instance of the self-improvement Corollary
\ref{cor:selfimprovement}. The argument given below is the general-state-space form
of the contraction proof, the one which tracks how far apart two copies of the chain
can be after $n$ steps.

One caveat is worth stating at the outset, because it is exactly the point at which
the general case stops being a routine transcription of the finite one. On a finite
state space a strictly positive transition matrix is automatically \emph{uniformly}
positive, since there are only finitely many entries to take a minimum over; the
finite fundamental theorem therefore always lands in the Doeblin situation of
Remark \ref{rem:uniform}, where the proof is four lines and produces a geometric
rate. On a general space, strict positivity of the density does not give any uniform
lower bound, no rate is available (Remark \ref{rem:notclaimed}), and supplying a
substitute is the actual work done in Sections \ref{sec:core}--\ref{sec:general}.

\section{The general setting}\label{sec:setting}

\subsection{Notations}\label{sec:notation}

Every object below is introduced together with its type: a number with the set of
numbers it belongs to, a map with its domain and codomain. The remaining conventions
are the following.

\begin{enumerate}[label=\textup{(N\arabic*)},ref=\textup{N\arabic*},leftmargin=3.2em,itemsep=2pt]

\item\label{conv:numbers}
$\N := \{1,2,3,\dots\}$, $\Nz := \N \cup \{0\}$, $\R$ is the set of real numbers, and
$[0,\infty] := [0,\infty) \cup \{+\infty\}$ carries its usual order and arithmetic.

\item\label{conv:maps}
A map is introduced as $F : A \to B$ with $A$ its domain and $B$ its codomain. If
$a \in A$ then $F(a) \in B$; a map and its values are never identified, so that in
particular a map into a set of numbers is not itself a number.

\item\label{conv:spaces}
Spaces are written with calligraphic capitals $\XX,\mathcal{Y},\dots$, and the
$\sigma$-algebra of a space is the letter $\mathcal{B}$ with that space as subscript:
that of $\XX$ is $\XXA$, that of a further space $\mathcal{Y}$ would be
$\mathcal{B}_{\mathcal{Y}}$. Only one space occurs below. Its points are denoted
$x,y,z$ and its measurable subsets by plain capitals $A, \YY, E, G, L$, apart from the
two distinguished families $\Xp$ (the support of the target density) and $\XX_{n}$
(the sets of \hyp{E}{E}). The only other measurable space to occur is the path
space $\XX^{\Nz}$ of Section \ref{sec:ergodic}, whose measurable subsets are written
$B$ and $\Gamma$, so that no letter denotes a subset of two different spaces. In
Subsection \ref{ssec:gibbs} the space $\XX$ is a finite product
$\mathcal{Y}_{1}\times\cdots\times\mathcal{Y}_{d}$ of measurable spaces; the factors
are the only spaces other than $\XX$ and $\XX^{\Nz}$ to occur.

\item\label{conv:measures}
$\Prob$ and $\Sign$ denote the sets of probability measures, respectively of finite
signed measures, on $(\XX,\XXA)$; each of their elements is itself a map
$\XXA \to [0,1]$, respectively $\XXA \to \R$.

\item\label{conv:order}
For measures $\alpha,\beta$ we write $\alpha \le \beta$ if $\alpha(A) \le \beta(A)$
for every $A \in \XXA$, and $M\pi$ for the measure $A \mapsto M\pi(A)$, so that
``$\mu \le M\pi$'' is a statement about maps and not about numbers. Recall that
$\alpha \le \beta$ implies $\int g \dd\alpha \le \int g \dd\beta$ for every
measurable $g : \XX \to [0,\infty]$. We write $\alpha \ll \beta$ for absolute
continuity, $\beta \gg \alpha$ for the same relation read the other way round --- so
that $\beta \gg \pi$ says that $\beta$ gives positive mass to every set that $\pi$
does --- and $\alpha \sim \beta$ for mutual absolute continuity; for
$\nu \in \Prob$ the relation $\nu \sim \pi$ says exactly that $\nu$ has a density
with respect to $\pi$ which is strictly positive $\aewhere$.

\item\label{conv:functions}
$\Fplus$ is the set of measurable maps $\XX \to [0,\infty]$ and, once $\pi$ is fixed,
$\Dens := \{ f \in \Fplus : \int_{\XX} f \dd\pi = 1 \}$ is the set of probability
densities with respect to $\pi$. The norm of $L^{1}(\pi)$ is written $\Lone{\cdot}$.

\item\label{conv:kernelnotation}
The conditioning variable stands to the right of a bar: $T(A \mid x)$ is the
probability of $A$ when the chain starts at $x$, and $t(y \mid x)$ the corresponding
density evaluated at $y$. Consistently, the action of a kernel on a measure is a
composition from the left, $\mu \mapsto T\circ\mu$, so that $T^{n}\circ\mu$ is the
law after $n$ steps and $T \circ (T^{n}\circ\mu) = T^{n+1}\circ\mu$.

\item\label{conv:Tx}
For $x \in \XX$ and $n \in \Nz$ we write
\[
  T^{n}_{x} \;:=\; T^{n}(\,\cdot\mid x) \;=\; T^{n}\circ\delta_{x} \;\in\; \Prob
\]
for the law of the chain at time $n$ started at the point $x$, and correspondingly
$T_{x} := T^{1}_{x}$. For the two parts of a Lebesgue decomposition the reference
measure is displayed in the subscript: for finite measures $\alpha$ and $\beta$ we
write
\[
  \alpha \;=\; \alpha_{\acp{\beta}} \;+\; \alpha_{\sgp{\beta}},
  \qquad
  \alpha_{\acp{\beta}} \ll \beta,
  \qquad
  \alpha_{\sgp{\beta}} \perp \beta ,
\]
so that $\alpha_{\acp{\beta}}$ is the part of $\alpha$ that $\beta$ sees and
$\alpha_{\sgp{\beta}}$ the part it does not (Definition \ref{def:lebesgue}); in
particular $T^{n}_{x,\acp{\pi}}$ and $T^{n}_{x,\sgp{\pi}}$ are the two parts of
$T^{n}_{x}$ with respect to $\pi$, while $\pi_{\sgp{T^{n}_{x}}}$ is the part of the
target that the chain does not see. Both orders occur below, which is why the
reference is displayed rather than left implicit; the mass of the second part is
written $\sing(\alpha\mid\beta) = \alpha_{\sgp{\beta}}(\XX)$
(Definition \ref{def:singmass}), and it is that mass, rather than the measure, that
appears in the hypotheses. The
bar notation of \ref{conv:kernelnotation} is kept for the kernel itself whenever a set
is named, as in $T^{n}(A \mid x)$; the parts of the decomposition and the auxiliary
kernels introduced later are written as measures, so that a set is supplied on the
right, as in $T^{n}_{x,\sgp{\pi}}(\XX)$. Densities are not abbreviated in this way: a
density such as $t(y \mid x)$ or $s(y \mid x)$ is a map of its \emph{first} argument,
and the form $t(\,\cdot\mid x)$ names that map.

\item\label{conv:representative}
Densities are genuine maps defined at every point, not equivalence classes; when a
density is produced by an integral formula, that formula is the representative used.
A statement ``$f \le M$'' is meant $\aewhere$ unless the word ``every'' appears,
whereas ``$f(y) \ge c$ for every $y \in G$'' is meant pointwise on $G$.
\end{enumerate}

{\small
\begin{longtable}{@{}ll>{\raggedright\arraybackslash}p{0.40\textwidth}@{}}
\caption{The recurring symbols of these notes, together with their types. Every
other object is introduced in the text together with its domain and codomain, or with
the set it belongs to.}\label{tab:symbols}\\
\toprule
Symbol & Type & Role \\
\midrule
\endfirsthead
\toprule
Symbol & Type & Role \\
\midrule
\endhead
$\XX$                  & a set                                        & state space \\
$\XXA$                 & a $\sigma$-algebra on $\XX$                  & measurable sets \\
$A, \YY, E, G, L$      & elements of $\XXA$, i.e.\ sets               & measurable subsets of $\XX$; $E$ is a Hahn set (Definition \ref{def:tv}) \\
$P$                    & an element of $\XXA$, i.e.\ a set            & a cell of a finite partition (Proposition \ref{prop:jointdensity}) \\
$\Xp, \XX_{n}$         & elements of $\XXA$, i.e.\ sets               & support of $\pi$; sets of \hyp{E}{E} \\
$x, y, z$              & elements of $\XX$, i.e.\ points              & states \\
$\pi$                  & a map $\XXA \to [0,1]$                       & invariant probability measure \\
$\mu, \nu, \alpha, \beta, \rho$ & maps $\XXA \to [0,1]$               & probability measures \\
$h, h_{n}$             & maps $\XXA \to \R$                           & finite signed measures \\
$\zeta, \sigma, \sigma_{n}$ & maps $\XXA \to [0,\infty)$              & finite nonnegative measures \\
$\lambda$              & a map $\XXA \to [0,\infty]$                  & $\sigma$-finite reference measure; occurs only in the statements phrased with one \\
$T$                    & a map $\XXA \times \XX \to [0,1]$            & Markov kernel \\
$T(A \mid x)$          & a number in $[0,1]$                          & transition probability \\
$T^{n}(A \mid x)$      & a number in $[0,1]$                          & $n$-step transition probability \\
$\Probac$              & a subset of $\Prob$                          & the laws with a density with respect to $\pi$ (Remark \ref{rem:singdist}) \\
$T_{x}, T^{n}_{x}$     & elements of $\Prob$                          & law of the chain started at $x$, after one step and after $n$ (\ref{conv:Tx}) \\
$\alpha_{\acp{\beta}}, \alpha_{\sgp{\beta}}$ & maps $\XXA \to [0,\infty)$        & the parts of $\alpha$ that $\beta$ does and does not see (Definition \ref{def:lebesgue}) \\
$T^{n}_{x,\acp{\pi}}, T^{n}_{x,\sgp{\pi}}$ & maps $\XXA \to [0,1]$              & the same for $\alpha := T^{n}_{x}$ and $\beta := \pi$ \\
$W_{x}$                & a map $\XXA \to [0,1]$                       & residual kernel after subtracting a minorant (Lemma \ref{lem:uniformminorant}) \\
$T \circ \mu$          & a map $\XXA \to [0,1]$                       & law after one step \\
$t$                    & a map $\XX \times \XX \to (0,\infty)$        & transition density \\
$t(y \mid x)$          & a number in $(0,\infty)$                     & value of the density \\
$t_{n}$                & a map $\XX \times \XX \to [0,\infty)$        & $n$-step density \\
$s$                    & a map $\XX \times \XX \to [0,\infty)$        & generic minorant density in \hyp{P}{P} and Lemma \ref{lem:positivity} \\
$u$                    & a map $\XX \times \XX \to [0,\infty)$        & uniform minorant density in Lemma \ref{lem:uniformminorant} \\
$k$                    & a map $\XX \times \XX \to [0,\infty)$        & density of the moving part (Section \ref{sec:mh}) \\
$a$                    & a map $\XX \times \XX \to [0,1]$             & acceptance probability (Section \ref{sec:mh}) \\
$\theta, r$            & maps $\XX \to [0,1]$                         & moving and holding probability (Section \ref{sec:mh}) \\
$\ind_{A}$             & a map $\XX \to \{0,1\}$                      & indicator map of the set $A \in \XXA$ \\
$f, g, \tilde f$       & maps $\XX \to [0,\infty]$                    & densities with respect to $\pi$ \\
$f$                    & a map $\XX \to \R$                           & the observable of Section \ref{sec:ergodic} (the only place where $f$ is not a density) \\
$g_{\alpha}$           & a map $\XX \to [0,\infty)$                   & density of the absolutely continuous part of $\alpha$ (Subsection \ref{ssec:lebesgue}) \\
$t_{n}$                & a map $\XX\times\XX \to [0,\infty)$          & jointly measurable version of $g_{T^{n}_{x}}$ (Proposition \ref{prop:jointdensity}) \\
$u, v$                 & maps $\XX \to [0,\infty)$                    & densities with respect to $\lambda = \pi + \alpha$ (Proposition \ref{prop:jointdensity}) \\
$\delta_{x}$           & a map $\XXA \to \{0,1\}$                     & the Dirac measure at the point $x$ \\
$\tv{\mu - \nu}$       & a number in $[0,1]$                          & total variation distance \\
$M, M'$                & numbers in $[1,\infty)$                      & domination constants \\
$M_{0}$                & a number in $[1,\infty)$                     & truncation level (Lemma \ref{lem:truncation}) \\
$\delta$               & a number in $(0,1)$                          & lower bound on a measure \\
$\eta$                 & a number in $(0,1]$                          & lower bound on a density \\
$\gamma$               & a number in $(0,1]$                          & overlap constant \\
$N, K, k, k_{0}$       & numbers in $\N$                              & numbers of steps, indices \\
$n, n_{1}, m, j$       & numbers in $\Nz$                             & indices \\
$\Delta$               & a map $\Nz \to [0,1]$, $n \mapsto \Delta_{n}$& distance along the chain \\
$d$                    & a number in $[0,1]$                          & limit of $\Delta$ \\
$c, c_{M}$             & numbers in $(0,1]$                           & masses \\
$u_{m}, v_{j}$         & numbers in $[0,1]$                           & error terms in Proposition \ref{prop:weaklaw} \\
$\varepsilon$          & a number in $(0,1)$                          & accuracy \\
$\Gamma$               & an element of $\XXA^{\otimes\Nz}$            & an event on path space (Section \ref{sec:ergodic}) \\
$\vartheta$            & a map $\XX^{\Nz} \to \XX^{\Nz}$              & the shift (Section \ref{sec:ergodic}) \\
$r$                    & a number in $[1,\infty)$                     & exponent of a space $L^{r}$; occurs only in Section \ref{sec:ergodic}, where the map $r$ of Section \ref{sec:mh} does not appear \\
\bottomrule
\end{longtable}
}

\subsection{Markov kernels, invariance and total variation}\label{ssec:kernels}

\begin{remark}[No structure on the state space is assumed]\label{rem:nostructure}
Throughout, $(\XX,\XXA)$ is nothing but a set together with a $\sigma$-algebra on it.
No topology, no metric, no countability and no standard Borel hypothesis are used
anywhere in the proofs below; the tools employed are Tonelli's and Fubini's theorems,
the Hahn--Jordan decomposition of a finite signed measure, and the Radon--Nikodym
theorem, all of which are available on an arbitrary measurable space; the Lebesgue
decomposition is not quoted but proved, in Theorem \ref{thm:lebesgue}.
More than that: \emph{no hypothesis of any convergence theorem proved here is a
hypothesis on the space.} Every further hypothesis below --- \hyp{D}{D}, \hyp{M}{M},
\hyp{E}{E}, \hyp{P}{P}, \hyp{J}{J}, \hyp{L}{L}, \hyp{S}{S} --- is a condition on the
kernel $T$ and its relation to $\pi$. Five of them (all but \hyp{L}{L} and \hyp{S}{S},
which mention no density at all) ask for a density measurable in the starting point and
the endpoint jointly, which on a bare measurable space is a real request and not a
formality; but it is a request made of $T$, not of $\XXA$.

\emph{Countable generation of the $\sigma$-algebra} $\XXA$ is accordingly assumed by no
theorem. It is assumed in one place only, and there as a \emph{sufficient condition}
rather than as a hypothesis of the theory: Proposition \ref{prop:jointdensity} shows
that it implies \hyp{J}{J}, by a martingale construction, and that is the whole of its
role. (It is discussed in several other places, and in
Remark \ref{rem:tvnotmeasurable} it appears in a second capacity, as one of two
conditions under which $x \mapsto \tv{T^{n}_{x}-\pi}$ would be measurable --- neither
of which is used either.) How the \emph{jointly measurable transition densities} that we
elsewhere \emph{assume} can be constructed, when one starts instead from absolute
continuity alone, is explained in Remarks \ref{rem:jointmeasurability} and
\ref{rem:measurableXn} and carried out in Proposition \ref{prop:jointdensity};
Remark \ref{rem:PvsJ} compares that assumption with \hyp{P}{P}. A standard Borel
space --- a measurable space whose $\sigma$-algebra is the Borel $\sigma$-algebra of
some Polish topology --- is countably generated, as are $\R^{d}$ with its Borel sets
and countable products of such spaces, so \hyp{C}{C}, and with it \hyp{J}{J}, holds
wherever the algorithms discussed here are run.
\end{remark}

\begin{definition}[Markov kernel]\label{def:kernel}
A \emph{Markov kernel} on $(\XX,\XXA)$ is a map
\[
  T : \XXA \times \XX \longrightarrow [0,1], \qquad (A,x) \longmapsto T(A \mid x),
\]
such that
\begin{enumerate}[label=(\roman*),nosep,topsep=3pt]
  \item for every $x \in \XX$, the map $T_{x} : \XXA \to [0,1]$ is a probability measure;
  \item for every $A \in \XXA$, the map $T(A \mid \cdot\,) : \XX \to [0,1]$ is measurable.
\end{enumerate}
The \emph{$n$-step kernels} are defined recursively by
$T^{0}(A \mid x) := \ind_{A}(x)$ and
\[
  T^{n+1}(A \mid x) := \int_{\XX} T^{n}(A \mid z)\, T(\dd z \mid x)
  \qquad (n \in \Nz,\ A \in \XXA,\ x \in \XX),
\]
each of which is again a Markov kernel.
\end{definition}

\begin{definition}[Action on measures]\label{def:action}
For a finite signed measure $h \in \Sign$ and a number $n \in \Nz$ we define
the finite signed measure $T^{n} \circ h \in \Sign$ by
\[
  (T^{n} \circ h)(A) := \int_{\XX} T^{n}(A \mid x)\, h(\dd x)
  \qquad (A \in \XXA),
\]
the integral of a bounded measurable map against $h$ being defined, as usual, as its
integral against $h^{+}$ minus its integral against $h^{-}$.
The map $h \mapsto T^{n}\circ h$ is a linear map $\Sign \to \Sign$ which
maps $\Prob$ into $\Prob$, and it satisfies
\[
  T \circ (T^{n} \circ h) = T^{n+1} \circ h \qquad (n \in \Nz),
\]
because for every bounded measurable map $g : \XX \to \R$ one has
$\int g \dd(T^{n}\circ h) = \int \bigl( \int g(z)\,T^{n}(\dd z \mid x) \bigr)
h(\dd x)$. That identity holds for indicator maps by the displayed formula, hence for
nonnegative simple maps by linearity, hence for every measurable
$g : \XX \to [0,\infty]$ by monotone convergence, applied to $h^{+}$ and to $h^{-}$
separately; a bounded measurable $g : \XX \to \R$ is then treated by writing
$g = g^{+}-g^{-}$.
\end{definition}

\begin{definition}[Invariant probability measure]\label{def:invariant}
A probability measure $\pi \in \Prob$ is called \emph{invariant} for $T$ if
$T \circ \pi = \pi$, that is, if $(T\circ\pi)(A) = \pi(A)$ for every $A \in \XXA$.
\end{definition}

\begin{definition}[Total variation distance]\label{def:tv}
For $h \in \Sign$ put
\[
  \tv{h} := \sup_{A \in \XXA} \bigl| h(A) \bigr| \in [0,\infty),
\]
so that $\tv{\cdot} : \Sign \to [0,\infty)$ is a map; it satisfies the triangle
inequality $\tv{h_{1}+h_{2}} \le \tv{h_{1}} + \tv{h_{2}}$, since
$|h_{1}(A)+h_{2}(A)| \le |h_{1}(A)| + |h_{2}(A)|$ for every $A$. For
$\mu,\nu \in \Prob$ the number $\tv{\mu - \nu} \in [0,1]$ is called the \emph{total
variation distance} of $\mu$ and $\nu$.

Recall the Hahn--Jordan decomposition: for $h \in \Sign$ there is a set
$E \in \XXA$, called a \emph{Hahn set for $h$}, such that $h(A) \ge 0$ for every
measurable $A \subseteq E$ and $h(A) \le 0$ for every measurable
$A \subseteq \XX \setminus E$; the two nonnegative measures
$h^{+}(A) := h(A \cap E)$ and $h^{-}(A) := -h(A \setminus E)$ then satisfy
$h = h^{+}-h^{-}$ and do not depend on the choice of $E$. If moreover $h(\XX) = 0$,
then
\[
  h^{+}(\XX) = h^{-}(\XX) = \tv{h} ,
\]
since $\sup_{A \in \XXA}h(A) = h^{+}(\XX)$ and
$\inf_{A \in \XXA}h(A) = -h^{-}(\XX)$.
\end{definition}

\noindent
When both measures have densities with respect to some common reference measure, the
distance can be computed from them. Below the reference is always the target $\pi$,
but nothing in the statement requires that.

\begin{lemma}[Total variation in terms of densities]\label{lem:tvformulas}
Let $(\XX,\XXA)$ be a measurable space, let $\lambda$ be a $\sigma$-finite measure on
it, and let $f, g : \XX \to [0,\infty]$ be measurable, finite $\lambda$-almost
everywhere, with $\int_{\XX} f \dd\lambda = \int_{\XX} g \dd\lambda = 1$. Let
$\mu,\nu \in \Prob$ be given by $\mu(A) := \int_{A} f \dd\lambda$ and
$\nu(A) := \int_{A} g \dd\lambda$ for $A \in \XXA$. Then
\[
  \tv{\mu - \nu}
  \;=\; \int_{\XX} (f-g)^{+} \dd\lambda
  \;=\; \tfrac12 \int_{\XX} |f-g| \dd\lambda ,
\]
the integrands being defined $\lambda$-almost everywhere.
\end{lemma}

\begin{proof}
For $A \in \XXA$ the number $\mu(A) - \nu(A) = \int_{A}(f-g)\dd\lambda$ is largest for
$A := \{f > g\} \in \XXA$, with value $\int (f-g)^{+}\dd\lambda$, and smallest for
$A := \{f < g\}$, with value $-\int (f-g)^{-}\dd\lambda$. Since
$\int (f-g)\dd\lambda = 1-1 = 0$, these two numbers agree up to sign, which gives the
first equality; their sum is $\int|f-g|\dd\lambda$, which gives the second.
\end{proof}

\subsection{The standing assumptions \texorpdfstring{\hyp{A}{A}}{(A)}}\label{ssec:standing}

\begin{assumption}[Standing assumptions]\label{ass:standing}
Throughout these notes:
\begin{enumerate}[nosep,topsep=3pt,leftmargin=3em]
  \item[\hyptgt{A1}{A1}] $(\XX,\XXA)$ is a measurable space, i.e.\ $\XX$ is a set and $\XXA$ is a $\sigma$-algebra on $\XX$;
  \item[\hyptgt{A2}{A2}] $\pi \in \Prob$, i.e.\ $\pi : \XXA \to [0,1]$ is a probability measure;
  \item[\hyptgt{A3}{A3}] $T$ is a Markov kernel on $(\XX,\XXA)$ with $T \circ \pi = \pi$.
\end{enumerate}
We write \hyptgt{A}{A} for the conjunction of \hyp{A1}{A1}, \hyp{A2}{A2} and
\hyp{A3}{A3}, and use that abbreviation in every statement that assumes all three; the
three individual labels remain available, and are used where only one of them is
needed. Everything in Section \ref{sec:core} uses \hyp{A}{A} only.
\end{assumption}

\begin{remark}[A guide to the labels]\label{rem:labels}
The further hypotheses are introduced where they are first needed, and each is
labelled by a single mnemonic letter. Every occurrence of a label in these notes,
including those in section headings, is a hyperlink to the place where it is defined.
\begin{center}
\begin{tabular}{@{}lll@{}}
\toprule
Label & Reads & Introduced in \\
\midrule
\hyp{A}{A} & the standing \emph{a}ssumptions \hyp{A1}{A1}, \hyp{A2}{A2}, \hyp{A3}{A3} together & Section \ref{sec:setting} \\
\midrule
\hyp{D}{D} & a strictly positive transition \emph{d}ensity & Section \ref{sec:first} \\
\hyp{M}{M} & a \emph{m}ixed kernel: absolutely continuous part plus an atom & Section \ref{sec:mh} \\
\hyp{E}{E} & such a density \emph{e}ventually, after finitely many steps & Section \ref{sec:general} \\
\hyp{P}{P} & a strictly \emph{p}ositive minorant density after finitely many steps & Section \ref{sec:core} \\
\hyp{J}{J} & a \emph{j}ointly measurable density for $T^{n}_{x,\acp{\pi}}$ & Section \ref{sec:core} \\
\hyp{L}{L} & asymptotic domination: $\sing(\pi \mid T^{n}_{x}) \to 0$, in the \emph{L}ebesgue sense & Section \ref{sec:core} \\
\hyp{S}{S} & asymptotic absolute continuity: the \emph{s}ingular part of $T^{n}_{x}$ vanishes & Section \ref{sec:core} \\
\midrule
\eqref{prop:U} & \emph{u}niform overlap of the images of dominated laws & Section \ref{sec:core} \\
\eqref{prop:R} & \emph{r}egularisation: closeness to a dominated law & Section \ref{sec:core} \\
\midrule
\hyp{C}{C} & $\XXA$ is \emph{c}ountably generated --- not assumed, but sufficient for \hyp{J}{J} & Section \ref{sec:core} \\
\bottomrule
\end{tabular}
\end{center}
\noindent
The first is not a hypothesis but an abbreviation: \hyp{A}{A} names the standing
assumptions of Assumption \ref{ass:standing}, which are in force everywhere, and every
numbered statement below begins by recalling it. The seven in the second block are
hypotheses on the kernel, and appear as assumptions:
\hyp{D}{D}, \hyp{M}{M} and \hyp{E}{E} are the three settings, and each of them implies
the pair \hyp{P}{P}, \hyp{S}{S} from which the convergence theorem is proved
(Theorem \ref{thm:criterion}), and also the weaker pair \hyp{S}{S}, \hyp{L}{L} in
which the criterion is most simply read, under the further hypothesis \hyp{J}{J} that
supplies a measurable density (Theorem \ref{thm:asympequiv}). The two in the
third block are the two halves of that proof; they are the hypotheses of
Proposition \ref{thm:dominated} and Theorem \ref{thm:abstract} and are verified once
and for all in Subsection \ref{ssec:PS}. After that they are referred to only in
Remark \ref{rem:aposteriori}, which upgrades them from the points of $\XX$ to
arbitrary initial laws, and in Remarks \ref{rem:uniform} and \ref{rem:pcn}. The
last, \hyp{C}{C}, is the one condition on the \emph{space} that occurs in these notes,
and no convergence theorem assumes it: it is assumed only in
Proposition \ref{prop:jointdensity}, as a sufficient condition for \hyp{J}{J}
(Remark \ref{rem:nostructure}).
\end{remark}

\begin{remark}[Which reference measure the proofs use]\label{rem:whypi}
The hypotheses of Sections \ref{sec:first}--\ref{sec:general} are stated in terms of
densities with respect to $\pi$ itself, and each of the first three convergence
theorems is accompanied by a corollary stated in terms of densities with respect to a
$\sigma$-finite reference measure $\lambda$, which is the form in which they are
usually applied (Corollaries \ref{cor:reference}, \ref{cor:mhreference} and
\ref{cor:referencegen}). The two formulations are related by
$t(y \mid x) = \tau(y \mid x)/p(y)$ on $\{p > 0\}$, where $\pi = p\lambda$, and the
translation is carried out in the proof of each corollary.

One could instead fix $\lambda$ once and for all and phrase the entire development in
terms of $\lambda$-densities, the case $\lambda = \pi$ being included. We have not
done so, because the quantitative content of the argument is normalised by $\pi$ and
not by $\lambda$. What Lemma \ref{lem:positivity} and
Proposition \ref{prop:minorisation} bound from below are the $\pi$-measure of a set
of starting points, $\pi(L_{\mu}) \ge \delta$, obtained from the domination
$\mu \le M\pi$, and the $\pi$-measure of a set of endpoints, $\pi(G) \ge \tfrac12$.
Neither statement survives the replacement of $\pi$ by a $\sigma$-finite measure, for
which ``a set of large measure'' is not meaningful. Sections
\ref{sec:core}--\ref{sec:general} therefore work with $\pi$, and $\lambda$ appears in
the statements addressed to applications.
\end{remark}

\section{General Markov chain convergence criteria}\label{sec:core}

Throughout this section only \hyp{A}{A} are assumed: a measurable
space, an invariant probability measure and a Markov kernel, and nothing else.

The section falls into two halves. Subsections
\ref{ssec:contraction}--\ref{ssec:criterion} prove the convergence theorem from two
properties of the pair $(T,\pi)$, called \eqref{prop:U} and \eqref{prop:R}. They
involve no density of any kind, and rest on two short lemmas. Subsections
\ref{ssec:minorants}--\ref{ssec:PS} then replace those two properties by two
hypotheses, \hyp{P}{P} and \hyp{S}{S}, which can be checked on a given kernel and
which are what the rest of these notes verifies; that is where densities enter the
argument. Subsection \ref{ssec:examples} shows that the two hypotheses are
independent of one another and that neither may be dropped.

\subsection{Contraction and preservation of domination}\label{ssec:contraction}

Two consequences of \hyp{A}{A} are all that the convergence theorem
needs. The first holds because $T$ is a Markov kernel, the second because $\pi$ is
invariant.

\begin{lemma}[Contraction]\label{lem:contraction}
For every $h \in \Sign$ with $h(\XX) = 0$ and every number $n \in \Nz$,
\[
  \tv{T^{n} \circ h} \;\le\; \tv{h} .
\]
In particular $\tv{T^{n}\circ\mu - T^{n}\circ\nu} \le \tv{\mu-\nu}$ for all
$\mu,\nu \in \Prob$, and the map
$\Nz \to [0,1]$, $n \mapsto \tv{T^{n}\circ\mu - T^{n}\circ\nu}$, is non-increasing.
\end{lemma}

\begin{proof}
It suffices to treat $n = 1$ and to iterate, using Definition \ref{def:action} and
$(T\circ h)(\XX) = h(\XX) = 0$. Let $h = h^{+}-h^{-}$ be
the Jordan decomposition into nonnegative measures
$h^{\pm} : \XXA \to [0,\infty)$, so that
$h^{+}(\XX) = h^{-}(\XX) = \tv{h}$ by Definition \ref{def:tv}. For $A \in \XXA$, using
$0 \le T(A\mid x) \le 1$,
\[
  (T\circ h)(A)
  = \int_{\XX} T(A \mid x)\,h^{+}(\dd x) - \int_{\XX} T(A \mid x)\,h^{-}(\dd x)
  \;\le\; h^{+}(\XX) - 0 \;=\; \tv{h},
\]
and symmetrically $(T\circ h)(A) \ge -\tv{h}$. Take the supremum over
$A \in \XXA$.
\end{proof}

\begin{lemma}[Domination by $M\pi$ is preserved]\label{lem:domination}
Assume \hyp{A}{A} and let $M \in [0,\infty)$ be a number and let $\mu : \XXA \to [0,\infty)$ be a measure
with $\mu \le M\pi$. Then $T^{n} \circ \mu \le M\pi$ for every number $n \in \Nz$.
\end{lemma}

\begin{proof}
Again it suffices to treat $n=1$. Let $A \in \XXA$. The map
$T(A \mid \cdot\,) : \XX \to [0,1]$ is measurable and nonnegative, so
\ref{conv:order} and the invariance of $\pi$ give
\[
  (T \circ \mu)(A) = \int_{\XX} T(A \mid x)\,\mu(\dd x)
  \;\le\; M \int_{\XX} T(A \mid x)\, \pi(\dd x)
  \;=\; M\,(T\circ\pi)(A) \;=\; M\,\pi(A). \qedhere
\]
\end{proof}

\noindent
This is what makes the contradiction argument in Proposition \ref{thm:dominated} close:
the domination constant does not deteriorate with the number of steps.

\subsection{The two properties \texorpdfstring{\eqref{prop:U} and \eqref{prop:R}}{(U) and (R)}}\label{ssec:UR}

The convergence theorem is proved from two properties of the pair $(T,\pi)$. The
first isolates what has to be true for two \emph{dominated} laws to be brought
together; the second says that every law becomes dominated, up to an arbitrarily
small error, after finitely many steps. Both are properties to be verified, not
standing assumptions; Subsection \ref{ssec:PS} derives both from the two hypotheses
that the rest of these notes verifies, after which they are needed only in
Remarks \ref{rem:aposteriori}, \ref{rem:uniform} and \ref{rem:pcn}.

\begin{property}[Uniform overlap]\label{prop:Ulab}
For every number $M \in [1,\infty)$ there exist numbers $N \in \N$ and
$\gamma \in (0,1]$ such that
\begin{equation}
  \tv{T^{N}\circ\mu - T^{N}\circ\nu} \le 1-\gamma
  \qquad \text{for all } \mu,\nu \in \Prob \text{ with } \mu \le M\pi,\ \nu \le M\pi .
  \tag{U}\label{prop:U}
\end{equation}
\end{property}

\begin{property}[Regularisation]\label{prop:Rlab}
For every point $x \in \XX$ and every number $\varepsilon \in (0,1)$ there exist
numbers $n_{1} \in \Nz$ and $M \in [1,\infty)$ and a probability measure
$\rho \in \Prob$ with
\begin{equation}
  \rho \le M\pi
  \qquad\text{and}\qquad
  \tv{T^{n_{1}}_{x} - \rho} \le \varepsilon .
  \tag{R}\label{prop:R}
\end{equation}
\end{property}

\noindent
The two properties quantify differently, and deliberately so.

In \eqref{prop:U} the numbers $N$ and $\gamma$ must serve \emph{every} pair of laws
dominated by $M\pi$ at once, and not merely each pair separately: the constant
$\gamma$ will have to serve the renormalised measures $\alpha_{n},\beta_{n}$
constructed in the proof of Proposition \ref{thm:dominated} simultaneously for all
$n \in \Nz$, which is why \eqref{prop:U} quantifies over the class and not over a
pair. Nor can that class be replaced by the points of $\XX$: a Dirac measure
$\delta_{x}$ satisfies $\delta_{x} \le M\pi$ only if $\pi(\{x\}) \ge 1/M$, so on a
space without atoms \eqref{prop:U} read at Dirac measures would be vacuous.
Domination is precisely the hypothesis under which two laws can be forced together,
and Dirac measures are the extreme case of its failure.

In \eqref{prop:R}, by contrast, everything may depend on $x$ and on $\varepsilon$, and
only \emph{one} good time $n_{1}$ is asked for, the passage from one such time to all
later ones being Lemma \ref{lem:monotone}(iv) (Remark \ref{rem:whymonotone}). Note
also that \eqref{prop:R} is imposed only at the points of $\XX$, that is, only at the
Dirac measures, and not at every initial law: Lemma \ref{lem:pointwise} below shows
that nothing is lost by this, and Remark \ref{rem:aposteriori} shows that the
apparently stronger statement for every $\mu \in \Prob$ follows a posteriori.

\subsection{The convergence theorem under \texorpdfstring{\hyp{A}{A}, \protect\eqref{prop:U} and \protect\eqref{prop:R}}{(A), (U) and (R)}}\label{ssec:criterion}

\begin{proposition}[Convergence for dominated pairs]\label{thm:dominated}
Assume \hyp{A}{A} and \eqref{prop:U}. Let $M \in [1,\infty)$ be a
number and let $\mu,\nu \in \Prob$ satisfy $\mu \le M\pi$ and $\nu \le M\pi$. Then
\[
  \lim_{n\to\infty} \tv{T^{n}\circ\mu - T^{n}\circ\nu} = 0 .
\]
In particular $\tv{T^{n}\circ\mu - \pi} \to 0$ for every $\mu \in \Prob$ with
$\mu \le M\pi$, because $\pi \le M\pi$ and $T^{n}\circ\pi = \pi$.
\end{proposition}

\begin{proof}
Define the signed measures $h_{n} := T^{n}\circ\mu - T^{n}\circ\nu \in \Sign$
$(n \in \Nz)$ and the map
\[
  \Delta : \Nz \to [0,1], \qquad \Delta_{n} := \tv{h_{n}} .
\]
By Lemma \ref{lem:contraction} the sequence $(\Delta_{n})_{n\in\Nz}$ is
non-increasing, so the number $d := \lim_{n\to\infty}\Delta_{n} \in [0,1]$ exists.
Assume, towards a contradiction, that $d > 0$; then $\Delta_{n} \ge d > 0$ for every
$n \in \Nz$.

Let $h_{n} = h_{n}^{+} - h_{n}^{-}$ be the Jordan decomposition into nonnegative
measures $h_{n}^{\pm} : \XXA \to [0,\infty)$. Since
$h_{n}(\XX) = 0$ we have $h_{n}^{+}(\XX) = h_{n}^{-}(\XX) = \Delta_{n}$ by
Definition \ref{def:tv}, so that
\[
  \alpha_{n} := \frac{h_{n}^{+}}{\Delta_{n}} \in \Prob,
  \qquad
  \beta_{n} := \frac{h_{n}^{-}}{\Delta_{n}} \in \Prob
  \qquad (n \in \Nz)
\]
are probability measures. If $E \in \XXA$ is a Hahn set for $h_{n}$ in the sense of Definition \ref{def:tv},
then for every set $A \in \XXA$
\[
  h_{n}^{+}(A) = h_{n}(A \cap E) \le (T^{n}\circ\mu)(A \cap E) \le (T^{n}\circ\mu)(A),
\]
so $h_{n}^{+} \le T^{n}\circ\mu \le M\pi$ by Lemma \ref{lem:domination}, and likewise
$h_{n}^{-} \le T^{n}\circ\nu \le M\pi$. Hence, with the number
$M' := M/d \in [1,\infty)$,
\[
  \alpha_{n} \le M'\pi \qquad\text{and}\qquad \beta_{n} \le M'\pi
  \qquad\text{for every } n \in \Nz ,
\]
and, crucially, $M'$ does not depend on $n$.

Let $N \in \N$ and $\gamma \in (0,1]$ be the numbers provided by \eqref{prop:U} for
the number $M'$. Since $T^{N}\circ h_{n} = h_{n+N}$ by linearity, and since
$T^{N}\circ\alpha_{n}$ and $T^{N}\circ\beta_{n}$ are probability measures,
\[
  \frac{\Delta_{n+N}}{\Delta_{n}}
  = \frac{\tv{T^{N}\circ h_{n}}}{\Delta_{n}}
  = \tv{T^{N}\circ\alpha_{n} - T^{N}\circ\beta_{n}}
  \;\le\; 1-\gamma
  \qquad (n \in \Nz),
\]
that is, $\Delta_{n+N} \le (1-\gamma)\Delta_{n}$. By induction
$\Delta_{kN} \le (1-\gamma)^{k}\Delta_{0}$ for every $k \in \N$, and the right-hand
side tends to the number $0$. This contradicts $\Delta_{n} \ge d > 0$. Hence $d=0$.
\end{proof}

\begin{remark}[Why the contradiction is legitimate]
The number $M' = M/d$ is defined in terms of the limit $d$, which is unknown; but $d$
is a fixed number once $\mu$ and $\nu$ are fixed, so $\gamma$ is a fixed number as
well and the geometric decay $\Delta_{kN} \le (1-\gamma)^{k}\Delta_{0}$ is a
legitimate consequence. Only the \emph{assumption} $d>0$ is provisional. Observe that
the argument produces no rate for the original problem, since $\gamma$ is not known
before $d$ is.
\end{remark}

Property \eqref{prop:R} is imposed only at the points of $\XX$, so the theorem below
first delivers convergence from every point. The passage from there to an arbitrary
initial law is free, and in particular needs no hypothesis on $(\XX,\XXA)$; we record
it first.

\begin{lemma}[Convergence from every point suffices]\label{lem:pointwise}
Assume \hyp{A}{A}, let $\mu \in \Prob$, and suppose there is a set $\XX' \in \XXA$
with $\mu(\XX') = 1$ such that
\[
  \lim_{n\to\infty} \tv{T^{n}_{x} - \pi} = 0
  \qquad\text{for every } x \in \XX' .
\]
Then $\lim_{n\to\infty}\tv{T^{n}\circ\mu - \pi} = 0$. In particular, if the chain
converges in total variation from every starting point, then it converges from every
initial distribution.
\end{lemma}

\begin{proof}
Let $n \in \Nz$. The signed measure $h_{n} := T^{n}\circ\mu - \pi$ lies in $\Sign$
and satisfies $h_{n}(\XX) = 0$, so by Definition \ref{def:tv} there is a Hahn set
$E_{n} \in \XXA$ for $h_{n}$ with $\tv{h_{n}} = h_{n}(E_{n})$. Fix one such $E_{n}$
for each $n$ and define
\[
  \varphi_{n} : \XX \longrightarrow [-1,1],
  \qquad
  \varphi_{n}(x) := T^{n}(E_{n} \mid x) - \pi(E_{n}) ,
\]
which is measurable by Definition \ref{def:kernel}, the set $E_{n}$ being fixed.
By Definition \ref{def:action} and $\mu(\XX) = 1$,
\[
  \tv{T^{n}\circ\mu - \pi} \;=\; h_{n}(E_{n})
  \;=\; \int_{\XX} \varphi_{n} \dd\mu .
\]
By Definition \ref{def:tv} again, $\varphi_{n}(x) \le \tv{T^{n}_{x}-\pi}$
for every $x \in \XX$, so $\limsup_{n\to\infty}\varphi_{n}(x) \le 0$ for every
$x \in \XX'$ and hence for $\mu$-almost every $x$. The maps $1 - \varphi_{n}$ are
nonnegative and measurable, so Fatou's lemma gives
\[
  \int_{\XX} \Bigl( 1 - \limsup_{n\to\infty}\varphi_{n} \Bigr) \dd\mu
  \;\le\; \liminf_{n\to\infty} \int_{\XX} (1-\varphi_{n}) \dd\mu
  \;=\; 1 - \limsup_{n\to\infty}\int_{\XX}\varphi_{n}\dd\mu ,
\]
that is,
$\limsup_{n}\int\varphi_{n}\dd\mu \le \int \limsup_{n}\varphi_{n}\dd\mu \le 0$. As
$\int\varphi_{n}\dd\mu = \tv{T^{n}\circ\mu-\pi} \ge 0$ for every $n$, the limit is
the number $0$.
\end{proof}

\begin{remark}[Why one Hahn set per step, and not dominated convergence]\label{rem:pointwise}
The expected proof of Lemma \ref{lem:pointwise} is to bound
$\tv{T^{n}\circ\mu-\pi} \le \int \tv{T^{n}_{x}-\pi}\,\mu(\dd x)$ and appeal
to dominated convergence. That proof is not available here: the integrand is a
supremum over $\XXA$ and need not be measurable
(Remarks \ref{rem:nostructure} and \ref{rem:tvnotmeasurable}). Fixing a single Hahn
set $E_{n}$ for each time $n$ replaces that supremum by one measurable map
$\varphi_{n}$, which is all Fatou's lemma needs; the supremum then enters only as a
pointwise upper bound on $\varphi_{n}$, where measurability is irrelevant. The device
costs nothing, and it is what allows \eqref{prop:R} and \hyp{S}{S} to be imposed at
the points of $\XX$ alone.
\end{remark}

\begin{theorem}[The general convergence criterion without densities]\label{thm:abstract}
Assume \hyp{A}{A}, \eqref{prop:U} and \eqref{prop:R}. Then
\[
  \lim_{n\to\infty} \tv{T^{n}\circ\mu - \pi} = 0
  \qquad\text{for every } \mu \in \Prob,
\]
and $\pi$ is the unique invariant probability measure of $T$.
\end{theorem}

\begin{proof}
Let $x \in \XX$, put $\mu := \delta_{x} \in \Prob$, so that
$T^{n}\circ\mu = T^{n}_{x}$ for every $n$, and let
$\varepsilon \in (0,1)$ be a number. Choose $n_{1} \in \Nz$, $M \in [1,\infty)$ and
$\rho \in \Prob$ with $\rho \le M\pi$ and
$\tv{T^{n_{1}}\circ\mu - \rho} \le \varepsilon$, as provided by \eqref{prop:R}. For
every $n \in \Nz$, Lemma \ref{lem:contraction} and $T^{n}\circ\pi = \pi$ give
\[
  \tv{T^{n+n_{1}}\circ\mu - \pi}
  \;\le\; \tv{T^{n}\circ(T^{n_{1}}\circ\mu) - T^{n}\circ\rho} + \tv{T^{n}\circ\rho - \pi}
  \;\le\; \varepsilon + \tv{T^{n}\circ\rho - \pi} .
\]
By Proposition \ref{thm:dominated} the last term tends to $0$, so
$\limsup_{n\to\infty}\tv{T^{n}_{x} - \pi} \le \varepsilon$. As the number
$\varepsilon \in (0,1)$ was arbitrary, the limit is $0$, and this holds for every
$x \in \XX$. Lemma \ref{lem:pointwise}, applied with $\XX' := \XX$, therefore gives
$\tv{T^{n}\circ\nu-\pi} \to 0$ for every $\nu \in \Prob$.

Uniqueness: let $\pi' \in \Prob$ satisfy $T \circ \pi' = \pi'$. Then for every
$n \in \Nz$ the number $\tv{\pi' - \pi}$ equals $\tv{T^{n}\circ\pi' - \pi}$, which
tends to $0$. Hence $\tv{\pi'-\pi} = 0$, i.e.\ $\pi'(A) = \pi(A)$ for every
$A \in \XXA$.
\end{proof}

\noindent
Neither property may be dispensed with, and neither implies the other. Examples are
given in Subsection \ref{ssec:examples}, phrased in terms of the two equivalent
hypotheses introduced below.

\subsection{Minorants and quantitative positivity}\label{ssec:minorants}

The rest of this section produces \eqref{prop:U} and \eqref{prop:R} from hypotheses
that can be read off a given kernel. We begin with \eqref{prop:U}, which needs two
ingredients: a criterion for two probability measures to be close, and a quantitative
consequence of positivity.

\begin{lemma}[A common minorant bounds the distance]\label{lem:minorant}
Let $\alpha,\beta \in \Prob$ be probability measures and let
$\zeta : \XXA \to [0,\infty)$ be a nonnegative measure with $\zeta \le \alpha$
and $\zeta \le \beta$ in the sense of \ref{conv:order}. Then
\[
  \tv{\alpha - \beta} \;\le\; 1 - \zeta(\XX) .
\]
\end{lemma}

\begin{proof}
The set function $\alpha - \zeta$ is a nonnegative measure, hence monotone, and
$(\alpha-\zeta)(\XX) = 1 - \zeta(\XX)$. Therefore, for every $A \in \XXA$,
\[
  \alpha(A) - \beta(A) \;\le\; \alpha(A) - \zeta(A) \;=\; (\alpha-\zeta)(A)
  \;\le\; 1 - \zeta(\XX),
\]
using $\beta(A) \ge \zeta(A)$ in the first step. Exchanging the roles of $\alpha$
and $\beta$ gives the same bound for $\beta(A)-\alpha(A)$; now take the supremum over
$A \in \XXA$.
\end{proof}

\noindent
Lemma \ref{lem:minorant} is the engine of the whole development, and note that it
mentions no densities: to show that two laws are close it suffices to exhibit
\emph{one} nonnegative measure of substantial total mass lying below both of them.

The next lemma is pure measure theory: it contains no reference to $T$ at all. It is
the only place where strict positivity is used, and it is stated in the form needed
for both theorems.

\begin{lemma}[Quantitative positivity]\label{lem:positivity}
Let $\YY, \YY' \in \XXA$ be sets, let $\delta \in (0,1)$ be a number with
$\pi(\XX\setminus\YY') \le \delta/4$, and let
\[
  s : \XX \times \XX \longrightarrow [0,\infty), \qquad (x,y) \longmapsto s(y \mid x),
\]
be a $\XXA \otimes \XXA$-measurable map such that
\[
  \pi\bigl( \{ x \in \YY : s(y \mid x) = 0 \} \bigr) \;\le\; \frac{\delta}{4}
  \qquad\text{for every } y \in \YY' .
\]
Then there exist a number $\eta \in (0,1]$ and a set $G \in \XXA$ with
$\pi(G) \ge \tfrac12$ such that
\[
  \int_{L} s(y \mid x)\, \pi(\dd x) \;\ge\; \frac{\eta\delta}{2}
  \qquad
  \text{for every } y \in G \text{ and every } L \in \XXA \text{ with }
  L \subseteq \YY \text{ and } \pi(L) \ge \delta .
\]
\end{lemma}

\begin{proof}
Define the map
\[
  \Phi : (0,\infty) \times \XX \longrightarrow [0,1],
  \qquad
  \Phi(\eta,y) := \pi\bigl( \{ x \in \YY : s(y \mid x) < \eta \} \bigr).
\]
For fixed $\eta \in (0,\infty)$ the map $\Phi(\eta,\cdot\,) : \XX \to [0,1]$ is
measurable: the map
\[
  \XX\times\XX \to \{0,1\}, \qquad (x,y) \longmapsto \ind_{\YY}(x)\,\ind_{\{s(y \mid x)<\eta\}} ,
\]
is $\XXA\otimes\XXA$-measurable, so Tonelli's theorem applies.

Fix a point $y \in \YY'$. The sets $\{x \in \YY : s(y \mid x) < 1/k\}$, $k \in \N$,
decrease as $k$ increases and their intersection is
$\{x \in \YY : s(y \mid x) = 0\}$, of $\pi$-measure at most $\delta/4$ by hypothesis.
Since $\pi$ is finite, continuity from above gives
$\lim_{k\to\infty}\Phi(1/k,y) \le \delta/4$ for every $y \in \YY'$. Consequently the
sets
\[
  G_{k} := \{ y \in \YY' : \Phi(1/k,y) \le \delta/2 \} \in \XXA, \qquad k \in \N,
\]
increase in $k$ --- because $\eta \mapsto \Phi(\eta,y)$ is non-decreasing, so that
$k \mapsto \Phi(1/k,y)$ is non-increasing --- and their union is $\YY'$, since
$\delta/4 < \delta/2$. Hence
$\pi\bigl(\bigcup_{k \in \N} G_{k}\bigr) = \pi(\YY') \ge 1 - \delta/4 > \tfrac12$, and
by continuity from below there is a number $k_{0} \in \N$ with
$\pi(G_{k_{0}}) \ge \tfrac12$. Put $\eta := 1/k_{0} \in (0,1]$ and
$G := G_{k_{0}} \in \XXA$.

Let $y \in G$ and let $L \in \XXA$ with $L \subseteq \YY$ and $\pi(L) \ge \delta$.
With $E_{y} := \{x \in \XX : s(y \mid x) \ge \eta\} \in \XXA$ we have, using
$L \subseteq \YY$ and the definition of $G$,
$\pi(L \setminus E_{y}) \le \Phi(\eta,y) \le \delta/2$, hence
$\pi(L \cap E_{y}) \ge \delta/2$ and
\[
  \int_{L} s(y \mid x)\, \pi(\dd x)
  \;\ge\; \int_{L \cap E_{y}} s(y \mid x)\, \pi(\dd x)
  \;\ge\; \eta\,\pi(L \cap E_{y})
  \;\ge\; \frac{\eta\delta}{2}. \qedhere
\]
\end{proof}

\begin{remark}[What is \emph{not} claimed: there is no rate]\label{rem:notclaimed}
This is the one place where the absence of a rate is discussed; every other mention
of it in these notes --- in the Introduction, in Remark \ref{rem:uniform}, in
Remark \ref{rem:temperingtheta} and in Remark \ref{rem:whybirkhoff} --- refers back
here.
The number $\eta$ depends on the number $\delta$, and in general
$\eta \to 0$ as $\delta \downarrow 0$; consequently the constant $\gamma$ produced by
Proposition \ref{prop:minorisation} degenerates as $M \to \infty$, and the argument
yields no rate. This is not an artefact of the proof. Example \ref{ex:norate} below
exhibits a chain satisfying \hyp{D}{D} for which there is no uniform lower bound
$T_{x} \ge \eta\,\pi$, and for which the convergence of
Theorem \ref{thm:main}, although valid for every single starting point, is not
uniform in the starting point. Compare Remark \ref{rem:uniform}, where a uniform
lower bound is assumed and a geometric rate results.
\end{remark}

\begin{example}[A strictly positive density with no rate]\label{ex:norate}
Let $\XX := \N$ with $\XXA$ its power set, let
\[
  \pi(\{i\}) := 2^{-i} \quad (i \in \N),
  \qquad
  b : \N \to (0,1], \quad b(i) := 1/i,
\]
\[
  m := \sum_{i \in \N} b(i)\,\pi(\{i\}) = \ln 2 \in (0,1),
\]
and define
\[
  T(\{j\} \mid i) := \bigl(1 - b(i)\,m\bigr)\,\ind_{\{i=j\}} + b(i)\,b(j)\,\pi(\{j\})
  \qquad (i,j \in \N).
\]
Then:
\begin{enumerate}[label=(\roman*),nosep,topsep=3pt]
  \item $T$ is a Markov kernel. Indeed $b(i)m \le m < 1$, so all entries are
        nonnegative, and
        $\sum_{j} T(\{j\}\mid i) = \bigl(1-b(i)m\bigr) + b(i)\sum_{j}b(j)\pi(\{j\}) = 1$.
  \item $\pi$ is invariant, because $T$ is reversible with respect to $\pi$:
        $\pi(\{i\})\,T(\{j\}\mid i) = \bigl(1-b(i)m\bigr)\pi(\{i\})\ind_{\{i=j\}}
        + b(i)b(j)\pi(\{i\})\pi(\{j\})$ is symmetric in $(i,j)$.
  \item \hyp{D}{D} holds, with the everywhere strictly positive density
        \[
          t(j \mid i) := \frac{T(\{j\}\mid i)}{\pi(\{j\})}
          = b(i)\,b(j) + \frac{1-b(i)\,m}{\pi(\{i\})}\,\ind_{\{i=j\}} > 0 .
        \]
        All measurability requirements are vacuous on a countable space.
  \item Nevertheless
        \[
          \sup_{i \in \N}\ \tv{T^{n}_{i} - \pi} = 1
          \qquad\text{for every } n \in \N .
        \]
        Indeed $T(\{i\}\mid i) \ge 1 - b(i)m = 1 - m/i$, and the probability of the
        path that stays at $i$ for $n$ steps gives
        $T^{n}(\{i\}\mid i) \ge (1-m/i)^{n}$, whence
        \[
          \tv{T^{n}_{i} - \pi}
          \;\ge\; T^{n}(\{i\}\mid i) - \pi(\{i\})
          \;\ge\; \Bigl(1-\frac{m}{i}\Bigr)^{n} - 2^{-i} .
        \]
        For fixed $n$ the right-hand side tends to $1$ as $i \to \infty$, and
        $\tv{\cdot} \le 1$ always.
\end{enumerate}
By Theorem \ref{thm:main}, $\tv{T^{n}_{i}-\pi} \to 0$ for each fixed
$i \in \N$; by (iv) this convergence is not uniform in $i$, so no rate valid for all
starting points --- and a fortiori none valid for all initial distributions --- can
be attached to Theorem \ref{thm:main}. The mechanism is transparent: the state $i$ is
left with probability only $m/i$ per step, so the chain started at $i$ has not moved
at all, with probability bounded away from $0$, until time of order $i$.
\end{example}

\subsection{The overlap bound}\label{ssec:overlap}

Combining the two lemmas of the previous subsection gives \eqref{prop:U} for a single
value of $M$ at a time. This is the statement that the hypothesis \hyp{P}{P} of
Subsection \ref{ssec:PS} is designed to feed.

\begin{proposition}[Uniform overlap after $N$ steps]\label{prop:minorisation}
Let $M \in [1,\infty)$ and $N \in \N$ be numbers, let $\YY, \YY' \in \XXA$ be sets and
let $s : \XX \times \XX \to [0,\infty)$ be a $\XXA\otimes\XXA$-measurable map. Put
$\delta := 1/(4M)$ and assume:
\begin{enumerate}[label=(\roman*),nosep,topsep=3pt]
  \item $\pi(\XX \setminus \YY) \le \delta$ and $\pi(\XX\setminus\YY') \le \delta/4$;
  \item $T^{N}(A \mid x) \ge \displaystyle\int_{A} s(y \mid x)\,\pi(\dd y)$ for every $x \in \YY$ and every $A \in \XXA$;
  \item $\pi(\{x \in \YY : s(y\mid x) = 0\}) \le \delta/4$ for every $y \in \YY'$.
\end{enumerate}
Then there is a number $\gamma \in (0,1]$, depending only on $M, N, \YY, \YY', s$ and
$\pi$, such that
\[
  \tv{T^{N}\circ\mu - T^{N}\circ\nu} \;\le\; 1 - \gamma
  \qquad\text{for all } \mu,\nu \in \Prob \text{ with } \mu \le M\pi \text{ and } \nu \le M\pi .
\]
\end{proposition}

\begin{proof}
Recall $\delta = \dfrac{1}{4M} \in (0,\tfrac14]$ and apply Lemma \ref{lem:positivity}
with this number $\delta$ and with the given $\YY$, $\YY'$ and $s$, whose hypotheses
are (i) and (iii); this yields a number
$\eta \in (0,1]$ and a set $G \in \XXA$ with $\pi(G) \ge \tfrac12$. Define the number
\[
  \gamma := \frac{\eta\delta}{8} \in \Bigl(0,\tfrac{1}{32}\Bigr] \subseteq (0,1]
\]
--- the upper bound because $\eta \le 1$ and $\delta \le \tfrac14$ --- and the
nonnegative measure
\[
  \zeta : \XXA \to [0,\infty), \qquad \zeta(A) := \frac{\eta\delta}{4}\,\pi(A \cap G),
\]
so that $\zeta(\XX) = \frac{\eta\delta}{4}\pi(G) \ge \gamma$.

\medskip
\noindent\textit{Step 1: a dominated probability measure charges a set of substantial
$\pi$-measure.} Let $\mu \in \Prob$ with $\mu \le M\pi$. Then $\mu \ll \pi$; let
$g : \XX \to [0,\infty)$ be a density of $\mu$ with respect to $\pi$, so that
$g \le M$ $\aewhere$. Put
$L_{\mu} := \{ x \in \XX : g(x) \ge \tfrac12 \} \cap \YY \in \XXA$. Splitting the
integral,
\[
  1 = \int_{\{g \ge 1/2\}} g \dd\pi + \int_{\{g < 1/2\}} g \dd\pi
    \;\le\; M\,\pi\bigl(\{g \ge \tfrac12\}\bigr) + \tfrac12 ,
\]
whence $\pi(\{g \ge \tfrac12\}) \ge \frac{1}{2M} = 2\delta$ and therefore, by (i),
\[
  \pi(L_{\mu}) \;\ge\; 2\delta - \pi(\XX\setminus\YY) \;\ge\; 2\delta - \delta = \delta .
\]
Moreover, since $g \ge \tfrac12$ pointwise on $L_{\mu}$, the restricted measures
satisfy
\begin{equation}
  \mu(A \cap L_{\mu}) \;\ge\; \tfrac12\,\pi(A \cap L_{\mu})
  \qquad (A \in \XXA),
  \label{eq:restrictedminorant}
\end{equation}
and consequently, by \ref{conv:order} applied to the two measures
$A \mapsto \mu(A\cap L_{\mu})$ and $A \mapsto \tfrac12\pi(A \cap L_{\mu})$,
\begin{equation}
  \int_{L_{\mu}} \varphi \dd\mu \;\ge\; \tfrac12 \int_{L_{\mu}} \varphi \dd\pi
  \qquad\text{for every measurable } \varphi : \XX \to [0,\infty] .
  \label{eq:restrictedminorantint}
\end{equation}

\medskip
\noindent\textit{Step 2: $\zeta$ is a common minorant of the two images.}
Let $A \in \XXA$. Using \eqref{eq:restrictedminorantint} with
$\varphi := T^{N}(A \mid \cdot\,)$, then (ii) --- which applies because
$L_{\mu} \subseteq \YY$ --- then Tonelli's theorem, then
Lemma \ref{lem:positivity} with $L := L_{\mu}$ (legitimate since
$L_{\mu} \subseteq \YY$ and $\pi(L_{\mu}) \ge \delta$):
\begin{align*}
  (T^{N}\circ\mu)(A)
  &\;\ge\; \int_{L_{\mu}} T^{N}(A \mid x)\, \mu(\dd x)
   \;\ge\; \frac12 \int_{L_{\mu}} T^{N}(A \mid x)\, \pi(\dd x) \\[1mm]
  &\;\ge\; \frac12 \int_{L_{\mu}} \Bigl( \int_{A} s(y \mid x)\,\pi(\dd y) \Bigr) \pi(\dd x) \\[1mm]
  &\;=\; \frac12 \int_{A} \Bigl( \int_{L_{\mu}} s(y \mid x)\,\pi(\dd x) \Bigr) \pi(\dd y)\\[1mm]
  &\;\ge\; \frac12 \int_{A \cap G} \frac{\eta\delta}{2} \,\pi(\dd y)
   \;=\; \zeta(A) .
\end{align*}
The same computation applies to $\nu$, so $\zeta \le T^{N}\circ\mu$ and
$\zeta \le T^{N}\circ\nu$.

\medskip
\noindent\textit{Step 3.} Lemma \ref{lem:minorant} gives
$\tv{T^{N}\circ\mu - T^{N}\circ\nu} \le 1 - \zeta(\XX) \le 1 - \gamma$.
\end{proof}

\subsection{The Lebesgue decomposition and the singular mass}\label{ssec:singmass}\label{ssec:lebesgue}

What \eqref{prop:R} asks of a law is that it be close to a dominated one. As
Lemma \ref{lem:truncation} below will show, the only obstruction to this is the part of the law that
has no density at all. That part is named by the Lebesgue decomposition, which we
state first, and its size is measured by the quantity defined immediately afterwards.

\begin{theorem}[Lebesgue decomposition]\label{thm:lebesgue}
Let $(\XX,\XXA)$ be a measurable space and let $\mu$ and $\nu$ be $\sigma$-finite
measures on $\XXA$. \emph{In this theorem and its proof alone, $\mu$, $\nu$, $\rho$
and $\alpha_{i},\sigma_{i}$ denote $\sigma$-finite measures, not necessarily finite,
departing from Table \ref{tab:symbols}.} Then there are measures $\nu_{\acp{\mu}}$ and $\nu_{\sgp{\mu}}$ on $\XXA$
with
\[
  \nu = \nu_{\acp{\mu}} + \nu_{\sgp{\mu}},
  \qquad \nu_{\acp{\mu}} \ll \mu,
  \qquad \nu_{\sgp{\mu}}(\XX\setminus S) = 0 \ \text{ for some } S \in \XXA
  \text{ with } \mu(S) = 0 ,
\]
and they are uniquely determined by these requirements. Moreover $\nu_{\acp{\mu}}$ has a
density with respect to $\mu$, and both $\nu_{\acp{\mu}} \le \nu$ and
$\nu_{\sgp{\mu}} \le \nu$, so that both are finite whenever $\nu$ is.
\end{theorem}

\begin{proof}
\emph{Existence.} Put $\rho := \tfrac12(\mu+\nu)$, a $\sigma$-finite measure with
$\mu \ll \rho$ and $\nu \ll \rho$. By the Radon--Nikodym theorem there are measurable
maps $f,g : \XX \to [0,\infty)$ with $\mu = f\rho$ and $\nu = g\rho$; since
$\mu + \nu = 2\rho$ we have $f + g = 2$ $\rho$-almost everywhere, so the set
$\{f = 0\}\cap\{g = 0\}$ is $\rho$-null. Put
\[
  A := \{ f > 0 \}\cap\{ g = 0 \},
  \qquad
  B := \{ f = 0 \}\cap\{ g > 0 \},
  \qquad
  C := \{ f > 0 \}\cap\{ g > 0 \} ,
\]
three disjoint measurable sets whose union is $\XX$ up to a $\rho$-null set, and
define
\[
  \nu_{\acp{\mu}} := \nu\restriction C,
  \qquad
  \nu_{\sgp{\mu}} := \nu\restriction B ,
\]
where $(\nu\restriction D)(E) := \nu(E \cap D)$. Since $\nu(A) = \int_{A}g\dd\rho = 0$
and the complement of $A\cup B\cup C$ is $\rho$-null, hence $\nu$-null, we have
$\nu = \nu_{\acp{\mu}} + \nu_{\sgp{\mu}}$. The measure $\nu_{\sgp{\mu}}$ is carried by $B$, and
$\mu(B) = \int_{B} f \dd\rho = 0$, which is the singularity. For the absolute
continuity, let $E \in \XXA$ with $\mu(E) = 0$; then $\int_{E \cap C} f\dd\rho = 0$
with $f > 0$ on $C$, so $\rho(E \cap C) = 0$ and hence
$\nu_{\acp{\mu}}(E) = \int_{E\cap C} g \dd\rho = 0$. Finally the map $h$ defined by $h := g/f$ on $C$ and $h := 0$ off $C$ --- well
defined, since $f > 0$ on $C$ --- is a density of $\nu_{\acp{\mu}}$ with respect to $\mu$,
since for $E \in \XXA$
\[
  \int_{E} h \dd\mu
  = \int_{E} h\, f \dd\rho
  = \int_{E \cap C} g \dd\rho
  = \nu_{\acp{\mu}}(E) .
\]

\emph{Uniqueness.} Let $\nu = \alpha_{1}+\sigma_{1} = \alpha_{2}+\sigma_{2}$ be two
such decompositions, with $\alpha_{i} \ll \mu$ and $\sigma_{i}$ carried by a
$\mu$-null set $S_{i}$, and put $S := S_{1}\cup S_{2}$, which is $\mu$-null. Fix
$E \in \XXA$. Then $\alpha_{i}(E \cap S) = 0$ because $\mu(E\cap S) = 0$, and
$\sigma_{i}(E\setminus S) = 0$ because $\sigma_{i}$ is carried by $S_{i} \subseteq S$.
Hence
\[
  \sigma_{i}(E) = \sigma_{i}(E \cap S) = \nu(E \cap S),
  \qquad
  \alpha_{i}(E) = \alpha_{i}(E\setminus S) = \nu(E \setminus S) ,
\]
for $i = 1,2$; the right-hand sides do not depend on $i$.
\end{proof}

\noindent
The proof uses no structure on $(\XX,\XXA)$ beyond the Radon--Nikodym theorem, which
needs none either (Remark \ref{rem:nostructure}); the statement is classical, see for
instance \cite[Chapter~1]{Kallenberg21}, and it is included with its proof because the
decomposition is what the criterion of Theorem \ref{thm:crit} is about and the proof
is short. Only finite measures occur below, but both orders do: $\pi$ is decomposed
with respect to a law of the chain as often as the other way round.

The next definition measures the singular part by a variational formula instead, and
does so deliberately: with it, the proof of Theorem \ref{thm:criterion} --- the
convergence theorem in the form from which
Sections \ref{sec:first}--\ref{sec:applications} are deduced --- needs no
decomposition theorem at all. Theorem \ref{thm:lebesgue} is used in
Proposition \ref{prop:singforms}, to identify that formula with five other descriptions of
the same number, and then from
Definition \ref{def:lebesgue} onwards, where it is load-bearing: the passage from
\hyp{L}{L} to \hyp{P}{P} in Proposition \ref{prop:LimpliesP}, and with it
Theorem \ref{thm:asympequiv}, runs through the density of the absolutely continuous
part.

\begin{definition}[Singular mass]\label{def:singmass}
For finite measures $\alpha, \beta : \XXA \to [0,\infty)$ put
\[
  \sing(\alpha \mid \beta) := \alpha(\XX) - \sup\, \Sigma(\alpha\mid\beta)
  \;\in\; [0,\alpha(\XX)],
  \qquad
  \Sigma(\alpha\mid\beta) := \bigl\{\, \gamma(\XX) \ :\ \gamma \le \alpha,\ \gamma \ll \beta \,\bigr\},
\]
the supremum being over the masses of all nonnegative measures
$\gamma : \XXA \to [0,\infty)$ lying below $\alpha$ and absolutely continuous with
respect to $\beta$. The set $\Sigma(\alpha\mid\beta)$ is nonempty, since
$\gamma := 0$ is admissible.

The bar is read as in \ref{conv:kernelnotation}: the reference measure stands to the
right of it. Both orders occur below and they say different things, and the criterion
of Theorem \ref{thm:crit} asks that both vanish in the limit, which is why the
reference measure has to be displayed rather than fixed by a subscript; see
Remark \ref{rem:singmass}.
\end{definition}

\noindent
This is one of six descriptions of the same number, and it is worth having all six in
front of one at once, because different ones are convenient at different points below:
the definition just given is the only one that presupposes no decomposition theorem,
\eqref{eq:singsg} and \eqref{eq:singnull} are the two that are used most often,
\eqref{eq:singdens} is the form in which \hyp{L}{L} will be read off a density
(Remark \ref{rem:LviaJ}), and \eqref{eq:singmin} exhibits the singular mass as a
distance. Any one of them could serve as the definition.

\begin{proposition}[Six descriptions of the singular mass]\label{prop:singforms}
Assume \hyp{A1}{A1} and let $\alpha, \beta : \XXA \to [0,\infty)$ be finite measures,
with Lebesgue decomposition $\alpha = \alpha_{\acp{\beta}} + \alpha_{\sgp{\beta}}$ of
$\alpha$ with respect to $\beta$ as in Theorem \ref{thm:lebesgue}, and let
$S \in \XXA$ be a set with $\beta(S) = 0$ carrying $\alpha_{\sgp{\beta}}$. Then the
following six numbers coincide, and their common value is $\sing(\alpha\mid\beta)$.
\begin{enumerate}[label=\textup{(\roman*)},nosep,topsep=3pt,leftmargin=2.4em,itemsep=4pt]
  \item \emph{The variational form}, which is Definition \ref{def:singmass}:
        \begin{equation}
          \alpha(\XX) \;-\; \max\,\bigl\{\, \gamma(\XX) \ :\ \gamma \le \alpha,\ \gamma \ll \beta \,\bigr\} ,
          \label{eq:singvar}
        \end{equation}
        the maximum being attained at $\gamma := \alpha_{\acp{\beta}}$.
  \item \emph{The mass of the singular part}:
        \begin{equation}
          \alpha_{\sgp{\beta}}(\XX) .
          \label{eq:singsg}
        \end{equation}
  \item \emph{The mass the absolutely continuous part leaves over}:
        \begin{equation}
          \alpha(\XX) \;-\; \alpha_{\acp{\beta}}(\XX) .
          \label{eq:singac}
        \end{equation}
  \item \emph{The largest mass hidden on a $\beta$-null set}:
        \begin{equation}
          \max\,\bigl\{\, \alpha(A) \ :\ A \in \XXA,\ \beta(A) = 0 \,\bigr\} ,
          \label{eq:singnull}
        \end{equation}
        the maximum being attained at $A := S$.
  \item \emph{The null set of the reversed density}:
        \begin{equation}
          \alpha\bigl( \{ g = 0 \} \bigr) ,
          \qquad\text{$g$ any density of $\beta_{\acp{\alpha}}$ with respect to $\alpha$.}
          \label{eq:singdens}
        \end{equation}
  \item \emph{The distance to the measures that have a density}:
        \begin{equation}
          \min\,\bigl\{\, \tv{\alpha - \gamma} \ :\ \gamma : \XXA \to [0,\infty)
          \text{ a finite measure},\ \gamma \ll \beta \,\bigr\} ,
          \label{eq:singmin}
        \end{equation}
        the minimum being attained at $\gamma := \alpha_{\acp{\beta}}$; and if
        $\alpha,\beta \in \Prob$ the same minimum is attained over the smaller set of
        competitors $\gamma \in \Prob$ with $\gamma \ll \beta$.
\end{enumerate}
Moreover $\alpha_{\acp{\beta}}(S) = 0$, so that the carrier $S$ splits the mass of
$\alpha$ as
\begin{equation}
  \alpha(S) = \sing(\alpha\mid\beta)
  \qquad\text{and}\qquad
  \alpha(\XX\setminus S) = \alpha(\XX) - \sing(\alpha\mid\beta) .
  \label{eq:singcarrier}
\end{equation}
In particular $\sing(\alpha\mid\beta) = 0$ if and only if $\alpha \ll \beta$, and
$\sing(\alpha\mid\beta) = \alpha(\XX)$ if and only if $\alpha \perp \beta$, that is
if and only if $\alpha$ is carried by a $\beta$-null set.
\end{proposition}

\begin{proof}
Throughout, $\alpha_{\acp{\beta}} \ll \beta$ and $\alpha_{\sgp{\beta}}$ is carried by
$S$ with $\beta(S) = 0$; since $\alpha_{\acp{\beta}} \ll \beta$ we have
$\alpha_{\acp{\beta}}(S) = 0$, whence
\begin{equation}
  \alpha(S) = \alpha_{\sgp{\beta}}(\XX)
  \qquad\text{and}\qquad
  \alpha(\XX\setminus S) = \alpha_{\acp{\beta}}(\XX) ,
  \label{eq:singcarrieraux}
\end{equation}
which is \eqref{eq:singcarrier} once \eqref{eq:singsg} is proved.

\emph{\eqref{eq:singvar} $=$ \eqref{eq:singsg} $=$ \eqref{eq:singac}.} The measure
$\alpha_{\acp{\beta}}$ is admissible in \eqref{eq:singvar}, being below $\alpha$ and
absolutely continuous with respect to $\beta$. Conversely, if $\gamma \le \alpha$ with
$\gamma \ll \beta$, then $\gamma(S) = 0$, so
$\gamma(\XX) = \gamma(\XX\setminus S) \le \alpha(\XX\setminus S) = \alpha_{\acp{\beta}}(\XX)$
by \eqref{eq:singcarrieraux}. So the supremum in Definition \ref{def:singmass} is a
maximum, attained at $\alpha_{\acp{\beta}}$, and \eqref{eq:singvar} equals
$\alpha(\XX) - \alpha_{\acp{\beta}}(\XX)$, which is \eqref{eq:singac} and, by
additivity of the decomposition, also \eqref{eq:singsg}.

\emph{\eqref{eq:singsg} $=$ \eqref{eq:singnull}.} If $\beta(A) = 0$ then
$\alpha_{\acp{\beta}}(A) = 0$, so
$\alpha(A) = \alpha_{\sgp{\beta}}(A) \le \alpha_{\sgp{\beta}}(\XX)$; and $A := S$
attains this, by \eqref{eq:singcarrieraux}.

\emph{\eqref{eq:singsg} $=$ \eqref{eq:singdens}.} Let $g$ be a density of
$\beta_{\acp{\alpha}}$ with respect to $\alpha$, put $Z := \{g = 0\} \in \XXA$, and let
$S' \in \XXA$ carry $\beta_{\sgp{\alpha}}$ with $\alpha(S') = 0$. We show that
$\alpha\restriction(\XX\setminus Z)$ and $\alpha\restriction Z$ are the two parts of the
Lebesgue decomposition of $\alpha$ with respect to $\beta$; the uniqueness in
Theorem \ref{thm:lebesgue} then gives $\alpha_{\sgp{\beta}} = \alpha\restriction Z$ and
hence $\alpha_{\sgp{\beta}}(\XX) = \alpha(Z)$. First,
$\alpha\restriction(\XX\setminus Z) \ll \beta$: if $\beta(A) = 0$ then
$\int_{A} g \dd\alpha = \beta_{\acp{\alpha}}(A) \le \beta(A) = 0$, so the nonnegative
map $g$ vanishes $\alpha$-almost everywhere on $A$, that is
$\alpha(A\setminus Z) = 0$. Second, $\alpha\restriction Z$ is carried by a $\beta$-null
set: put $E := Z \setminus S'$; then
$\beta(E) = \int_{E} g \dd\alpha + \beta_{\sgp{\alpha}}(E) = 0 + 0 = 0$, the first term
because $g = 0$ on $Z$ and the second because $\beta_{\sgp{\alpha}}$ is carried by
$S'$, while $(\alpha\restriction Z)(\XX\setminus E) = \alpha(Z \cap S') = 0$ because
$\alpha(S') = 0$.

\emph{\eqref{eq:singsg} $=$ \eqref{eq:singmin}.} If $\gamma \ll \beta$ is a finite
measure then $\gamma(S) = 0$, so by Definition \ref{def:tv} and \eqref{eq:singcarrieraux}
\[
  \tv{\alpha-\gamma} \;\ge\; \alpha(S) - \gamma(S) \;=\; \alpha_{\sgp{\beta}}(\XX) ;
\]
and $\gamma := \alpha_{\acp{\beta}}$ is admissible, with
$\alpha - \alpha_{\acp{\beta}} = \alpha_{\sgp{\beta}} \ge 0$, so that
$\tv{\alpha-\alpha_{\acp{\beta}}} = \alpha_{\sgp{\beta}}(\XX)$.

\emph{The refinement in \eqref{eq:singmin} for laws.} Let $\alpha,\beta \in \Prob$ and
write $s := \sing(\alpha\mid\beta)$. If $s = 1$, then $\gamma := \beta$ is a
probability measure with $\gamma \ll \beta$ and $\tv{\alpha-\beta} \le 1 = s$, so the
bound just proved is attained. If $s < 1$, put
$\gamma := \alpha_{\acp{\beta}}/(1-s) \in \Prob$, which is absolutely continuous with
respect to $\beta$, and $h := \alpha - \gamma \in \Sign$. Then
\[
  h \;=\; \alpha_{\sgp{\beta}} \;-\; \frac{s}{1-s}\,\alpha_{\acp{\beta}} ,
\]
a difference of two nonnegative measures carried by the disjoint sets $S$ and
$\XX\setminus S$ respectively. So $S$ is a Hahn set for $h$ in the sense of
Definition \ref{def:tv}, with $h^{+} = \alpha_{\sgp{\beta}}$ and
$h^{-} = \tfrac{s}{1-s}\alpha_{\acp{\beta}}$; both have total mass $s$, so
$h(\XX) = 0$ and $\tv{h} = h^{+}(\XX) = s$.

\emph{The last two sentences.} Display \eqref{eq:singcarrier} is
\eqref{eq:singcarrieraux} combined with \eqref{eq:singsg} and \eqref{eq:singac}. If
$\alpha \ll \beta$ then $\gamma := \alpha$ is admissible in \eqref{eq:singvar};
conversely $\sing(\alpha\mid\beta) = 0$ forces $\alpha_{\sgp{\beta}} = 0$ by
\eqref{eq:singsg}, that is $\alpha = \alpha_{\acp{\beta}} \ll \beta$. Finally
$\sing(\alpha\mid\beta) = \alpha(\XX)$ says $\alpha_{\acp{\beta}} = 0$ by
\eqref{eq:singac}, that is $\alpha = \alpha_{\sgp{\beta}}$, which is carried by the
$\beta$-null set $S$; and conversely, if $\alpha$ is carried by a set $N$ with
$\beta(N) = 0$, then $A := N$ is admissible in \eqref{eq:singnull} and gives
$\sing(\alpha\mid\beta) \ge \alpha(N) = \alpha(\XX)$.
\end{proof}

\begin{remark}[How to read the two orders]\label{rem:singmass}
Form \eqref{eq:singnull} is the one to keep in mind when reading the two hypotheses of
Theorem \ref{thm:crit}, and it is worth spelling out in both orders. In words,
$\sing(\alpha\mid\pi)$ is the largest amount of mass that $\alpha$ can hide on a set
the target ignores; read with the two arguments exchanged, $\sing(\pi\mid\alpha)$ is
the largest amount of $\pi$-mass that can hide on a set $\alpha$ ignores --- the part
of the target the chain has not yet learned to see. Neither number is symmetric in its
two arguments, and neither dominates the other: the example of
Remark \ref{rem:threepieces} has $\sing(\alpha\mid\pi) = \tfrac12$ and
$\sing(\pi\mid\alpha) = 0$, and exchanging $\alpha$ and $\pi$ there exchanges the two
values.

Proposition \ref{prop:singforms} is stated for arbitrary finite measures, and is
applied below with several different reference measures --- with $\pi$ in
Definition \ref{def:lebesgue}, with a law of the chain in
Lemma \ref{lem:monotone}(v), and with an auxiliary $\rho$ in
Corollary \ref{cor:singtv}. Nothing about $\pi$ beyond finiteness is used.
\end{remark}

\noindent
The necessity half of the criterion is one line, and it comes with an identity that
says exactly what the bound discards. Recall from
Definition \ref{def:tv} that $h^{+}$ and $h^{-}$ denote the two parts of the Jordan
decomposition of $h \in \Sign$, and that $\tv{h} = h^{+}(\XX) = h^{-}(\XX)$ whenever
$h(\XX) = 0$, as it is for a difference of two probability measures.

\begin{corollary}[Convergence forces both singular masses to vanish]\label{cor:singtv}
Assume \hyp{A1}{A1} and let $\alpha, \rho \in \Prob$. Then the distance splits, in
either order, as
\begin{equation}
  \tv{\alpha - \rho}
  \;=\; \sing(\alpha \mid \rho) + \bigl( \alpha_{\acp{\rho}} - \rho \bigr)^{+}(\XX)
  \;=\; \sing(\rho \mid \alpha) + \bigl( \rho_{\acp{\alpha}} - \alpha \bigr)^{+}(\XX) ,
  \label{eq:singsplit}
\end{equation}
and consequently
\[
  \sing(\alpha \mid \rho) \;\le\; \tv{\alpha - \rho}
  \qquad\text{and}\qquad
  \sing(\rho \mid \alpha) \;\le\; \tv{\alpha - \rho} .
\]
In particular, under \hyp{A}{A}, if $\tv{T^{n}_{x} - \pi} \to 0$ for some $x \in \XX$,
then $\sing(T^{n}_{x} \mid \pi) \to 0$ and $\sing(\pi \mid T^{n}_{x}) \to 0$ for that
same $x$.
\end{corollary}

\begin{proof}
Put $h := \alpha - \rho$; then $h(\XX) = 0$, so $\tv{h} = h^{+}(\XX)$. Let
$S \in \XXA$ be a set with $\rho(S) = 0$ carrying $\alpha_{\sgp{\rho}}$, and split
\[
  h \;=\; \alpha_{\sgp{\rho}} \;+\; \bigl( \alpha_{\acp{\rho}} - \rho \bigr) .
\]
The first summand is nonnegative and carried by $S$; the second vanishes on every
measurable subset of $S$, since $\rho(S) = 0$ and $\alpha_{\acp{\rho}} \ll \rho$. The
two are therefore carried by the disjoint sets $S$ and $\XX\setminus S$, so the Jordan
parts of $h$ split over them:
\[
  h^{+} \;=\; \alpha_{\sgp{\rho}} \;+\; \bigl( \alpha_{\acp{\rho}} - \rho \bigr)^{+} .
\]
Taking total masses and using \eqref{eq:singsg} gives the first form of
\eqref{eq:singsplit}; exchanging the roles of $\alpha$ and $\rho$ gives the second,
since $\tv{\rho-\alpha} = \tv{\alpha-\rho}$. The two bounds follow because the
remaining term is nonnegative.
\end{proof}

\begin{remark}[What the splitting says, and what the bound costs]\label{rem:threepieces}
Identity \eqref{eq:singsplit} separates the two ways in which $\alpha$ can carry mass
that $\rho$ does not: it can put mass where $\rho$ is blind, which is the first term,
and on the region $\rho$ does see it can put down too much, which is the second. The
criterion of Theorem \ref{thm:asympequiv} kills the first term in both orders and says
nothing at all about the second; the entire content of its sufficiency half is that,
under \hyp{A}{A} and \hyp{J}{J}, killing the first forces the second to vanish as well.
That is why the necessity half is one line and the sufficiency half needs the whole
apparatus of Section \ref{sec:core}.

Both bounds are sharp, and \eqref{eq:singsplit} says by how much each fails to be an
equality. For $\alpha := \tfrac12(\delta_{0}+\lambda)$ and $\rho := \lambda$ the
uniform law on $[0,1]$ the first is an equality and the second is not:
$\sing(\alpha\mid\rho) = \tfrac12 = \tv{\alpha-\rho}$, because
$\alpha_{\acp{\rho}} = \tfrac12\lambda \le \rho$ leaves no second term, whereas
$\sing(\rho\mid\alpha) = 0$, the whole distance sitting in
$(\rho_{\acp{\alpha}}-\alpha)^{+}(\XX) = \tfrac12$. Exchanging $\alpha$ and $\rho$
exchanges the two readings, which is the sense in which neither singular mass dominates
the other (Remark \ref{rem:singmass}).
\end{remark}

\begin{remark}[The reading of \hyp{S}{S} that the name records]\label{rem:singdist}
Write $\Probac := \{\beta \in \Prob : \beta \ll \pi\}$ for the set of laws possessing a
density with respect to $\pi$. The refinement in \eqref{eq:singmin} says that, for
$\alpha \in \Prob$, the number $\sing(\alpha\mid\pi)$ is the distance from $\alpha$ to
$\Probac$ in the total
variation metric of
Definition \ref{def:tv}; so the hypothesis \hyp{S}{S} of
Subsection \ref{ssec:PS}, namely $\sing(T^{n}_{x}\mid\pi) \to 0$ for every $x$, says exactly
that
\[
  \min_{\beta \in \Probac} \tv{T^{n}_{x} - \beta} \;\longrightarrow\; 0
  \qquad\text{as } n \to \infty, \text{ for every } x \in \XX .
\]
The law of the chain converges to the \emph{set} of laws with a density, not to any
particular one; it need not itself have a density at any time, and
Example \ref{ex:UnotR} and Remark \ref{rem:pcn} exhibit chains for which
$T^{n}_{x,\sgp{\pi}} \ne 0$ for every $n$. This is what the phrase \emph{asymptotic
absolute continuity} is meant to record, and by \eqref{eq:singmin} it records it
exactly rather than by analogy.
\end{remark}

\begin{lemma}[Propagation and monotonicity of the two singular masses]\label{lem:monotone}
Assume \hyp{A}{A}, let $\alpha, \sigma : \XXA \to [0,\infty)$ be finite measures and
let $\mu, \nu \in \Prob$.
\begin{enumerate}[label=\textup{(\roman*)},nosep,topsep=3pt,leftmargin=2.4em,itemsep=3pt]
  \item \emph{Absolute continuity propagates forwards:} if $\nu \ll \pi$ then
        $T\circ\nu \ll \pi$.
  \item \emph{So does domination of the target:} if $\pi \ll \nu$ then
        $\pi \ll T\circ\nu$. Hence $T\circ\nu \sim \pi$ whenever $\nu \sim \pi$, and,
        for every $x \in \XX$ and every $m \in \N$, each of the relations
        $T^{m}_{x} \ll \pi$, $\pi \ll T^{m}_{x}$ and $T^{m}_{x} \sim \pi$, once it
        holds, holds at every later time $n \ge m$.
  \item \emph{The singular mass is monotone in the measure:} if $\sigma \le \alpha$
        then $\sing(\sigma\mid\pi) \le \sing(\alpha\mid\pi)$.
  \item \emph{It does not increase along the dynamics:}
        $\sing(T\circ\alpha\mid\pi) \le \sing(\alpha\mid\pi)$. Consequently the map
        $\Nz \to [0,1]$, $n \mapsto \sing(T^{n}\circ\mu\mid\pi)$, is non-increasing.
  \item \emph{Nor does the reversed singular mass:}
        $\sing(\pi \mid T\circ\nu) \le \sing(\pi\mid\nu)$. Consequently the map
        $\Nz \to [0,1]$, $n \mapsto \sing(\pi\mid T^{n}_{x})$, is non-increasing, for
        every $x \in \XX$.
\end{enumerate}
Part (i), and the first assertion of (ii), are the vanishing cases of (iv) and (v), by
the last sentence of Proposition \ref{prop:singforms}; they are proved first because
(iv) uses (i).
\end{lemma}

\begin{proof}
\emph{(i).} Let $A \in \XXA$ with $\pi(A) = 0$. Invariance gives
$0 = \pi(A) = \int T(A \mid x)\,\pi(\dd x)$, and the integrand is nonnegative, so
$T(A \mid \cdot\,) = 0$ $\aewhere$, hence $\nu$-almost everywhere because
$\nu \ll \pi$. Therefore $(T\circ\nu)(A) = \int T(A\mid x)\,\nu(\dd x) = 0$.

\emph{(ii).} Let $A \in \XXA$ with $(T\circ\nu)(A) = 0$. Since
$T(A\mid\cdot\,) \ge 0$, this forces $T(A\mid\cdot\,) = 0$ $\nu$-almost everywhere,
hence $\aewhere$ because $\pi \ll \nu$. Therefore
$\pi(A) = (T\circ\pi)(A) = \int T(A\mid x)\,\pi(\dd x) = 0$. The statement about $\sim$
follows by combining this with (i), and the propagation statements by iteration, using
$T \circ T^{n}_{x} = T^{n+1}_{x}$.

\emph{(iii).} Let $S \in \XXA$ be a set with $\pi(S) = 0$ carrying $\alpha_{\sgp{\pi}}$
and put $\beta(A) := \sigma(A\setminus S)$. Then $\beta \le \sigma$, and
$\beta \ll \pi$: if $\pi(A) = 0$ then $\alpha_{\acp{\pi}}(A) = 0$, so
$\alpha(A\setminus S) = \alpha_{\acp{\pi}}(A\setminus S) + \alpha_{\sgp{\pi}}(A\setminus S) = 0$,
the second term because $\alpha_{\sgp{\pi}}$ is carried by $S$, and hence
$\beta(A) \le \alpha(A\setminus S) = 0$. So $\beta$ is admissible in
Definition \ref{def:singmass} and, using \eqref{eq:singcarrier} for $\alpha$,
\[
  \sing(\sigma\mid\pi) \;\le\; \sigma(\XX) - \beta(\XX) \;=\; \sigma(S)
  \;\le\; \alpha(S) \;=\; \sing(\alpha\mid\pi) .
\]

\emph{(iv).} Let $\beta \le \alpha$ with $\beta \ll \pi$; we may assume $\beta \ne 0$.
Then $T\circ\beta \le T\circ\alpha$ by \ref{conv:order}, since
$(T\circ\beta)(A) = \int T(A\mid x)\beta(\dd x) \le \int T(A \mid x)\alpha(\dd x)$ for
every $A \in \XXA$; and $T \circ \beta \ll \pi$ by (i) applied to
$\beta/\beta(\XX) \in \Prob$. Since $(T\circ\beta)(\XX) = \beta(\XX)$, every
number admissible in the supremum for $\alpha$ is admissible in the supremum for
$T\circ\alpha$; as moreover $(T\circ\alpha)(\XX) = \alpha(\XX)$, the claim follows by
taking suprema. The second statement follows by iteration, using
$T \circ (T^{n}\circ\mu) = T^{n+1}\circ\mu$.

\emph{(v).} Let $\pi = \alpha' + \sigma'$ be the Lebesgue decomposition of $\pi$ with
respect to $\nu$, so that $\alpha' \ll \nu$, the measure $\sigma'$ is carried by a
$\nu$-null set, and $\alpha'(\XX) = 1 - s$ with $s := \sing(\pi\mid\nu)$ by
\eqref{eq:singac}. Three observations about $T\circ\alpha'$:
\begin{enumerate}[label=\textup{(\alph*)},nosep,topsep=3pt,leftmargin=2.4em]
  \item $T\circ\alpha' \le T\circ\pi = \pi$, by \ref{conv:order} applied to
        $\alpha' \le \pi$ and by the invariance \hyp{A3}{A3};
  \item $T\circ\alpha' \ll T\circ\nu$: if $(T\circ\nu)(A) = 0$ then the nonnegative map
        $T(A\mid\cdot\,)$ vanishes $\nu$-almost everywhere, hence $\alpha'$-almost
        everywhere because $\alpha' \ll \nu$, and so $(T\circ\alpha')(A) = 0$;
  \item $(T\circ\alpha')(\XX) = \alpha'(\XX) = 1-s$.
\end{enumerate}
So $T\circ\alpha'$ is admissible in the supremum $\Sigma(\pi \mid T\circ\nu)$ of
Definition \ref{def:singmass}, whence
$\sing(\pi\mid T\circ\nu) \le \pi(\XX) - (1-s) = s$. The second statement follows by
iteration, using $T\circ T^{n}_{x} = T^{n+1}_{x}$.
\end{proof}

\noindent
Part (ii) is what allows a hypothesis with a step number depending on the starting
point to be used for a \emph{pair} of starting points at once: given $x$ and $y$ with
$\pi \ll T^{m_{x}}_{x}$ and $\pi \ll T^{m_{y}}_{y}$, the single number
$m := \max\{m_{x},m_{y}\}$ serves both; this is how Section \ref{sec:general} uses it,
in the form of Corollary \ref{cor:selfimprovement}. That no aperiodicity hypothesis is
needed anywhere in these notes has a more basic reason, given after
Assumption \ref{ass:L}: a chain of period $d \ge 2$ violates \hyp{L}{L} outright.

\noindent
In the notation of \ref{conv:order} the hypothesis $\pi \ll \nu$ of (ii) reads
$\nu \gg \pi$: the measure $\nu$ \emph{dominates} $\pi$, and this is strictly weaker
than $\nu \sim \pi$, which would in addition require $\nu \ll \pi$, that is
$\sing(\nu\mid\pi) = 0$. The two hypotheses of Subsection \ref{ssec:L} are exactly
these two halves, the second only in the limit.

\noindent
So far the Lebesgue decomposition has served only to name a quantity that is then
handled variationally. From here on the density of the absolutely continuous part is
the object of interest itself, and it is given a symbol.

\begin{definition}[The density of the absolutely continuous part]\label{def:lebesgue}
Let $\alpha : \XXA \to [0,\infty)$ be a finite measure and let
$\alpha = \alpha_{\acp{\pi}} + \alpha_{\sgp{\pi}}$ be its Lebesgue decomposition with
respect to $\pi$, as in Theorem \ref{thm:lebesgue} applied with $\mu := \pi$ and
$\nu := \alpha$, with $\alpha_{\sgp{\pi}}$ carried by a set $S \in \XXA$ with
$\pi(S) = 0$; both summands are finite, and
$\sing(\alpha\mid\pi) = \alpha_{\sgp{\pi}}(\XX)$ by \eqref{eq:singsg}. By the
Radon--Nikodym theorem we may fix a map
\[
  g_{\alpha} : \XX \longrightarrow [0,\infty),
  \qquad
  \alpha_{\acp{\pi}}(A) = \int_{A} g_{\alpha} \dd\pi \quad (A \in \XXA),
\]
finite at every point, and call it \emph{the density of the absolutely continuous part
of $\alpha$}; by \ref{conv:representative} it is a genuine map, and every statement
below about it is invariant under changing it on a $\pi$-null set.

Read with the pair $(\pi,\alpha)$ in place of $(\alpha,\beta)$, form
\eqref{eq:singdens} of Proposition \ref{prop:singforms} expresses the \emph{reversed}
singular mass through this same density:
\begin{equation}
  \sing(\pi\mid\alpha) \;=\; \pi\bigl( \{ g_{\alpha} = 0 \} \bigr) ,
  \label{eq:acpositive}
\end{equation}
so that $\pi \ll \alpha$ if and only if $g_{\alpha} > 0$ $\aewhere$, by the last
sentence of Proposition \ref{prop:singforms}. The set on which
the density of the absolutely continuous part still vanishes is thus exactly the part
of the target that $\alpha$ does not see.
\end{definition}

\begin{remark}[Why the singular mass is not asserted to be measurable in the starting point]\label{rem:singnotmeasurable}
For fixed $n$ the map $x \mapsto \sing(T^{n}_{x}\mid\pi)$ is a supremum, over the
uncountably many measures $\beta \le T^{n}_{x}$ with $\beta \ll \pi$, of quantities
each of which depends on $x$; nothing in Remark \ref{rem:nostructure} makes it
measurable, and the situation is the same as for $x \mapsto \tv{T^{n}_{x}-\pi}$
(Remark \ref{rem:tvnotmeasurable}). This is why \hyp{S}{S}, when it is stated in
Subsection \ref{ssec:PS}, will be a condition on that number for each
$x \in \XX$ separately, and never an integral of it against a law. Nothing
in these notes needs more, and the passage from the
points to an arbitrary initial law is carried out by Lemma \ref{lem:pointwise},
which fixes one Hahn set per time step precisely in order to avoid integrating a map
that has not been shown to be measurable.
\end{remark}

\subsection{The two assumptions \texorpdfstring{\hyp{P}{P} and \hyp{S}{S}}{(P) and (S)}}
\label{ssec:PS}

Theorem \ref{thm:abstract} is not yet in a form that can be checked on a given
kernel: \eqref{prop:U} quantifies over all pairs of dominated laws, and
\eqref{prop:R} asks for a dominated approximant without saying where to find one. We
now replace the two by hypotheses that are statements about $T$ and $\pi$ alone, and
that are what Sections \ref{sec:first}--\ref{sec:applications} actually verify. Like
\eqref{prop:R}, the second of them is a condition at each point of $\XX$ separately.

\begin{assumption}[An eventual positive minorant transition density]\label{ass:P}
In addition to \hyp{A}{A}:
\begin{enumerate}[nosep,topsep=3pt,leftmargin=3em]
  \item[\hyptgt{P}{P}] For every number $\varepsilon \in (0,1)$ there are a number
        $N \in \N$, sets $\YY, \YY' \in \XXA$ and a $\XXA\otimes\XXA$-measurable map
        \[
          s : \XX \times \XX \longrightarrow [0,\infty), \qquad (x,y) \longmapsto s(y \mid x),
        \]
        all four of which may depend on $\varepsilon$, such that
        \begin{enumerate}[label=(\alph*),nosep,topsep=2pt,leftmargin=2em]
          \item $\pi(\XX \setminus \YY) \le \varepsilon$ and $\pi(\XX\setminus\YY') \le \varepsilon$;
          \item $T^{N}(A \mid x) \ge \displaystyle\int_{A} s(y \mid x)\,\pi(\dd y)$ for every $x \in \YY$ and every $A \in \XXA$;
          \item $\pi(\{ x \in \YY : s(y \mid x) = 0 \}) \le \varepsilon$ for every $y \in \YY'$.
        \end{enumerate}
\end{enumerate}
\end{assumption}

\noindent
Hypotheses (a)--(c) are the hypotheses (i)--(iii) of
Proposition \ref{prop:minorisation}, with the single tolerance $\varepsilon$ in place
of the two thresholds $\delta$ and $\delta/4$ used there;
\hyp{P}{P} says nothing more than that those hypotheses can be met for an arbitrarily
small tolerance. Only an inequality is required in (b).

The tolerance $\varepsilon$ has a double role, and it is worth naming it now. In (a)
it is the $\pi$-measure of the set of starting points that one is allowed to discard;
and it will be instantiated as $\varepsilon := 1/(16M)$ in the proof of
Lemma \ref{lem:PimpliesU}, where $M$ is the domination constant for which
\eqref{prop:U} is being verified. So ``arbitrarily small $\varepsilon$'' is the same
demand as ``arbitrarily large $M$'': the more concentrated the initial laws are
allowed to be, the smaller the exceptional set that \hyp{P}{P} must tolerate.

\begin{lemma}[\hyp{P}{P} implies \eqref{prop:U}]\label{lem:PimpliesU}
Assume \hyp{A}{A} and \hyp{P}{P}. Then \eqref{prop:U} holds.
\end{lemma}

\begin{proof}
Let $M \in [1,\infty)$ be a number and put $\delta := 1/(4M)$. Apply \hyp{P}{P} with
the number $\varepsilon := \delta/4 = 1/(16M) \in (0,1)$, obtaining $N$, $\YY$, $\YY'$
and $s$, and then Proposition \ref{prop:minorisation} with these $M$, $N$, $\YY$,
$\YY'$ and $s$: its hypotheses (i), (ii) and (iii) are \hyp{P}{P}(a), (b) and (c),
since $\varepsilon = \delta/4 \le \delta$. This yields a number $\gamma \in (0,1]$ as
required.
\end{proof}

\begin{assumption}[Asymptotic absolute continuity]\label{ass:S}
In addition to \hyp{A}{A}:
\begin{enumerate}[nosep,topsep=3pt,leftmargin=3em]
  \item[\hyptgt{S}{S}] $\displaystyle\lim_{n\to\infty} \sing\bigl( T^{n}_{x} \bigm| \pi\bigr) = 0$ for every $x \in \XX$.
\end{enumerate}
\end{assumption}

\noindent
In words: from every starting point, after enough steps, all but an arbitrarily small
proportion of the mass of the law of the chain is absolutely continuous with respect
to $\pi$. By Lemma \ref{lem:monotone}(iv) applied to $\delta_{x}$ the limit exists
in $[0,1]$ for each $x$, so \hyp{S}{S} asserts only that it vanishes. Like
\eqref{prop:R}, of which it is the exact counterpart, it is a condition at each point
separately: no measurability of $x \mapsto \sing(T^{n}_{x}\mid\pi)$ is asserted,
and none is used (Remark \ref{rem:singnotmeasurable}). That the same statement for an
arbitrary initial law $\mu \in \Prob$ then holds a posteriori is
Remark \ref{rem:aposteriori}.

\noindent
One fact is still needed, and it is the only place where a density is truncated: a law
with a density can be replaced, at the cost of an arbitrarily small error, by one
dominated by a multiple of $\pi$. This is what turns the qualitative statement
``$\nu \ll \pi$'' into the quantitative $\rho \le M\pi$ that \eqref{prop:R} asks for.

\begin{lemma}[Truncation]\label{lem:truncation}
Let $\nu \in \Prob$ with $\nu \ll \pi$ and let $\varepsilon \in (0,1)$ be a number.
Then there exist a number $M \in [1,\infty)$ and a probability measure
$\rho \in \Prob$ with
\[
  \rho \le M\pi \qquad\text{and}\qquad \tv{\nu - \rho} \le \varepsilon .
\]
\end{lemma}

\begin{proof}
Let $f : \XX \to [0,\infty]$ be a density of $\nu$ with respect to $\pi$, so that
$f \in \Dens$. Consider the map $c_{\bullet} : [1,\infty) \to [0,1]$,
$c_{M_{0}} := \int_{\XX} \min(f,M_{0})\dd\pi$, in which $M_{0}$ is a
\emph{truncation level} and not the domination constant asserted in the statement.
If $c_{M_{0}} = 0$ for some $M_{0}$, then $f = 0$ $\aewhere$, contradicting
$\int f \dd\pi = 1$; so $c_{M_{0}} \in (0,1]$ for every $M_{0} \in [1,\infty)$, and
$c_{M_{0}} \to 1$ as $M_{0} \to \infty$ by monotone convergence.
Choose a number $M_{0} \in [1,\infty)$ with $c_{M_{0}} \ge 1-\varepsilon$ and set
\[
  \tilde f := \frac{\min(f,M_{0})}{c_{M_{0}}} \in \Dens,
  \qquad
  \rho(A) := \int_{A} \tilde f \dd\pi \quad (A \in \XXA).
\]
Then $\tilde f \le M := M_{0}/c_{M_{0}}$ everywhere, i.e.\ $\rho \le M\pi$ with the
number $M \in [1,\infty)$ --- this is the number whose existence is asserted --- and
\begin{align*}
  \Lone{f - \tilde f}
  &\le \Lone{f - \min(f,M_{0})} + \Lone{\min(f,M_{0}) - \tilde f}\\
  &= (1-c_{M_{0}}) + c_{M_{0}}\Bigl( \frac{1}{c_{M_{0}}} - 1 \Bigr)
   = 2\,(1-c_{M_{0}}),
\end{align*}
so that $\tv{\nu - \rho} = \tfrac12\Lone{f-\tilde f} \le 1 - c_{M_{0}} \le \varepsilon$
by Lemma \ref{lem:tvformulas} with $\lambda := \pi$.
\end{proof}

\begin{lemma}[\hyp{S}{S} is equivalent to \eqref{prop:R}]\label{lem:SiffR}
Assume \hyp{A}{A}. Then \hyp{S}{S} holds if and only if \eqref{prop:R} holds. More
precisely, for each fixed $\mu \in \Prob$ the condition
$\lim_{n}\sing(T^{n}\circ\mu\mid\pi) = 0$ is equivalent to the conclusion of
\eqref{prop:R} with $T^{n_{1}}\circ\mu$ in place of $T^{n_{1}}_{x}$; the
two hypotheses are the instances of this equivalence at the Dirac measures.
\end{lemma}

\begin{proof}
Both implications are proved for one initial law at a time, and \hyp{S}{S} and
\eqref{prop:R} are their instances at $\mu := \delta_{x}$, for which
$T^{n}\circ\delta_{x} = T^{n}_{x}$.

\emph{\hyp{S}{S} implies \eqref{prop:R}.} Let $\mu \in \Prob$ and let
$\varepsilon \in (0,1)$ be a number. Choose a number $n_{1} \in \Nz$ with
$\sing(T^{n_{1}}\circ\mu\mid\pi) \le \varepsilon/4$ and then, by
Definition \ref{def:singmass}, a nonnegative measure $\beta \le T^{n_{1}}\circ\mu$
with $\beta \ll \pi$ and
\[
  \beta(\XX) \;\ge\; 1 - \sing(T^{n_{1}}\circ\mu\mid\pi) - \varepsilon/4 \;\ge\; 1-\varepsilon/2 \;>\; 0 .
\]
Put $\hat\beta := \beta/\beta(\XX) \in \Prob$, so that $\hat\beta \ll \pi$. For every
$A \in \XXA$,
\[
  (T^{n_{1}}\circ\mu)(A) - \hat\beta(A)
  = \bigl( T^{n_{1}}\circ\mu - \beta \bigr)(A) - \bigl( \hat\beta - \beta \bigr)(A)
  \;\in\; \bigl[ -\varepsilon/2,\ \varepsilon/2 \bigr],
\]
because the two subtracted terms are nonnegative measures of total mass
$1 - \beta(\XX) \le \varepsilon/2$ each. Hence
$\tv{T^{n_{1}}\circ\mu - \hat\beta} \le \varepsilon/2$. Applying
Lemma \ref{lem:truncation} to $\hat\beta$ with the number $\varepsilon/2$ yields a
number $M \in [1,\infty)$ and a probability measure $\rho \le M\pi$ with
$\tv{\hat\beta - \rho} \le \varepsilon/2$, and the triangle inequality gives
$\tv{T^{n_{1}}\circ\mu - \rho} \le \varepsilon$.

\emph{\eqref{prop:R} implies \hyp{S}{S}.} Let $\mu \in \Prob$ and let
$\varepsilon \in (0,1)$ be a number, and choose $n_{1}$, $M$ and $\rho \le M\pi$ with
$\tv{T^{n_{1}}\circ\mu - \rho} \le \varepsilon$. Put
$h := T^{n_{1}}\circ\mu - \rho \in \Sign$, let $E \in \XXA$ be a Hahn set for $h$ in
the sense of Definition \ref{def:tv}, and put $\beta := T^{n_{1}}\circ\mu - h^{+}$.
Then $\beta$ is a nonnegative measure, since for $A \in \XXA$
\[
  \beta(A) = (T^{n_{1}}\circ\mu)(A) - h(A \cap E)
           = (T^{n_{1}}\circ\mu)(A \setminus E) + \rho(A \cap E) \ \ge\ 0 ;
\]
moreover $\beta \le T^{n_{1}}\circ\mu$ because $h^{+} \ge 0$, and
$\beta = \rho - h^{-} \le \rho \le M\pi$, so $\beta \ll \pi$. Since $h(\XX) = 0$ we
have $h^{+}(\XX) = \tv{h}$ by Definition \ref{def:tv}, whence
$\beta(\XX) = 1 - \tv{h} \ge 1-\varepsilon$ and therefore
$\sing(T^{n_{1}}\circ\mu\mid\pi) \le \varepsilon$ by Definition \ref{def:singmass}. By
Lemma \ref{lem:monotone}(iv), $\sing(T^{n}\circ\mu\mid\pi) \le \varepsilon$ for every
$n \ge n_{1}$. As $\varepsilon \in (0,1)$ was arbitrary, the limit is $0$.
\end{proof}

\begin{remark}[Where the monotonicity is needed]\label{rem:whymonotone}
Property \eqref{prop:R} asserts the existence of \emph{one} good time $n_{1}$,
whereas \hyp{S}{S} asserts a limit. The passage from the first to the second is
exactly Lemma \ref{lem:monotone}(iv), and without it the two would not be
equivalent. It is also why \hyp{S}{S} may always be verified at a single convenient
time, as it is in each of Sections \ref{sec:first}--\ref{sec:applications}.
\end{remark}

\begin{remark}[\hyp{S}{S} is contiguity to $\pi$]\label{rem:contiguity}
Le Cam's \emph{contiguity} \cite{LeCam60} is the asymptotic form of absolute
continuity: for sequences $(P_{n})$ and $(Q_{n})$ of probability measures on
$(\XX,\XXA)$ one writes $Q_{n} \contig P_{n}$ when
\[
  P_{n}(A_{n}) \to 0
  \quad\Longrightarrow\quad
  Q_{n}(A_{n}) \to 0
  \qquad\text{for every sequence } (A_{n}) \text{ in } \XXA ;
\]
see \cite[Chapter~6]{vanderVaart98}. Assumption \hyp{S}{S} is exactly this relation
between the law of the chain and the constant sequence $\pi$:
\[
  \hyp{S}{S}
  \qquad\Longleftrightarrow\qquad
  \bigl( T^{n}_{x} \bigr)_{n \in \Nz} \contig \pi
  \quad\text{for every } x \in \XX .
\]
For the implication from left to right, let $\varepsilon \in (0,1)$ and let $(A_{n})$
satisfy $\pi(A_{n}) \to 0$. By Lemma \ref{lem:SiffR} there are $n_{1}$, $M$ and
$\rho \le M\pi$ with $\tv{T^{n_{1}}_{x} - \rho} \le \varepsilon$; for $n \ge n_{1}$,
Lemma \ref{lem:contraction} gives
$\tv{T^{n}_{x} - T^{n-n_{1}}\circ\rho} \le \varepsilon$ and
Lemma \ref{lem:domination} gives $T^{n-n_{1}}\circ\rho \le M\pi$, whence
$T^{n}_{x}(A_{n}) \le M\pi(A_{n}) + \varepsilon$ and
$\limsup_{n}T^{n}_{x}(A_{n}) \le \varepsilon$. For the converse, suppose \hyp{S}{S}
fails at $x$; by Lemma \ref{lem:monotone}(iv) the limit
$\delta := \lim_{n}\sing(T^{n}_{x}\mid\pi)$ exists and is positive, and taking for $A_{n}$ a
carrier of $T^{n}_{x,\sgp{\pi}}$ as in \eqref{eq:singnull} gives $\pi(A_{n}) = 0$
while $T^{n}_{x}(A_{n}) \ge \delta$ for every $n$.

Two remarks on the connection. It is a genuine instance of Le Cam's notion, but a
degenerate one: one of the two sequences is constant, which is the case his theory
does not need and ours does. And it is not used below --- \hyp{S}{S} is applied only
through Lemma \ref{lem:SiffR} --- so the equivalence is recorded for orientation, and
because it locates \hyp{S}{S} among notions already in use.

The reversed relation $\pi \contig (T^{n}_{x})$ is \emph{not} \hyp{L}{L}: it is
strictly stronger. Contiguity in that direction would require
$T^{n}_{x}(A_{n}) \to 0 \Rightarrow \pi(A_{n}) \to 0$, and the absolute continuity of
the surviving part of $\pi$ with respect to $T^{n}_{x}$ carries no uniform bound on
the corresponding density, so nothing prevents $T^{n}_{x}$ from putting arbitrarily
little mass on a set that $\pi$ charges. What \hyp{L}{L} is, by
\eqref{eq:singnull}, is the null-set form of the same idea:
\[
  \pi(A_{n}) \;\longrightarrow\; 0
  \qquad\text{for every sequence } (A_{n}) \text{ in } \XXA
  \text{ with } T^{n}(A_{n}\mid x) = 0 .
\]
So the two hypotheses are the same relation read in the two directions, but the
directions are not equally strong: \hyp{S}{S} is contiguity, \hyp{L}{L} only its
null-set shadow.
\end{remark}

\subsection{The convergence theorem under \texorpdfstring{\hyp{A}{A}, \hyp{P}{P} and \hyp{S}{S}}{(A), (P) and (S)}}

\begin{theorem}[The convergence theorem]\label{thm:criterion}
Assume \hyp{A}{A}, \hyp{P}{P} and \hyp{S}{S}. Then
\[
  \lim_{n\to\infty}\ \sup_{A \in \XXA} \bigl| (T^{n}\circ\mu)(A) - \pi(A) \bigr| = 0
  \qquad\text{for every } \mu \in \Prob,
\]
and $\pi$ is the unique invariant probability measure of $T$.
\end{theorem}

\begin{proof}
Property \eqref{prop:U} holds by Lemma \ref{lem:PimpliesU} and \eqref{prop:R} holds
by Lemma \ref{lem:SiffR}. Theorem \ref{thm:abstract} applies.
\end{proof}

\begin{remark}[The two hypotheses hold for every initial law a posteriori]\label{rem:aposteriori}
Both \eqref{prop:R} and \hyp{S}{S} are imposed only at the points of $\XX$, and both
then hold for every initial law --- already under \hyp{A}{A}, \eqref{prop:U} and
\eqref{prop:R}, hence in particular here. Indeed, Corollary \ref{cor:singtv} applied
with $\alpha := T^{n}\circ\mu$ and $\rho := \pi$ gives
\[
  \sing\bigl( T^{n}\circ\mu \bigm| \pi\bigr) \;\le\; \tv{T^{n}\circ\mu - \pi}
  \qquad (\mu \in \Prob,\ n \in \Nz) ,
\]
whose right-hand side tends to $0$ by Theorem \ref{thm:abstract}. Hence
$\sing(T^{n}\circ\mu\mid\pi) \to 0$ for every $\mu \in \Prob$, and \eqref{prop:R} for every
$\mu$ follows by Lemma \ref{lem:SiffR}. So the pointwise and the global forms of the
two hypotheses are equivalent in the presence of \hyp{P}{P}, both being equivalent to
the conclusion; the pointwise form is the weaker hypothesis and the one that is
actually checked, which is why it is the one assumed.
\end{remark}

\noindent
Everything after this point is the verification of \hyp{P}{P} and \hyp{S}{S} in
particular settings. The following criterion disposes of \hyp{S}{S} whenever the
minorant of \hyp{P}{P} happens to be available at \emph{every} starting point and to
carry at least a fixed proportion of the mass; both applications of
Section \ref{sec:applications} are of this kind, and so is the first of the three
settings below.

\begin{lemma}[A uniform minorant is a sufficient criterion for \hyp{S}{S}]\label{lem:uniformminorant}
Assume \hyp{A}{A}. Let $N \in \N$ and $c \in (0,1]$ be numbers and
let $u : \XX\times\XX \to [0,\infty)$ be a $\XXA\otimes\XXA$-measurable map with
\begin{enumerate}[label=(\alph*),nosep,topsep=3pt]
  \item $T^{N}(A \mid x) \ge \displaystyle\int_{A} u(y \mid x)\,\pi(\dd y)$ for every $x \in \XX$ and every $A \in \XXA$;
  \item $\displaystyle\int_{\XX} u(y \mid x)\,\pi(\dd y) \ge c$ for every $x \in \XX$.
\end{enumerate}
Then
\[
  \sing\bigl( T^{n}\circ\mu \bigm| \pi\bigr) \;\le\; (1-c)^{\lfloor n/N \rfloor}
  \qquad (\mu \in \Prob,\ n \in \Nz),
\]
and in particular \hyp{S}{S} holds, since $\delta_{x} \in \Prob$ for every
$x \in \XX$. The bound is uniform in $\mu$ and quantitative, so it gives considerably
more than \hyp{S}{S} asks.
\end{lemma}

\begin{proof}
For $x \in \XX$ the set function
$W_{x} := T^{N}_{x} - \int_{(\cdot)} u(y \mid x)\pi(\dd y)$
is a nonnegative measure by (a), of total mass
$W_{x}(\XX) = 1 - \int_{\XX} u(y\mid x)\pi(\dd y) \le 1-c$ by (b), and
$x \mapsto W_{x}(A)$ is measurable for every $A$. Fix $\mu \in \Prob$ and
define finite measures $\nu_{n} : \XXA \to [0,\infty)$ recursively by
\[
  \nu_{0} := \mu, \qquad \nu_{n+1}(A) := \int_{\XX} W_{x}(A)\,\nu_{n}(\dd x)
  \qquad (A \in \XXA,\ n \in \Nz),
\]
so that $\nu_{n+1}(\XX) \le (1-c)\,\nu_{n}(\XX)$ and hence
$\nu_{n}(\XX) \le (1-c)^{n}$.

We claim that $T^{Nn}\circ\mu = \gamma_{n} + \nu_{n}$ for a nonnegative measure
$\gamma_{n} \ll \pi$. For $n = 0$ this holds with $\gamma_{0} := 0$. Assume it for
$n$. Splitting $T^{N}$ according to (a),
\[
  T^{N(n+1)}\circ\mu
  = T^{N}\circ\gamma_{n}
    \;+\; \Bigl( A \mapsto \int_{A}\Bigl(\int_{\XX} u(y \mid x)\,\nu_{n}(\dd x)\Bigr)\pi(\dd y) \Bigr)
    \;+\; \nu_{n+1} .
\]
The first summand is absolutely continuous with respect to $\pi$ by
Lemma \ref{lem:monotone}(i), applied $N$ times to $\gamma_{n}/\gamma_{n}(\XX)$ if
$\gamma_{n} \ne 0$; the second is absolutely continuous by construction, its density
being measurable by Tonelli's theorem. Their sum is the required $\gamma_{n+1}$.

Consequently $\sing(T^{Nn}\circ\mu\mid\pi) \le \nu_{n}(\XX) \le (1-c)^{n}$ by
Definition \ref{def:singmass}, and the bound for a general $n \in \Nz$ follows from
Lemma \ref{lem:monotone}(iv), since $\lfloor n/N\rfloor N \le n$.
\end{proof}

\begin{remark}[What the criterion does and does not cover]\label{rem:uniformminorantscope}
Lemma \ref{lem:uniformminorant} contains \hyp{D}{D}, with $N := 1$, $u := t$ and
$c := 1$, and it is what makes both applications below immediate. It covers neither
\hyp{M}{M} nor \hyp{E}{E}, which is why Sections \ref{sec:mh} and \ref{sec:general}
carry out their own verifications: under \hyp{M}{M} the mass of the minorant is
$\theta(x)$, which need not be bounded away from $0$, and under \hyp{E}{E} the
minorant is available only on the set $\XX_{N}$ and not at every starting point.
\end{remark}

\begin{remark}[The main hypotheses, and where they are verified]\label{rem:roadmap}
The table records how the two hypotheses are met in the settings treated below. In
each line, $N$, $\YY$ and $s$ are the data of \hyp{P}{P}.
\begin{center}
\begin{tabular}{@{}llll@{}}
\toprule
Setting & \hyp{P}{P} via & \hyp{S}{S} via & Proved in \\
\midrule
\hyp{D}{D} & $N=1$, $\YY=\XX$, $s=t$
           & $T\circ\mu \ll \pi$
           & Theorem \ref{thm:main} \\
\hyp{M}{M} & $N=1$, $\YY=\XX$, $s=k$
           & $\sing \le \int r^{n}\dd\mu$
           & Theorem \ref{thm:mh} \\
\hyp{E}{E} & $N=N(\varepsilon)$, $\YY=\XX_{N}$, $s=t_{N}$
           & $\sing \le 1-\mu(\XX_{N})$
           & Theorem \ref{thm:general} \\
Gibbs      & one sweep, $s = d^{-d}\tau/p$
           & Lemma \ref{lem:uniformminorant}, same $s$
           & Corollary \ref{cor:gibbs} \\
Tempering  & one parallel update, $s = u$
           & Lemma \ref{lem:uniformminorant}, same $s$
           & Corollary \ref{cor:tempering} \\
\bottomrule
\end{tabular}
\end{center}
\noindent
Each of these five verifications also establishes, in the stronger form
$\pi \ll T^{n}_{x}$ at a finite time, the hypothesis \hyp{L}{L} of
Subsection \ref{ssec:L}; the corresponding table is Remark \ref{rem:Lverification},
and Remark \ref{rem:whylimit} explains why \hyp{L}{L} is nevertheless stated in the
limit.
\end{remark}

\subsection{Two examples: \texorpdfstring{\hyp{P}{P} versus \hyp{S}{S}}{(P) versus (S)}}\label{ssec:examples}

Neither hypothesis may be dropped, and neither implies the other.

\begin{example}[\hyp{S}{S} without \hyp{P}{P}]\label{ex:RnotU}
Let $\XX := \{1,2\}$ with $\XXA$ its power set, let $\pi$ be the uniform
distribution and let $T(A \mid x) := \ind_{A}(x)$ be the identity kernel. Then
\hyp{A}{A} hold, and \hyp{S}{S} holds because every $\mu \in \Prob$
satisfies $\mu \le 2\pi$, so that $\sing(T^{n}\circ\mu\mid\pi) = 0$ for every $n$. But
$T^{n}\circ\delta_{1} = \delta_{1}$ for every $n$, so the conclusion of
Theorem \ref{thm:criterion} fails, and with it \hyp{P}{P}. Concretely, \hyp{P}{P} for
a number $\varepsilon < \tfrac12$ would force $\YY = \XX$ and, by (c), a map $s$ with
$s(2 \mid 1) > 0$, whereas $T^{N}(\{2\} \mid 1) = 0$. The weaker hypothesis
\hyp{L}{L} of Subsection \ref{ssec:L} fails as well, and at every point:
$T^{n}_{x} = \delta_{x}$ for every $n$, while $\pi(\{y\}) = \tfrac12 > 0$
for the point $y \ne x$, so that $\sing(\pi \mid T^{n}_{x}) = \tfrac12$ for every
$x \in \XX$ and every $n \in \N$, and in particular does not tend to $0$.
\end{example}

\begin{example}[\hyp{P}{P} without \hyp{S}{S}]\label{ex:UnotR}
Let $\XX := [0,1]$ with its Borel $\sigma$-algebra and let $\pi$ be Lebesgue measure.
Fix a sequence of pairwise distinct points $x_{0},x_{1},x_{2},\dots \in [0,1]$, put
$G := \{x_{i} : i \in \Nz\} \in \XXA$, and fix numbers $a_{i} \in (0,1)$ with
$\sum_{i \in \Nz} a_{i} < \infty$. Define
\[
  T_{x_{i}} := (1-a_{i})\,\delta_{x_{i+1}} + a_{i}\,\pi
  \quad (i \in \Nz),
  \qquad
  T_{x} := \pi \quad (x \in \XX \setminus G).
\]
Then:
\begin{enumerate}[label=(\roman*),nosep,topsep=3pt]
  \item $T$ is a Markov kernel: each $T_{x}$ is a probability measure, and
        for $A \in \XXA$ the map
        $x \mapsto T(A \mid x) = \pi(A) + \sum_{i}\ind_{\{x = x_{i}\}}
        \bigl[(1-a_{i})\delta_{x_{i+1}}(A) + a_{i}\pi(A) - \pi(A)\bigr]$
        is measurable.
  \item $\pi$ is invariant, because $\pi(G) = 0$ and therefore
        $(T\circ\pi)(A) = \int_{\XX} T(A\mid x)\pi(\dd x) = \pi(A)$. Thus
        \hyp{A}{A} hold.
  \item \hyp{P}{P} holds with $N := 1$, $\YY := \XX$ and the everywhere strictly
        positive map $s(y \mid x) := \sum_{i} a_{i}\ind_{\{x = x_{i}\}} + \ind_{\XX\setminus G}(x)$,
        which does not depend on $y$: for $x \notin G$ one has
        $T_{x} = \pi$, and for $x = x_{i}$ one has
        $T_{x_{i}} \ge a_{i}\pi$. Hypothesis (c) is vacuous because $s$
        never vanishes.
  \item The conclusion of Theorem \ref{thm:criterion}, and therefore \hyp{S}{S},
        fails. Writing $p_{n} := \prod_{i<n}(1-a_{i})$, an induction using
        Lemma \ref{lem:monotone}(i) gives
        \[
          T^{n}_{x_{0}} = p_{n}\,\delta_{x_{n}} + \sigma_{n},
          \qquad \sigma_{n} \ll \pi ,
        \]
        because one step maps $p_{n}\delta_{x_{n}}$ to
        $p_{n}(1-a_{n})\delta_{x_{n+1}} + p_{n}a_{n}\pi$ and maps the absolutely
        continuous part to an absolutely continuous measure. Since
        $\sum_{i}a_{i} < \infty$ we have $p_{n} \downarrow p_{\infty} > 0$, whence
        \[
          \tv{T^{n}\circ\delta_{x_{0}} - \pi}
          \;\ge\; T^{n}(\{x_{n}\} \mid x_{0}) - \pi(\{x_{n}\})
          \;=\; p_{n} \;\ge\; p_{\infty} \;>\; 0
          \qquad (n \in \N),
        \]
        and likewise $\sing(T^{n}\circ\delta_{x_{0}}\mid\pi) \ge p_{\infty}$, since any
        $\beta \le T^{n}_{x_{0}}$ with $\beta \ll \pi$ gives no mass to
        $\{x_{n}\}$.
\end{enumerate}
This example shows what the two
properties do and do not control: the chain mixes perfectly as soon as the initial
law is dominated by a multiple of $\pi$, but from the point $x_{0}$ it never acquires
such a law, because with probability $p_{\infty} > 0$ it stays forever in the
$\pi$-null set $G$.
\end{example}

\subsection{The assumption \texorpdfstring{\hyp{P}{P}}{(P)} in its simplest form: \texorpdfstring{\hyp{L}{L} under \hyp{J}{J}}{(L) under (J)}}\label{ssec:L}

Assumption \hyp{P}{P} is shaped by the decision to assume no structure whatever on
$(\XX,\XXA)$: what it asks for is a \emph{jointly measurable} minorant density,
because on a bare measurable space a family of Radon--Nikodym derivatives indexed by
the starting point need not admit a jointly measurable version. In every setting
treated below such a minorant is at hand, and \hyp{P}{P} therefore costs nothing. But
it is not the shortest way to say what is being assumed. If one is willing to grant a
jointly measurable density for the absolutely continuous part of each iterate --- no
minorant, no positivity, no uniformity, only measurability in the two variables at
once --- then the hypothesis can be stated with no density in it at all, and the two
hypotheses of the convergence theorem become the two halves of a single sentence about
the Lebesgue decomposition of $T^{n}_{x}$. That grant is the following assumption.

\begin{assumption}[A jointly measurable density of the absolutely continuous part]\label{ass:J}
In addition to \hyp{A}{A}:
\begin{enumerate}[nosep,topsep=3pt,leftmargin=3em]
  \item[\hyptgt{J}{J}] For every number $n \in \N$ there is a
        $\XXA\otimes\XXA$-measurable map
        \[
          t_{n} : \XX\times\XX \longrightarrow [0,\infty),
          \qquad (x,y) \longmapsto t_{n}(y\mid x) ,
        \]
        such that $t_{n}(\,\cdot\mid x)$ is a density of $T^{n}_{x,\acp{\pi}}$ with
        respect to $\pi$, for \emph{every} $x \in \XX$:
        \[
          T^{n}_{x,\acp{\pi}}(A) \;=\; \int_{A} t_{n}(y \mid x)\,\pi(\dd y)
          \qquad (x \in \XX,\ A \in \XXA) .
        \]
\end{enumerate}
\end{assumption}

\noindent
Assumption \hyp{J}{J} is a hypothesis on the kernel, not on the space, and it is the
last of the measurable selections that these notes assume rather than construct: it
stands to the Lebesgue density of a general iterate as \hyp{D}{D} stands to the
transition density and \hyp{E}{E} to the family $(t_{n})$
(Remarks \ref{rem:jointmeasurability} and \ref{rem:measurableXn}). By
Proposition \ref{prop:jointdensity} below it follows from the sufficient condition
\hyp{C}{C}, which is easier to check and is what one verifies in practice.

\begin{assumption}[The $\sigma$-algebra is countably generated]\label{ass:C}
In addition to \hyp{A1}{A1}:
\begin{enumerate}[nosep,topsep=3pt,leftmargin=3em]
  \item[\hyptgt{C}{C}] There is a countable family $\mathcal{G} \subseteq \XXA$ with
        $\sigma(\mathcal{G}) = \XXA$.
\end{enumerate}
\end{assumption}

\begin{remark}[\hyp{J}{J} against \hyp{P}{P}: measurability against positivity]\label{rem:PvsJ}
Both \hyp{P}{P} and \hyp{J}{J} ask for a jointly measurable function of $(x,y)$, and it
is worth being clear about how they differ, since neither implies the other and the
paper uses both.

\hyp{J}{J} asks for the \emph{exact} density of $T^{n}_{x,\acp{\pi}}$, for every $n$,
and asks nothing else of it: it may vanish identically, and \hyp{J}{J} still holds ---
with $t_{n} \equiv 0$ --- for a kernel every one of whose iterates is singular with
respect to $\pi$, such as $T_{x} := \delta_{x}$ when the singletons are measurable and
$\pi$ is atomless. It is a pure
measurable-selection hypothesis, with no analytic content whatever. \hyp{P}{P}, by
contrast, asks for a \emph{minorant} $s$ of $T^{N}_{x}$ for a single well-chosen $N$,
and demands that the minorant be strictly positive on a large set, quantitatively so;
that positivity is the entire analytic content of the convergence theorem, and no
measurable selection could supply it. Neither hypothesis implies the other, and the
paper uses both.

The two therefore sit above and below each other in different senses, and it is the
combination of \hyp{J}{J} with \hyp{L}{L} that produces \hyp{P}{P}
(Proposition \ref{prop:LimpliesP}): \hyp{L}{L} contributes the positivity and
\hyp{J}{J} contributes the measurability. The gap between a minorant of $T^{N}_{x}$ and
a density is not an artefact either. In three of the five settings verified below, what
is exhibited for \hyp{P}{P} is a strict minorant of $T^{N}_{x}$ and nothing more: for
the random scan Gibbs sampler \eqref{eq:gibbsminorant} keeps a single one of the
$d^{d}$ update sequences, the systematic scan $(1,2,\dots,d)$, at the cost of the
factor $d^{-d}$, and discards the rest; for parallel tempering
\eqref{eq:temperingminorant} keeps only the part of the kernel in which no exchange is
performed, at the cost of the factor $1-\omega$; and under \hyp{M}{M} the atom at the
starting point is discarded outright. Whether the minorant happens also to be a density
of the absolutely continuous part varies, and one has to know more about $\pi$ to
decide: when the singletons are measurable and $\pi$ is atomless it does under
\hyp{M}{M} and for parallel tempering, everything discarded being then carried by a
$\pi$-null set, while for the Gibbs sampler it does not, the absolutely continuous part
of $T^{d}_{x}$ having a density that is a sum over the $d!$ sweeps --- the surjective
sequences among the $d^{d}$, the others contributing nothing absolutely continuous ---
of which $s$ retains one. Under the hypotheses of Lemmas \ref{lem:gibbssingular} and
\ref{lem:temperingsingular} one knows in addition that $T^{n}_{x}$ itself has no
density at all for any $n$, being not absolutely continuous with respect to $\pi$.
Assumption \hyp{P}{P} needs none of this: it asks about $T^{N}_{x}$ and not about its
absolutely continuous part, and the minorant is written down in one line either way.
\end{remark}

\begin{assumption}[Asymptotic domination of the target]\label{ass:L}
In addition to \hyp{A}{A}:
\begin{enumerate}[nosep,topsep=3pt,leftmargin=3em]
  \item[\hyptgt{L}{L}] $\displaystyle\lim_{n\to\infty}\ \sing\bigl( \pi \bigm| T^{n}_{x} \bigr) = 0$
        for $\aeevery$ $x \in \XX$.
\end{enumerate}
\end{assumption}

\noindent
Three readings of \hyp{L}{L} are worth having side by side. By
\eqref{eq:acpositive} it says that
$\pi(\{ g_{T^{n}_{x}} = 0 \}) \to 0$: the set on which the density of the absolutely
continuous part of $T^{n}_{x}$ still vanishes shrinks to $\pi$-measure zero. By
Definition \ref{def:singmass} it says that $\pi$ has minorants $\gamma_{n} \ll
T^{n}_{x}$ of mass tending to $1$: all but arbitrarily little of the target is
eventually seen by the chain. And by Lemma \ref{lem:monotone}(v) the sequence
$n \mapsto \sing(\pi\mid T^{n}_{x})$ is non-increasing, so \hyp{L}{L} is not
strengthened by demanding the bound for all large $n$ rather than along a subsequence.
Aperiodicity is not assumed separately, and does not have to be: a chain of period
$d \ge 2$ has $\sing(\pi\mid T^{n}_{x}) \ge 1 - 1/d$ for every $n$, since $T^{n}_{x}$
is then carried by a single cyclic class, so \hyp{L}{L} excludes periodicity of its
own accord. Like \hyp{S}{S}, it is a statement
about the chain started at a point; between them they make no reference to any initial
law, to any pair of starting points, or to any density.

\begin{remark}[\hyp{L}{L} read through \hyp{J}{J}]\label{rem:LviaJ}
Under \hyp{J}{J} the quantity in \hyp{L}{L} is the $\pi$-measure of the null set of a
single jointly measurable map:
\[
  \sing\bigl( \pi \bigm| T^{n}_{x} \bigr)
  \;=\; \pi\bigl( \{ y \in \XX : t_{n}(y \mid x) = 0 \} \bigr)
  \qquad (x \in \XX,\ n \in \N) .
\]
Indeed $t_{n}(\,\cdot\mid x)$ is a density of $T^{n}_{x,\acp{\pi}}$, so it is a
version of $g_{T^{n}_{x}}$, and \eqref{eq:acpositive} gives
$\sing(\pi\mid T^{n}_{x}) = \pi(\{g_{T^{n}_{x}} = 0\})$.

Two consequences, both used below. The right-hand side is measurable in $x$ by
Tonelli's theorem, which is what makes the sets $\XX_{n}$ of
Proposition \ref{prop:LimpliesP} measurable; that, together with supplying the jointly
measurable map that \hyp{P}{P} asks for, is all that \hyp{J}{J} does there --- no
analytic content of it is used. And \hyp{L}{L} becomes a statement one can read off a
single formula:
the set on which the density still vanishes shrinks to $\pi$-measure zero. Without
\hyp{J}{J} the same identity holds for each fixed $x$ with $g_{T^{n}_{x}}$ in place of
$t_{n}(\,\cdot\mid x)$, but $g_{T^{n}_{x}}$ is determined only up to a $\pi$-null set
for each $x$ separately, and no measurability in $x$ is available.

Measurability in $x$ also makes an integrated form of \hyp{L}{L} available, and that
form is not weaker. Write
\[
  Z_{n} \;:=\; \bigl\{ (x,y) \in \XX\times\XX \ :\ t_{n}(y\mid x) = 0 \bigr\}
  \;\in\; \XXA\otimes\XXA .
\]
\emph{Under \hyp{A}{A} and \hyp{J}{J}, hypothesis \hyp{L}{L} holds if and only if}
\begin{equation}
  (\pi\otimes\pi)(Z_{n}) \;\longrightarrow\; 0 \qquad\text{as } n \to \infty .
  \label{eq:Lintegrated}
\end{equation}
Indeed, put $\varphi_{n}(x) := \pi(\{y : t_{n}(y\mid x) = 0\})$, so that
$\varphi_{n}(x) = \sing(\pi\mid T^{n}_{x})$ for every $x$ by the identity above, so
that $0 \le \varphi_{n} \le 1$, and so that
$(\pi\otimes\pi)(Z_{n}) = \int_{\XX}\varphi_{n}\dd\pi$
by Tonelli's theorem. By Lemma \ref{lem:monotone}(v) the sequence
$n \mapsto \varphi_{n}(x)$ is non-increasing for \emph{every} $x$, so it converges
pointwise to some measurable $\varphi_{\infty} : \XX \to [0,1]$; since $\pi$ is finite,
dominated convergence gives
$\int_{\XX}\varphi_{n}\dd\pi \to \int_{\XX}\varphi_{\infty}\dd\pi$. Hence
\eqref{eq:Lintegrated} holds if and only if $\int_{\XX}\varphi_{\infty}\dd\pi = 0$, if
and only if $\varphi_{\infty} = 0$ $\aewhere$, which is \hyp{L}{L}.

The two directions cost different things. That \hyp{L}{L} implies
\eqref{eq:Lintegrated} is bounded convergence alone, and needs no monotonicity. The
converse does need Lemma \ref{lem:monotone}(v): convergence in $L^{1}(\pi)$ by
itself yields convergence $\aewhere$ only along a subsequence, and it is the
monotonicity in $n$ --- which makes the limit exist at every $x$ --- that upgrades this
to the full sequence. Two further points are worth recording. The quantity in
\eqref{eq:Lintegrated} does not depend on the version of $t_{n}$ chosen in \hyp{J}{J},
since for each fixed $x$ any two versions agree $\aewhere$, so that $\varphi_{n}(x)$ is
unchanged for every $x$. And reading the double integral in the other order exhibits
the same number as
$\int_{\XX}\pi(\{x : t_{n}(y\mid x) = 0\})\,\pi(\dd y)$, the average over targets $y$ of
the mass of starting points from which the chain still does not see $y$ --- which is
what asymptotic domination of the target asserts, read backwards. Form
\eqref{eq:Lintegrated} is a single number for each $n$, so it is the form that admits a
rate; \hyp{L}{L} is kept as the hypothesis because it is a statement about the chain
started at one point, and it is in that form that it stands beside \hyp{S}{S}.
\end{remark}

\begin{remark}[Why the limit, and not domination at a finite time]\label{rem:whylimit}
The stronger hypothesis that $\pi \ll T^{n}_{x}$ for \emph{some} $n = n(x)$ --- that
is, $\sing(\pi\mid T^{n}_{x}) = 0$ eventually --- is what one verifies in each of the
five settings of Remark \ref{rem:Lverification}, and it implies \hyp{L}{L} at once by
Lemma \ref{lem:monotone}(v). It is nevertheless strictly stronger, and the gap
is not a curiosity: it is exactly the case of a \emph{local} proposal.

Run random-walk Metropolis on $\XX = \R^{d}$ with a proposal that is uniform on the
ball of radius $r$ about the current point, against a target $\pi = p\lambda$ with
$p > 0$ everywhere and unbounded support. Every proposal, accepted or not, moves at
most $r$, so $T^{n}_{x}$ is carried by the closed ball $\overline{B}(x,nr)$; since
$\pi(\R^{d}\setminus\overline{B}(x,nr)) > 0$, we have $\pi \not\ll T^{n}_{x}$ for
\emph{every} $n$, and domination at a finite time fails at every point. On the other
hand the component of $T^{n}_{x}$ in which all $n$ proposals are accepted has, with
respect to $\lambda$, a density bounded below by the $n$-fold convolution of the
uniform density times the product of the acceptance probabilities, which is strictly
positive $\lambda$-almost everywhere on $B(x,nr)$ because $p$ is; hence
$\pi\restriction B(x,nr) \ll T^{n}_{x}$ and
\[
  \sing\bigl( \pi \bigm| T^{n}_{x} \bigr) \;\le\; \pi\bigl( \R^{d}\setminus B(x,nr) \bigr)
  \;\longrightarrow\; 0 ,
\]
so \hyp{L}{L} holds. The same happens for a \emph{lazy} nearest-neighbour chain on a
connected, countably infinite graph with $\pi$ of full support: the $n$-step reachable
set is the graph ball of radius $n$, which is finite for every $n$ but increases to
the whole space, so $\sing(\pi\mid T^{n}_{x})$ equals the $\pi$-measure of its
complement and tends to $0$. Laziness is what is needed here and not elsewhere: on a
bipartite graph the chain without self-loops has period $2$, and then \hyp{L}{L} fails,
as it must.

Local proposals on unbounded state spaces are precisely the case that domination at a
finite time cannot reach and \hyp{L}{L} can, which is why the hypothesis is stated in
the limit. None of the five settings of Remark \ref{rem:Lverification} is of this
kind --- each of them exhibits a minorant valid on the whole space --- so the gain is
not used below; it is recorded because it is the first thing a reader with a local
proposal will ask about.
\end{remark}

\begin{lemma}[\hyp{P}{P} implies \hyp{L}{L}]\label{lem:PimpliesL}
Assume \hyp{A}{A} and \hyp{P}{P}. Then \hyp{L}{L} holds.
\end{lemma}

\begin{proof}
We use twice the observation that, by \hyp{P}{P}(b), the measure
$A \mapsto \int_{A}s(y\mid x)\pi(\dd y)$ lies below $T^{N}_{x}$, so that
$\pi\restriction\{s(\,\cdot\mid x) > 0\} \ll T^{N}_{x}$ and hence, by
Definition \ref{def:singmass},
\begin{equation}
  \sing\bigl( \pi \bigm| T^{N}_{x} \bigr)
  \;\le\; \pi\bigl( \{ y \in \XX : s(y \mid x) = 0 \} \bigr)
  \qquad (x \in \YY) .
  \label{eq:PtoL}
\end{equation}

Let $k \in \N$ and apply \hyp{P}{P} with $\varepsilon_{k} := 4^{-k}$, obtaining
$N_{k}$, $\YY_{k}$, $\YY_{k}'$ and $s_{k}$. The set
$Z_{k} := \{(x,y) \in \YY_{k}\times\XX : s_{k}(y\mid x) = 0\}$ is
$\XXA\otimes\XXA$-measurable, and integrating \hyp{P}{P}(c) over $y$, splitting off
$\XX\setminus\YY_{k}'$, Tonelli's theorem gives
\[
  (\pi\otimes\pi)(Z_{k})
  \;=\; \int_{\XX} \pi\bigl(\{x\in\YY_{k} : s_{k}(y\mid x) = 0\}\bigr)\,\pi(\dd y)
  \;\le\; \varepsilon_{k} + \varepsilon_{k}
  \;=\; 2\cdot 4^{-k} .
\]
Reading the same double integral in the other order and applying Markov's inequality,
the set
\[
  A_{k} := \bigl\{ x \in \YY_{k} \ :\ \pi(\{y : s_{k}(y\mid x) = 0\}) \le 2^{-k} \bigr\}
\]
satisfies $\pi(\YY_{k}\setminus A_{k}) \le 2\cdot4^{-k}/2^{-k} = 2\cdot 2^{-k}$, so
that $\pi(\XX\setminus A_{k}) \le 4^{-k} + 2\cdot2^{-k} \le 3\cdot 2^{-k}$. By
\eqref{eq:PtoL}, every $x \in A_{k}$ satisfies
$\sing(\pi \mid T^{N_{k}}_{x}) \le 2^{-k}$, hence
$\sing(\pi \mid T^{n}_{x}) \le 2^{-k}$ for every $n \ge N_{k}$ by
Lemma \ref{lem:monotone}(v).

Since $\sum_{k}\pi(\XX\setminus A_{k}) < \infty$, the Borel--Cantelli lemma gives
$\pi(\liminf_{k} A_{k}) = 1$. So for $\aeevery$ $x$ there is a number $k_{0}$ with
$x \in A_{k}$ for every $k \ge k_{0}$, and then
\[
  \limsup_{n\to\infty}\ \sing(\pi\mid T^{n}_{x}) \;\le\; 2^{-k}
  \qquad (k \ge k_{0}) .
\]
Hence $\sing(\pi\mid T^{n}_{x}) \to 0$ for $\aeevery$ $x$, which is \hyp{L}{L}.
\end{proof}

\noindent
Under \hyp{J}{J} the converse of Lemma \ref{lem:PimpliesL} holds as well, so that
there \hyp{L}{L} and \hyp{P}{P} are equivalent. We give the argument in two steps:
first that \hyp{J}{J} is available on a countably generated $\sigma$-algebra, then the
deduction of \hyp{P}{P} from it. The first step is the one place in these notes where a
genuine piece of martingale theory is used, and it is precisely the piece that both
\hyp{P}{P} and \hyp{J}{J} were introduced in order to avoid assuming; a reader content
to assume \hyp{J}{J} outright, as we do in Theorem \ref{thm:asympequiv}, may skip it.

\begin{proposition}[\hyp{C}{C} implies \hyp{J}{J}]\label{prop:jointdensity}
Assume \hyp{A}{A} and \hyp{C}{C}. Then \hyp{J}{J} holds. Moreover the map $t_{n}$ it
provides satisfies $t_{n}(\,\cdot\mid x) = g_{T^{n}_{x}}$ $\aewhere$ and
$T^{n}(A\mid x) \ge \int_{A}t_{n}(y\mid x)\,\pi(\dd y)$, for every $x \in \XX$ and
every $A \in \XXA$.
\end{proposition}

\begin{proof}
Fix $n \in \N$. \emph{The construction.} Let $\mathcal{G} = \{G_{1},G_{2},\dots\}$ be a countable
family with $\sigma(\mathcal{G}) = \XXA$, as provided by \hyp{C}{C}, and put
$\mathcal{F}_{k} := \sigma(G_{1},\dots,G_{k})$. Each $\mathcal{F}_{k}$ is a finite
$\sigma$-algebra; let $\mathcal{P}_{k} \subseteq \XXA$ be its finite partition of
$\XX$ into atoms. The $\mathcal{F}_{k}$ increase, each $\mathcal{P}_{k+1}$ refines
$\mathcal{P}_{k}$, and $\sigma\bigl(\bigcup_{k}\mathcal{F}_{k}\bigr) = \XXA$. Define
\[
  t_{n,k}(y \mid x) := \sum_{\substack{P \in \mathcal{P}_{k} \\ \pi(P) > 0}}
  \frac{T^{n}(P \mid x)}{\pi(P)}\;\ind_{P}(y)
  \qquad (x,y \in \XX,\ k \in \N) ,
\]
a finite sum of products of a measurable map of $x$ with a measurable map of $y$, and
therefore $\XXA\otimes\XXA$-measurable; and put
\[
  t_{n} := \Bigl( \limsup_{k\to\infty} t_{n,k} \Bigr)\,
           \ind_{\{\,\limsup_{k} t_{n,k} \,<\, \infty\,\}} ,
\]
which is $\XXA\otimes\XXA$-measurable with values in $[0,\infty)$.

\emph{The limit, at a fixed starting point.} Fix $x \in \XX$, write
$\alpha := T^{n}_{x} \in \Prob$ and put $\lambda := \pi + \alpha$, a finite
measure with $\pi \ll \lambda$ and $\alpha \ll \lambda$. Let
$u := \dd\pi/\dd\lambda$ and $v := \dd\alpha/\dd\lambda$; these lie in
$L^{1}(\lambda)$ and satisfy $u+v = 1$ $\lambda$-almost everywhere. Write
$u_{k} := \Ex_{\lambda}[\,u \mid \mathcal{F}_{k}\,]$ and
$v_{k} := \Ex_{\lambda}[\,v \mid \mathcal{F}_{k}\,]$, conditional expectations under
$\lambda$; since $\mathcal{F}_{k}$ is generated by the finite partition
$\mathcal{P}_{k}$, these are the corresponding cell averages, so that on a cell
$P \in \mathcal{P}_{k}$ with $\pi(P) > 0$ --- whence $\lambda(P) > 0$ ---
\[
  u_{k} = \frac{\pi(P)}{\lambda(P)} \;>\; 0,
  \qquad
  v_{k} = \frac{\alpha(P)}{\lambda(P)},
  \qquad\text{and therefore}\qquad
  t_{n,k}(\,\cdot\mid x) = \frac{v_{k}}{u_{k}}
\]
there. The remaining cells of $\mathcal{P}_{k}$ are finitely many sets of $\pi$-measure
$0$, so the displayed identity holds $\aewhere$. By L\'evy's upward theorem applied
under $\lambda$, $u_{k} \to u$ and $v_{k} \to v$ $\lambda$-almost everywhere, hence
$\pi$-almost everywhere. Moreover
$\pi(\{u = 0\}) = \int_{\{u=0\}} u \dd\lambda = 0$, so $u > 0$ $\aewhere$. Therefore
$t_{n,k}(y\mid x) \to v(y)/u(y)$ for $\aeevery$ $y$; in particular the limit superior
is finite $\aewhere$, and
\[
  t_{n}(\,\cdot\mid x) = \frac{v}{u} \qquad \aewhere .
\]

\emph{Identification of the limit.} Put $S := \{u = 0\} \in \XXA$, so $\pi(S) = 0$.
The measure $A \mapsto \alpha(A \cap S)$ is carried by the $\pi$-null set $S$. The
measure $A \mapsto \alpha(A \setminus S)$ is absolutely continuous with respect to
$\pi$: if $\pi(A) = 0$ then $\int_{A}u\dd\lambda = 0$, and $u > 0$ on $A \setminus S$,
so $\lambda(A\setminus S) = 0$ and hence $\alpha(A \setminus S) = 0$. By the
uniqueness in Definition \ref{def:lebesgue} we conclude
$\alpha_{\acp{\pi}}(A) = \alpha(A\setminus S)$ for every $A \in \XXA$. Finally, since
$\dd\pi = u \dd\lambda$ and $\pi(S) = 0$,
\[
  \int_{A} \frac{v}{u}\,\ind_{\XX\setminus S} \dd\pi
  \;=\; \int_{A} \frac{v}{u}\,\ind_{\XX\setminus S}\,u \dd\lambda
  \;=\; \int_{A\setminus S} v \dd\lambda
  \;=\; \alpha(A\setminus S)
  \;=\; \alpha_{\acp{\pi}}(A)
  \qquad (A \in \XXA) ,
\]
so $v/u$, and with it $t_{n}(\,\cdot\mid x)$, is a density of $\alpha_{\acp{\pi}}$ with
respect to $\pi$. The last assertion of the statement follows from
$\alpha_{\acp{\pi}} \le \alpha$.
\end{proof}

\begin{remark}[Attribution]\label{rem:AJ}
For a single fixed $x$ the convergence established in the proof --- that the
elementary densities $t_{n,k}(\,\cdot\mid x)$ of $\alpha$ over a filtration generating
$\XXA$ converge $\aewhere$ to the density of the \emph{absolutely continuous part} of
$\alpha$, the singular part washing out in the limit --- is the theorem of Andersen
and Jessen \cite{AndersenJessen48}; see also \cite{Horowitz78}. Note that
$k \mapsto t_{n,k}(\,\cdot\mid x)$ is in general only a nonnegative
\emph{super}\-martingale under $\pi$, the defect at stage $k$ being the mass that
$\alpha$ places on the $\pi$-null cells of $\mathcal{P}_{k+1}$; it is a martingale
exactly when $\alpha$ charges no $\pi$-null cell of any $\mathcal{P}_{k}$, which is
implied by $\alpha \ll \pi$ but is strictly weaker than it, since whether a given
singular $\alpha$ meets a $\pi$-null cell depends on the generating family chosen in
\hyp{C}{C}. We have given the proof rather than quoted it for two
reasons: what is needed here is not the classical statement for one measure but the
\emph{joint} measurability of the resulting map in $(x,y)$, which the classical
statement does not address and which the explicit construction supplies for free; and
the route through $\lambda = \pi + \alpha$ makes the identification of the limit a
three-line computation from L\'evy's upward theorem, so that nothing is being taken on
trust.
\end{remark}

\begin{proposition}[\hyp{L}{L} implies \hyp{P}{P} under \hyp{J}{J}]\label{prop:LimpliesP}
Assume \hyp{A}{A}, \hyp{J}{J} and \hyp{L}{L}. Then \hyp{P}{P} holds. Consequently,
under \hyp{A}{A} and \hyp{J}{J} the hypotheses \hyp{L}{L} and \hyp{P}{P} are
equivalent, and the same holds under \hyp{A}{A} and \hyp{C}{C} by
Proposition \ref{prop:jointdensity}.
\end{proposition}

\begin{proof}
Let $\varepsilon \in (0,1)$ be a number. For $n \in \N$ let $t_{n}$ be the map
provided by \hyp{J}{J} and put
\[
  \XX_{n} := \bigl\{ x \in \XX \ :\ \pi(\{ y \in \XX : t_{n}(y \mid x) = 0 \})
  \le \varepsilon^{2} \bigr\} .
\]
Each $\XX_{n}$ lies in $\XXA$: the map
$x \mapsto \int_{\XX}\ind_{\{t_{n}(y\mid x) = 0\}}\,\pi(\dd y)$ is measurable by
Tonelli's theorem, the set $\{(x,y) : t_{n}(y\mid x) = 0\}$ being
$\XXA\otimes\XXA$-measurable. By Remark \ref{rem:LviaJ},
\[
  \XX_{n} = \bigl\{ x \in \XX \ :\ \sing(\pi \mid T^{n}_{x}) \le \varepsilon^{2} \bigr\} ,
\]
so $\XX_{n} \subseteq \XX_{n+1}$ by Lemma \ref{lem:monotone}(v), and \hyp{L}{L}
gives $\pi\bigl(\bigcup_{n\in\N}\XX_{n}\bigr) = 1$. By continuity from
below there is $N \in \N$ with $\pi(\XX\setminus\XX_{N}) \le \varepsilon$. Take this
$N$, the set $\YY := \XX_{N}$ and the map $s := t_{N}$, which is
$\XXA\otimes\XXA$-measurable by \hyp{J}{J}. Then \hyp{P}{P}(b) holds for every
$x \in \XX$, hence in particular on $\YY$: by \hyp{J}{J} and
Theorem \ref{thm:lebesgue},
\[
  \int_{A} s(y \mid x)\,\pi(\dd y) \;=\; T^{N}_{x,\acp{\pi}}(A) \;\le\; T^{N}(A \mid x)
  \qquad (x \in \XX,\ A \in \XXA) ,
\]
the inequality being $\nu_{\acp{\mu}} \le \nu$ of Theorem \ref{thm:lebesgue}.

It remains to produce $\YY'$. By the definition of $\XX_{N}$ and Tonelli's theorem,
the $\XXA\otimes\XXA$-measurable set
$Z := \{(x,y) \in \YY\times\XX : s(y\mid x) = 0\}$ satisfies
\[
  \int_{\XX} \pi\bigl( \{ x \in \YY : s(y\mid x) = 0 \} \bigr) \pi(\dd y)
  \;=\; (\pi\otimes\pi)(Z)
  \;=\; \int_{\YY} \pi\bigl( \{ y : s(y\mid x) = 0 \} \bigr) \pi(\dd x)
  \;\le\; \varepsilon^{2} .
\]
By Markov's inequality the set
$\YY' := \{ y \in \XX : \pi(\{x \in \YY : s(y\mid x) = 0\}) \le \varepsilon \}$
satisfies $\pi(\XX\setminus\YY') \le \varepsilon^{2}/\varepsilon = \varepsilon$. So
\hyp{P}{P}(a) holds for $\YY$ and $\YY'$, and \hyp{P}{P}(c) holds by the definition of
$\YY'$. The converse implication is Lemma \ref{lem:PimpliesL}.
\end{proof}

\subsection{The main theorem: asymptotic equivalence with the target}\label{ssec:main}

\begin{theorem}[Asymptotic equivalence with the target is necessary and sufficient]\label{thm:asympequiv}
Assume \hyp{A}{A} and \hyp{J}{J} --- in particular, by
Proposition \ref{prop:jointdensity}, whenever \hyp{A}{A} and \hyp{C}{C} hold. Then the
following are equivalent.
\begin{enumerate}[label=\textup{(\roman*)},nosep,topsep=3pt,leftmargin=2.4em,itemsep=3pt]
  \item \emph{Asymptotic equivalence with the target:} both
        \[
          \sing\bigl( T^{n}_{x} \bigm| \pi \bigr) \;\longrightarrow\; 0
          \qquad\text{and}\qquad
          \sing\bigl( \pi \bigm| T^{n}_{x} \bigr) \;\longrightarrow\; 0
        \]
        as $n \to \infty$, the first for every $x \in \XX$ and the second for
        $\aeevery$ $x \in \XX$; that is, \hyp{S}{S} and \hyp{L}{L}.
  \item $\displaystyle\lim_{n\to\infty}\tv{T^{n}_{x} - \pi} = 0$ for every
        $x \in \XX$.
  \item $\displaystyle\lim_{n\to\infty}\ \sup_{A\in\XXA}\bigl|(T^{n}\circ\mu)(A) - \pi(A)\bigr| = 0$
        for every $\mu \in \Prob$.
\end{enumerate}
In that case $\pi$ is the unique invariant probability measure of $T$.
\end{theorem}

\begin{proof}
\emph{(i) $\Rightarrow$ (iii).} \hyp{P}{P} holds by
Proposition \ref{prop:LimpliesP}, and Theorem \ref{thm:criterion} applies. This is the
only implication that uses \hyp{J}{J}.

\emph{(iii) $\Rightarrow$ (ii)} is the case $\mu := \delta_{x}$.

\emph{(ii) $\Rightarrow$ (i)} is Corollary \ref{cor:singtv}, applied at each
$x \in \XX$ with $\alpha := T^{n}_{x}$ and $\rho := \pi$; it gives both halves of (i)
at \emph{every} $x$, which is more than (i) asks. Uniqueness follows as in
Theorem \ref{thm:abstract}.
\end{proof}

\begin{remark}[What the equivalence says, and what it costs]\label{rem:iff}
Theorem \ref{thm:asympequiv} is Theorem \ref{thm:crit} of the introduction together with
its converse. The two hypotheses are the two halves of the single relation
$T^{n}_{x} \sim \pi$ of \ref{conv:order}, each asked only in the limit: the law of the
chain and the target become mutually absolutely continuous as $n \to \infty$, in the
sense that both singular masses vanish. Neither half alone suffices
(Examples \ref{ex:RnotU} and \ref{ex:UnotR}), and neither can be weakened, since
together they are implied by the conclusion.

Two remarks on the price. Only the implication (i) $\Rightarrow$ (iii) uses
\hyp{J}{J}, and it uses it only through Proposition \ref{prop:LimpliesP}, that is only
to make the sets $\XX_{n}$ there measurable (Remark \ref{rem:LviaJ}); the implication
(ii) $\Rightarrow$ (i) is Corollary \ref{cor:singtv} and is free of every hypothesis
but \hyp{A1}{A1}. And Theorem \ref{thm:criterion} is the same sufficiency statement
with \hyp{L}{L} replaced by \hyp{P}{P}, which assumes the jointly measurable minorant
outright and therefore needs neither \hyp{J}{J} nor anything else
(Remark \ref{rem:PvsJ}).

Finally, a word on the quantifiers, which are not the same in the two halves of (i):
\hyp{S}{S} is asked at \emph{every} $x$ and \hyp{L}{L} only at $\aeevery$ $x$. The
asymmetry is forced by how each is used. Assumption \hyp{S}{S} is consumed through
\eqref{prop:R}, which Theorem \ref{thm:abstract} applies at $\delta_{x}$ separately
for each $x$, so it must hold at every starting point one wants convergence from; and
by Corollary \ref{cor:singtv} it \emph{has} to, since convergence at $x$ bounds
$\sing(T^{n}_{x}\mid\pi)$ by $\tv{T^{n}_{x}-\pi}$. Assumption \hyp{L}{L}, by contrast,
is consumed only through \hyp{P}{P} and hence through \eqref{prop:U}, which quantifies
over laws dominated by $M\pi$ and is therefore blind to $\pi$-null sets.

The equivalence then conceals a self-improvement, which is worth stating because the
mismatched quantifiers otherwise look like an oversight: since (i) implies (ii), and
(ii) gives both halves at \emph{every} $x$ by Corollary \ref{cor:singtv},
\begin{quote}
under \hyp{A}{A} and \hyp{J}{J}, and in the presence of \hyp{S}{S} at every $x$,
hypothesis \hyp{L}{L} at $\aeevery$ $x$ implies \hyp{L}{L} at every $x$.
\end{quote}
We know no direct argument for this; it is obtained by going round the equivalence,
and the passage (i) $\Rightarrow$ (iii) that it goes through is exactly the one that
uses \hyp{J}{J}.
\end{remark}

\begin{remark}[The almost everywhere variant]\label{rem:aevariant}
Weakening \hyp{S}{S} to $\aeevery$ $x$ as well weakens the conclusion in exactly the
same way, and the equivalence survives. Assume \hyp{A}{A} and \hyp{J}{J}. Then
\[
  \begin{aligned}
    &\Bigl[\ \sing(T^{n}_{x}\mid\pi) \to 0
      \ \text{ and } \ \sing(\pi\mid T^{n}_{x}) \to 0
      \ \text{ for } \aeevery\ x \ \Bigr] \\[2pt]
    &\qquad\Longleftrightarrow\quad
      \Bigl[\ \tv{T^{n}_{x} - \pi} \to 0 \ \text{ for } \aeevery\ x \ \Bigr] ,
  \end{aligned}
\]
and either implies $\tv{T^{n}\circ\mu-\pi} \to 0$ for every $\mu \ll \pi$.

For ``$\Rightarrow$'': \hyp{P}{P} holds by Proposition \ref{prop:LimpliesP}, which
uses only the $\aeevery$ form of \hyp{L}{L}, hence \eqref{prop:U} holds by
Lemma \ref{lem:PimpliesU}; and for each $x$ at which the first limit vanishes,
Lemma \ref{lem:SiffR} gives \eqref{prop:R} at $\delta_{x}$, so the proof of
Theorem \ref{thm:abstract} gives $\tv{T^{n}_{x}-\pi} \to 0$ there. For
``$\Leftarrow$'': Corollary \ref{cor:singtv}. The final clause is
Lemma \ref{lem:pointwise}, whose hypothesis is convergence at $\mu$-almost every
point, applied with $\XX'$ the set of full $\pi$-measure just obtained --- which has
full $\mu$-measure whenever $\mu \ll \pi$.

What is lost is exactly the Dirac measures at the exceptional points: convergence for
every $\mu \in \Prob$ is not available, because $\delta_{x}$ is not absolutely
continuous with respect to $\pi$ when $\pi$ is atomless. This is the distinction
between convergence of \emph{all} and of $\aeevery$ transition probability drawn by
Scheutzow and Schindler \cite{ScheutzowSchindler21}; see Remark \ref{rem:necessary}.
\end{remark}

\begin{remark}[The criterion in one sentence]\label{rem:onesentence}
In words, and without a symbol: \emph{a Markov chain with an invariant probability
measure, whose iterates have jointly measurable Lebesgue densities, converges in total
variation from every starting point if and only if, from every starting point, its law
and the target become mutually absolutely continuous in the limit.} The measurability
proviso is needed for one direction only, and is automatic on a countably generated
$\sigma$-algebra (Proposition \ref{prop:jointdensity}).
Theorem \ref{thm:criterion} is that direction with \hyp{L}{L} traded for \hyp{P}{P},
and needs no proviso at all.
\end{remark}

\begin{remark}[How \hyp{L}{L} is met in the five settings]\label{rem:Lverification}
Assumption \hyp{L}{L} holds in every setting treated in these notes, and each
verification is one line, being a weaker statement than the verification of
\hyp{P}{P} already carried out there. In every one of the five it holds in the
stronger form $\pi \ll T^{n}_{x}$, that is $\sing(\pi\mid T^{n}_{x}) = 0$, at a finite
time; \hyp{L}{L} then follows by Lemma \ref{lem:monotone}(v). In each line, the
number $n$ is that finite time.
\begin{center}
\begin{tabular}{@{}lll@{}}
\toprule
Setting & \hyp{L}{L} via & with \\
\midrule
\hyp{D}{D} & $T_{x} = t(\,\cdot\mid x)\pi$, $t > 0$ & $n=1$, every $x$ \\
\hyp{M}{M} & $T_{x} \ge k(\,\cdot\mid x)\pi$, $k > 0$ a.e. & $n=1$, $\aeevery$ $x$ \\
\hyp{E}{E} & $T^{n}_{x} = t_{n}(\,\cdot\mid x)\pi$ on $\XX_{n}$ & $n = n(x)$, every $x$ \\
Gibbs, random scan & $T^{d}_{x} \ge s(\,\cdot\mid x)\pi$, $s>0$ & $n=d$, every $x$ \\
Parallel tempering & $T_{x} \ge u(\,\cdot\mid x)\pi$, $u>0$ a.e. & $n=1$, $\aeevery$ $x$ \\
\bottomrule
\end{tabular}
\end{center}
In detail, and in each case by \eqref{eq:acpositive} together with the fact that
$\pi \ll \sigma \le \alpha$ implies $\pi \ll \alpha$:
\begin{enumerate}[label=\textup{(\roman*)},nosep,topsep=3pt,leftmargin=2.4em,itemsep=3pt]
  \item \hyp{D}{D}: the density $t(\,\cdot\mid x)$ is strictly positive at every point,
        so $g_{T_{x}} = t(\,\cdot\mid x) > 0$ and $\pi \ll T_{x}$
        for every $x \in \XX$. In the almost everywhere form of
        Remark \ref{rem:aepositivity} the same follows for $\aeevery$ $x$, by Tonelli's
        theorem, which is all \hyp{L}{L} asks.
  \item \hyp{M}{M}: by \hyp{M}{M}(a) and $r \ge 0$ one has
        $T(A\mid x) \ge \int_{A}k(y\mid x)\pi(\dd y)$ for every $x$, and \hyp{M}{M}(c)
        with Tonelli's theorem gives $k(\,\cdot\mid x) > 0$ $\aewhere$ for $\aeevery$
        $x$; for such $x$, $\pi \ll k(\,\cdot\mid x)\pi \le T_{x}$.
  \item \hyp{E}{E}: for $x \in \XX_{n}$ one has $T^{n}_{x} = t_{n}(\,\cdot\mid x)\pi$
        with $t_{n}(\,\cdot\mid x) > 0$ $\aewhere$ by \hyp{E}{E}(b),(c), so
        $\pi \ll T^{n}_{x}$; and $\bigcup_{n}\XX_{n} = \XX$ by
        \hyp{E}{E}(a). This is the setting in which $n$ genuinely depends on $x$, and
        it is exactly what forces \hyp{P}{P} to carry the tolerance $\varepsilon$ and
        the set $\YY$: a single number $N$ serving all starting points does not exist,
        whereas \hyp{L}{L} never asks for one.
  \item Random scan Gibbs sampler: \eqref{eq:gibbsminorant} gives
        $T^{d}(A\mid x) \ge \int_{A}s(y\mid x)\pi(\dd y)$ for every $x \in \XX$ with
        $s > 0$ at every point, so $\pi \ll T^{d}_{x}$ for every $x$.
  \item Parallel tempering: \eqref{eq:temperingminorant} gives
        $T(A \mid x) \ge \int_{A}u(y\mid x)\pi(\dd y)$ for every $x \in \XX$, and the
        proof of Corollary \ref{cor:tempering} shows
        $\{(x,y) : u(y\mid x) = 0\}$ to be $(\pi\otimes\pi)$-null, so that by Tonelli's
        theorem $u(\,\cdot\mid x) > 0$ $\aewhere$ for $\aeevery$ $x$.
\end{enumerate}
Two observations. In four of the five settings a single number $n$ serves all
starting points at once --- every $x$ under \hyp{D}{D} and for the Gibbs sampler,
$\aeevery$ $x$ under \hyp{M}{M} and for parallel tempering --- and only under
\hyp{E}{E} must $n$ depend on the starting point.
And in every case what is exhibited is a minorant \emph{density}, which is more than
\hyp{L}{L} asks: \hyp{L}{L} does not require the domination of $\pi$ to be witnessed
by any density at all, jointly measurable or not, and that is the whole of its
advantage. Consequently each of Theorems \ref{thm:main}, \ref{thm:mh} and
\ref{thm:general} and each of Corollaries \ref{cor:gibbs} and \ref{cor:tempering} may
be read off from Theorem \ref{thm:asympequiv} in place of Theorem \ref{thm:criterion},
at the price of \hyp{J}{J}; we have kept the proofs through \hyp{P}{P}, which needs no
such measurable selection (Remark \ref{rem:PvsJ}).
\end{remark}

\section{Convergence under a positive transition density}\label{sec:first}

\subsection{The assumption \texorpdfstring{\hyp{D}{D}}{(D)}}

\begin{assumption}[A strictly positive transition density]\label{ass:A4}
In addition to \hyp{A}{A}:
\begin{enumerate}[nosep,topsep=3pt,leftmargin=3em]
  \item[\hyptgt{D}{D}] There is a $\XXA\otimes\XXA$-measurable map
        \[
          t : \XX \times \XX \longrightarrow (0,\infty), \qquad (x,y) \longmapsto t(y \mid x),
        \]
        such that
        \[
          T(A \mid x) = \int_{A} t(y \mid x)\, \pi(\dd y)
          \qquad\text{for every } x \in \XX \text{ and every } A \in \XXA .
        \]
        In particular $t(y \mid x) > 0$ for \emph{every} pair $(x,y) \in \XX\times\XX$.
\end{enumerate}
\end{assumption}

\begin{remark}[The two normalisations]\label{rem:normalisation}
If $T$ has a density $t$ as in \hyp{D}{D}, then the invariance \hyp{A3}{A3} of
$\pi$ is precisely the statement that $\int t(y \mid x)\,\pi(\dd x) = 1$ for
$\aeevery$ $y$ --- indeed $\pi(A) = (T\circ\pi)(A)
= \int_{A}\bigl(\int t(y\mid x)\pi(\dd x)\bigr)\pi(\dd y)$ for every $A \in \XXA$ ---
whereas the kernel property of $T$ is the statement that
$\int t(y \mid x)\,\pi(\dd y) = 1$ for every $x$: the density is ``doubly
stochastic'', the rows normalising because $T$ is a kernel and the columns because
$\pi$ is invariant. Lemma \ref{lem:domination} is one consequence of the first of
these two normalisations. We shall not need this reformulation, but it explains the
structure of what follows.
\end{remark}

\begin{remark}[On joint measurability]\label{rem:jointmeasurability}
Suppose that $\XXA$ is countably generated, as it is when $(\XX,\XXA)$ is standard
Borel. If one assumes
only that $T_{x} \ll \pi$ for every $x \in \XX$, a jointly measurable
version of the density can be constructed: choose a refining sequence of countable
measurable partitions generating $\XXA$, form the ratios $T(P \mid x)/\pi(P)$ over
the partition cells $P$, and pass to the limit, which exists $\aewhere$ by the
martingale convergence theorem and is jointly measurable as a pointwise limit of
jointly measurable maps. We do not use this construction and simply assume \hyp{D}{D};
it is written out and proved, in the form needed for a general iterate and for the
absolutely continuous \emph{part} rather than the whole of $T^{n}_{x}$, in
Proposition \ref{prop:jointdensity}, which supplies \hyp{J}{J} and with it, through
Proposition \ref{prop:LimpliesP}, the passage from \hyp{L}{L} to \hyp{P}{P}.
Measurable Radon--Nikodym derivatives and the Lebesgue decomposition of a kernel are
treated in \cite{Douc18} and \cite{MeynTweedie09}; see also
Remark \ref{rem:referencegen}, where the whole issue disappears once a reference
measure is available.
\end{remark}

\begin{remark}[Almost sure positivity suffices]\label{rem:aepositivity}
Theorem \ref{thm:main} and its proof remain valid word for word if, in \hyp{D}{D},
the map $t$ is only assumed to be an $\XXA\otimes\XXA$-measurable map
$\XX\times\XX \to [0,\infty)$ with
$(\pi\otimes\pi)(\{(x,y) : t(y \mid x) = 0\}) = 0$, the identity
$T(A\mid x) = \int_{A}t(y\mid x)\pi(\dd y)$ still being required for \emph{every}
$x \in \XX$. Indeed positivity enters the proof of Theorem \ref{thm:main} only
through hypothesis (iii) of Proposition \ref{prop:minorisation}, which follows from
the displayed condition by Fubini's theorem, while Lemma \ref{lem:onestep} uses no
positivity at all. We refer to this as the \emph{almost everywhere form} of
\hyp{D}{D}. It should be noted that Remark \ref{rem:CK} below, which derives
\hyp{E}{E} from \hyp{D}{D}, does use positivity at every point.
\end{remark}

\subsection{The convergence theorem under \texorpdfstring{\hyp{A}{A} and \hyp{D}{D}}{(A) and (D)}}

The verification of \hyp{S}{S} needs one observation, and only this one.

\begin{lemma}[One step produces a density]\label{lem:onestep}
Assume \hyp{A}{A} and \hyp{D}{D}. Then $T\circ\mu \ll \pi$ for every
$\mu \in \Prob$; the density is given by the map
\[
  f_{\mu} : \XX \longrightarrow [0,\infty], \qquad
  f_{\mu}(y) := \int_{\XX} t(y \mid x)\, \mu(\dd x) .
\]
\end{lemma}

\begin{proof}
Since $t$ is $\XXA\otimes\XXA$-measurable and nonnegative, Tonelli's theorem gives,
for every $A \in \XXA$,
\[
  (T\circ\mu)(A) = \int_{\XX} T(A \mid x)\,\mu(\dd x)
  = \int_{\XX}\Bigl( \int_{A} t(y \mid x)\,\pi(\dd y) \Bigr)\mu(\dd x)
  = \int_{A} f_{\mu} \dd\pi . \qedhere
\]
\end{proof}

\noindent
This is the only consequence of the density needed for \hyp{S}{S}: after a single
step every initial law, however singular, acquires a density with respect to $\pi$.

\begin{theorem}[Convergence under a strictly positive density]\label{thm:main}
Assume \hyp{A}{A} and \hyp{D}{D}. Then
\[
  \lim_{n\to\infty}\ \sup_{A \in \XXA} \bigl| (T^{n}\circ\mu)(A) - \pi(A) \bigr| = 0
  \qquad\text{for every } \mu \in \Prob,
\]
and $\pi$ is the unique invariant probability measure of $T$. In particular
$\tv{T^{n}_{x} - \pi} \to 0$ for every point $x \in \XX$, since
$T^{n}_{x} = T^{n}\circ\delta_{x}$.
\end{theorem}

\noindent
This is the discrete-time version of Doob's theorem; for the continuous-time
statement, and for its use in infinite dimensions, see \cite{DaPratoZabczyk96}.

\begin{proof}
By Theorem \ref{thm:criterion} it suffices to verify \hyp{P}{P} and \hyp{S}{S}.

\emph{Assumption \hyp{P}{P}.} Let $\varepsilon \in (0,1)$ be a number and take
$N := 1$, $\YY := \YY' := \XX$ and $s := t$, none of which depends on $\varepsilon$.
Hypothesis (a) holds because $\pi(\XX\setminus\XX) = 0$; hypothesis (b) holds with
equality by \hyp{D}{D}; hypothesis (c) holds because $t(y \mid x) > 0$ for every pair
$(x,y)$, so that the set occurring in it is empty and its measure is $0 \le \varepsilon$.

\emph{Assumption \hyp{S}{S}.} By Lemma \ref{lem:onestep} the probability measure
$T^{n}\circ\mu$ is absolutely continuous with respect to $\pi$ for every
$\mu \in \Prob$ and every $n \in \N$, so $\sing(T^{n}\circ\mu\mid\pi) = 0$ by
Definition \ref{def:singmass}. Alternatively, apply
Lemma \ref{lem:uniformminorant} with $N := 1$, $u := t$ and $c := 1$.
\end{proof}

\noindent
One should note what the two verifications mean: under \hyp{D}{D}
\emph{one} step already both mixes ($N=1$ in \hyp{P}{P}) and smooths
($\sing(T\circ\mu\mid\pi) = 0$ in \hyp{S}{S}). The two sections that follow keep these same
two hypotheses and
pay for them differently: Section \ref{sec:mh} gives up exact absolute continuity and
recovers \hyp{S}{S} only approximately, while Section \ref{sec:general} gives up
$N = 1$ and pays with a number of steps depending on $M$, respectively on $\mu$.

\subsection{The convergence theorem under \texorpdfstring{\hyp{A}{A} and \hyp{D}{D}}{(A) and (D)} with reference measure}

\begin{corollary}[Convergence, stated with a reference measure; Theorem \ref{thm:posdens}]\label{cor:reference}
Let $\XX$ be a set, $\XXA$ a $\sigma$-algebra on $\XX$, and $\lambda$ a
$\sigma$-finite measure on $(\XX,\XXA)$. Let
\[
  p : \XX \longrightarrow [0,\infty)
  \qquad\text{and}\qquad
  \tau : \XX\times\XX \longrightarrow [0,\infty), \quad (x,y) \longmapsto \tau(y \mid x),
\]
be measurable, respectively $\XXA\otimes\XXA$-measurable, and define the probability
measure $\pi$ and the Markov kernel $T$ by
\[
  \pi(A) := \int_{A} p \dd\lambda,
  \qquad
  T(A \mid x) := \int_{A} \tau(y \mid x)\,\lambda(\dd y)
  \qquad (A \in \XXA,\ x \in \XX),
\]
where $\int_{\XX} p \dd\lambda = 1$ and $\int_{\XX}\tau(y \mid x)\,\lambda(\dd y) = 1$
for every $x \in \XX$. Write $\Xp := \{ y \in \XX : p(y) > 0 \} \in \XXA$ and assume:
\begin{enumerate}[label=(\roman*),nosep,topsep=3pt]
  \item \emph{(invariance)} $\displaystyle\int_{\XX} \tau(y \mid x)\,p(x)\,\lambda(\dd x) = p(y)$ for $\lambda$-almost every $y \in \XX$;
  \item \emph{(the support is not left)} for every $x \in \XX$: $\ \tau(y \mid x) = 0$ for $\lambda$-almost every $y \in \XX\setminus\Xp$;
  \item \emph{(positivity)} $\tau(y \mid x) > 0$ for $(\lambda\otimes\lambda)$-almost every $(x,y) \in \Xp \times \Xp$.
\end{enumerate}
Then $\pi$ is the unique invariant probability measure of $T$, and
\[
  \lim_{n\to\infty} \tv{T^{n}\circ\mu - \pi} = 0 \qquad\text{for every } \mu \in \Prob .
\]
Theorem \ref{thm:posdens} of the introduction is the special case in which $p > 0$
everywhere, so that $\Xp = \XX$ and hypothesis (ii) is vacuous.
\end{corollary}

\begin{proof}
Hypothesis \hyp{A1}{A1} holds by assumption and \hyp{A2}{A2} because
$\int p \dd\lambda = 1$. For \hyp{A3}{A3}, Tonelli's theorem gives, for
$A \in \XXA$,
\[
  (T\circ\pi)(A) = \int_{\XX}\Bigl(\int_{A}\tau(y\mid x)\lambda(\dd y)\Bigr) p(x)\lambda(\dd x)
  = \int_{A}\Bigl(\int_{\XX}\tau(y \mid x)\,p(x)\,\lambda(\dd x)\Bigr)\lambda(\dd y),
\]
which equals $\int_{A} p \dd\lambda = \pi(A)$ by (i).

Define $t : \XX\times\XX \to [0,\infty)$ by $t(y \mid x) := \tau(y \mid x)/p(y)$ for
$y \in \Xp$ and $t(y \mid x) := 0$ for $y \notin \Xp$; this map is
$\XXA\otimes\XXA$-measurable, since $p > 0$ on $\Xp$. By (ii), for every $x \in \XX$
and $A \in \XXA$,
\[
  \int_{A} t(y \mid x)\,\pi(\dd y) = \int_{A \cap \Xp} \tau(y \mid x)\,\lambda(\dd y)
  = \int_{A} \tau(y \mid x)\,\lambda(\dd y) = T(A \mid x),
\]
so $t$ is a jointly measurable density of $T$ with respect to $\pi$. It is here that
(ii) is used, and it is used for \emph{every} $x$: invariance (i) by itself yields
$\tau(y \mid x) = 0$ for $\lambda$-almost every $y \notin \Xp$ only for
$\lambda$-almost every $x \in \Xp$, that is, for $\pi$-almost every $x$, whereas the
conclusion is asserted for every initial distribution $\mu$, in particular for
$\mu := \delta_{x}$ with $p(x) = 0$. Finally
$(\pi\otimes\pi)(\{(x,y) : t(y \mid x) = 0\}) = \int\!\!\int \ind_{\{\tau(y\mid x)=0\}}\,
p(x)p(y)\,\lambda(\dd x)\lambda(\dd y)$, which vanishes precisely under (iii), since
$p > 0$ exactly on $\Xp$. Thus \hyp{D}{D} holds in the almost-everywhere form of
Remark \ref{rem:aepositivity}, and Theorem \ref{thm:main} applies.
\end{proof}

\begin{remark}[The usual special case, and the remaining dictionary]\label{rem:reference}
If $p > 0$ $\lambda$-almost everywhere --- equivalently $\pi \sim \lambda$ --- then
$\lambda(\XX\setminus\Xp) = 0$, hypothesis (ii) of Corollary \ref{cor:reference} is
vacuous and (iii) reads simply: $\tau > 0$ $(\lambda\otimes\lambda)$-almost
everywhere. In words: the chain has a transition density with respect to $\lambda$
which is almost everywhere strictly positive, and $\pi$ is invariant. That is the
form in which the hypothesis is usually met, and it is the form used in
Section \ref{sec:mh}.

Two further translations are worth recording. The pair of identities
$\int \tau(y \mid x)\lambda(\dd y) = 1$ $(x \in \XX)$ and (i) of
Corollary \ref{cor:reference} is the $\lambda$-form of the double normalisation of
Remark \ref{rem:normalisation}. And a probability measure
$\mu(A) = \int_{A} m \dd\lambda$, given by a measurable map
$m : \XX \to [0,\infty)$, satisfies $\mu \le M\pi$ for a number $M \in [1,\infty)$ if
and only if $m \le M\,p$ $\lambda$-almost everywhere.
\end{remark}

\begin{remark}[A uniformly positive density: geometric rate]\label{rem:uniform}
Suppose, instead of \hyp{D}{D}, that there is a number $\varepsilon_{0} \in (0,1]$ with
$t(y \mid x) \ge \varepsilon_{0}$ for all $(x,y) \in \XX\times\XX$. Then for every
$\mu \in \Prob$ and every $A \in \XXA$,
\[
  (T\circ\mu)(A) = \int_{\XX}\Bigl(\int_{A} t(y\mid x)\,\pi(\dd y)\Bigr)\mu(\dd x)
  \;\ge\; \varepsilon_{0}\,\pi(A),
\]
so the measure $\zeta := \varepsilon_{0}\pi$ is a common minorant of $T\circ\mu$
and $T\circ\nu$ for \emph{all} $\mu,\nu \in \Prob$, dominated or not. By
Lemma \ref{lem:minorant}, $\tv{T\circ\mu - T\circ\nu} \le 1-\varepsilon_{0}$;
running the renormalisation of Proposition \ref{thm:dominated} without the domination
step --- that is, applying the last display to $\alpha_{n} := h_{n}^{+}/\Delta_{n}$
and $\beta_{n} := h_{n}^{-}/\Delta_{n}$ whenever $\Delta_{n} > 0$, the case
$\Delta_{n} = 0$ being trivial --- gives
$\Delta_{n+1} \le (1-\varepsilon_{0})\Delta_{n}$ and hence
\[
  \tv{T^{n}\circ\mu - \pi} \le (1-\varepsilon_{0})^{n} \qquad (n \in \Nz).
\]
This is the classical Doeblin situation; for what can be extracted from a
minorisation valid only on a small set, rather than on all of $\XX$, see
\cite{MeynTweedie09} and \cite{HairerMattingly11}, and for a catalogue of conditions
equivalent to geometric ergodicity see \cite{Gallegos24}. The content of these notes is
precisely that convergence survives, without a rate, when the uniform bound
$\varepsilon_{0}$ is replaced by pointwise positivity, and even when positivity is
reached only after finitely many steps.
\end{remark}

\section{Convergence under a positive transition density plus an atom}\label{sec:mh}

Assumption \hyp{D}{D} requires $T_{x}$ to be absolutely continuous with respect
to $\pi$, and this already excludes the single most important example, namely the
Metropolis--Hastings algorithm \cite{Metropolis53,Hastings70}. Its kernel rejects the
proposed move with a positive probability $r(x)$, and therefore keeps an atom at the
starting point at every time and in every iterate; when $\pi$ is atomless, no iterate
is absolutely continuous. The core of Section \ref{sec:core} covers such kernels
nevertheless, and the verification is short. The reason is that
Assumption \hyp{S}{S} asks only that the singular mass be \emph{small}, not that it
vanish, and that the atom, although never absent, carries a mass which tends to $0$.

\subsection{The assumption \texorpdfstring{\hyp{M}{M}}{(M)}}

\begin{assumption}[Absolutely continuous part plus an atom]\label{ass:M}
In addition to \hyp{A}{A}:
\begin{enumerate}[nosep,topsep=3pt,leftmargin=3em]
  \item[\hyptgt{M}{M}] There is a $\XXA\otimes\XXA$-measurable map
        \[
          k : \XX \times \XX \longrightarrow [0,\infty), \qquad (x,y) \longmapsto k(y \mid x),
        \]
        such that, with the maps
        \[
          \theta : \XX \to [0,\infty], \quad \theta(x) := \int_{\XX} k(y \mid x)\,\pi(\dd y),
          \qquad
          r : \XX \to [-\infty,1], \quad r(x) := 1 - \theta(x),
        \]
        the following hold:
        \begin{enumerate}[label=(\alph*),nosep,topsep=2pt,leftmargin=2em]
          \item $\theta(x) \le 1$ for every $x \in \XX$ --- so that in fact
                $\theta : \XX \to [0,1]$ and $r : \XX \to [0,1]$ --- and
                $T(A \mid x) = \displaystyle\int_{A} k(y \mid x)\,\pi(\dd y) + r(x)\,\ind_{A}(x)$ for every $x \in \XX$ and every $A \in \XXA$;
          \item $\theta(x) > 0$ for \emph{every} $x \in \XX$;
          \item $k(y \mid x) > 0$ for $(\pi\otimes\pi)$-almost every $(x,y) \in \XX\times\XX$.
        \end{enumerate}
\end{enumerate}
\end{assumption}

\begin{remark}[The bound $\theta \le 1$ is a hypothesis, not a consequence]\label{rem:thetabound}
The requirement $\theta \le 1$ in \hyp{M}{M}(a) does not follow from the
representation of $T$ alone. Take $\XX := \{x\}$ a single point, $\pi := \delta_{x}$
and $k(x \mid x) := 2$; then $\theta(x) = 2$, $r(x) = -1$, and
$\int_{A}k(y\mid x)\pi(\dd y) + r(x)\ind_{A}(x) = 2\ind_{A}(x) - \ind_{A}(x)
= \ind_{A}(x) = T(A\mid x)$ for the Markov kernel $T := \delta_{x}$. So the
representation holds while $\theta > 1$. The bound is genuinely used below: it is what
makes $r$ nonnegative, hence what makes $T(A\mid x) \ge \int_{A}k(y\mid x)\pi(\dd y)$
in the verification of \hyp{P}{P}, and what makes the set functions $\nu_{n}$ in
the verification of \hyp{S}{S} nonnegative measures. If singletons are measurable
and $\pi$ is atomless --- the situation the assumption is designed for --- the bound
is automatic, since then $\pi(\{x\}) = 0$ and hence
$\theta(x) = \int_{\XX\setminus\{x\}} k(y \mid x)\pi(\dd y)
= T(\XX \setminus \{x\} \mid x) \le 1$, the term $r(x)\ind_{A}(x)$ contributing
nothing to $A := \XX\setminus\{x\}$.
\end{remark}

\noindent
Condition (b) says that from every single point the chain has a positive chance of
moving; without it a point $x$ with $\theta(x) = 0$ would be absorbing and
$T^{n}\circ\delta_{x} = \delta_{x}$ for every $n$.

\noindent
In the applications, invariance \hyp{A3}{A3} is not checked directly but obtained from
a symmetry. The following lemma is stated so as to \emph{produce} \hyp{A3}{A3}, and
therefore assumes only \hyp{A1}{A1}, \hyp{A2}{A2} and the representation of
\hyp{M}{M}(a); it does not presuppose the invariance that the rest of this section
assumes.

\begin{lemma}[Symmetry implies invariance and reversibility]\label{lem:symmetry}
Assume \hyp{A1}{A1} and \hyp{A2}{A2}, let $T$ be a Markov kernel on $(\XX,\XXA)$ and
let $k : \XX\times\XX \to [0,\infty)$ be a $\XXA\otimes\XXA$-measurable map such that,
with $\theta(x) := \int_{\XX}k(y \mid x)\pi(\dd y)$ and $r := 1-\theta$, one has
$\theta \le 1$ and
$T(A \mid x) = \int_{A} k(y \mid x)\pi(\dd y) + r(x)\ind_{A}(x)$ for every
$x \in \XX$ and every $A \in \XXA$. Assume in addition
that $k(y \mid x) = k(x \mid y)$ for $(\pi\otimes\pi)$-almost every $(x,y)$. Then
$\pi$ is invariant for $T$, i.e.\ \hyp{A3}{A3} holds, and $T$ is reversible with respect to
$\pi$.
\end{lemma}

\begin{proof}
Let $A \in \XXA$. By Fubini's theorem the $(\pi\otimes\pi)$-almost everywhere
symmetry of $k$ implies that for $\aeevery$ $y$ one has $k(y\mid x) = k(x \mid y)$ for
$\aeevery$ $x$, whence $\int_{\XX}k(y \mid x)\pi(\dd x) = \int_{\XX}k(x\mid y)\pi(\dd x)
= \theta(y)$ for $\aeevery$ $y$. Therefore, by Tonelli's theorem,
\begin{align*}
  (T\circ\pi)(A)
  &= \int_{A} \Bigl( \int_{\XX} k(y \mid x)\,\pi(\dd x) \Bigr)\pi(\dd y) + \int_{A} r \dd\pi \\
  &= \int_{A} \theta \dd\pi + \int_{A} (1-\theta) \dd\pi
   = \pi(A) .
\end{align*}
Reversibility is the symmetry of the measure
\[
  \pi(\dd x)\,T(\dd y \mid x)
  = k(y\mid x)\,\pi(\dd x)\,\pi(\dd y) + r(x)\,\pi(\dd x)\,\delta_{x}(\dd y)
\]
on $(\XX\times\XX, \XXA\otimes\XXA)$, the second summand being the measure
$C \mapsto \int_{\XX} r(x)\ind_{C}(x,x)\,\pi(\dd x)$, which is well defined because the
diagonal map $x \mapsto (x,x)$ is measurable. Both summands are invariant under the
swap $(x,y)\mapsto(y,x)$, the first because $k$ is $(\pi\otimes\pi)$-almost everywhere
symmetric and the second because $\ind_{C}(x,x)$ is.
\end{proof}

\noindent
Lemma \ref{lem:symmetry} is not invoked in any proof below. It is recorded for the
reader who wishes to \emph{produce} \hyp{A3}{A3} from a symmetry of $k$, which is how
invariance is verified in practice and which is its evident purpose; the corresponding
statement in terms of a reference measure, which is the form actually used in the
applications, is Remark \ref{rem:detailedbalance}.

\subsection{The convergence theorem under \texorpdfstring{\hyp{A}{A} and \hyp{M}{M}}{(A) and (M)}}

\begin{theorem}[Convergence for a strictly positive density plus an atom]\label{thm:mh}
Assume \hyp{A}{A} and \hyp{M}{M}. Then
\[
  \lim_{n\to\infty}\ \sup_{A\in\XXA} \bigl| (T^{n}\circ\mu)(A) - \pi(A) \bigr| = 0
  \qquad\text{for every } \mu \in \Prob,
\]
and $\pi$ is the unique invariant probability measure of $T$. More precisely, the
singular mass is bounded by the probability that no move has yet been made:
\[
  \sing\bigl( T^{n}\circ\mu \bigm| \pi\bigr) \;\le\; \int_{\XX} r^{n} \dd\mu
  \qquad (\mu \in \Prob,\ n \in \Nz) ,
\]
and the right-hand side tends to $0$ for every $\mu \in \Prob$.
\end{theorem}

\begin{proof}
By Theorem \ref{thm:criterion} it suffices to verify \hyp{P}{P} and \hyp{S}{S}.

\emph{Assumption \hyp{P}{P}.} Let $\varepsilon \in (0,1)$ be a number and take
$N := 1$, $\YY := \XX$ and $s := k$, none of which depends on $\varepsilon$.
Hypothesis (a) is trivial; hypothesis (b)
holds because $r(x)\ind_{A}(x) \ge 0$, which is where the bound $\theta \le 1$ of
\hyp{M}{M}(a) is used, so that \hyp{M}{M}(a) gives
$T(A\mid x) \ge \int_{A}k(y\mid x)\pi(\dd y)$ --- this is the one place in these
notes where the inequality in \hyp{P}{P}(b) is used, rather than an
equality; and hypothesis (c) follows from \hyp{M}{M}(c) by Fubini's theorem applied
to the $\XXA\otimes\XXA$-measurable set $\{(x,y) \in \XX\times\XX : k(y \mid x) = 0\}$,
which is $(\pi\otimes\pi)$-null by \hyp{M}{M}(c):
\[
  0 = \int_{\XX} \pi\bigl( \{ y \in \XX : k(y \mid x) = 0 \} \bigr)\,\pi(\dd x)
    = \int_{\XX} \pi\bigl( \{ x \in \XX : k(y \mid x) = 0 \} \bigr)\,\pi(\dd y) ,
\]
so that the inner term on the right vanishes for $\aeevery$ $y \in \XX$. Taking for
$\YY'$ the full-measure set of such $y$, this is \hyp{P}{P}(a) and (c) with
$\YY = \XX$, the bound in (c) holding with $0$ in place of $\varepsilon$.

\emph{Assumption \hyp{S}{S}.} Fix $\mu \in \Prob$. We claim that for every
$n \in \Nz$
\[
  T^{n}\circ\mu = \nu_{n} + \sigma_{n},
  \qquad\text{where}\qquad
  \nu_{n}(A) := \int_{A} r^{n} \dd\mu \quad (A \in \XXA)
\]
and $\sigma_{n} : \XXA \to [0,\infty)$ is a nonnegative measure with
$\sigma_{n} \ll \pi$. For $n=0$ this
holds with $\sigma_{0} := 0$. Assume it for $n$. Then
$T^{n+1}\circ\mu = T\circ\nu_{n} + T\circ\sigma_{n}$, and:
\begin{itemize}[nosep,topsep=2pt,leftmargin=1.4em]
  \item $T\circ\sigma_{n} \ll \pi$ by Lemma \ref{lem:monotone}(i) (applied to
        $\sigma_{n}/\sigma_{n}(\XX)$ if $\sigma_{n} \ne 0$);
  \item by \hyp{M}{M}(a) and Tonelli's theorem, for $A \in \XXA$,
        \[
          (T\circ\nu_{n})(A)
          = \int_{A} \Bigl( \int_{\XX} k(y\mid x)\,r(x)^{n}\,\mu(\dd x) \Bigr)\pi(\dd y)
            + \int_{A} r^{n+1} \dd\mu ,
        \]
        whose first summand is a nonnegative measure absolutely continuous with
        respect to $\pi$ and whose second summand is $\nu_{n+1}(A)$.
\end{itemize}
This proves the claim with $\sigma_{n+1}$ the sum of the two absolutely continuous
contributions.

Since $\sigma_{n} \le T^{n}\circ\mu$ and $\sigma_{n} \ll \pi$, the measure
$\sigma_{n}$ is admissible in Definition \ref{def:singmass}, whence
\[
  \sing\bigl( T^{n}\circ\mu \bigm| \pi\bigr) \;\le\; 1 - \sigma_{n}(\XX)
  \;=\; \nu_{n}(\XX) \;=\; \int_{\XX} r^{n}\dd\mu
  \qquad (n \in \Nz) .
\]
By \hyp{M}{M}(a) and \hyp{M}{M}(b) we have $0 \le r(x) < 1$ for every $x \in \XX$,
hence $r(x)^{n} \to 0$ for every $x$, and dominated convergence gives
$\int_{\XX}r^{n}\dd\mu \to 0$. This is \hyp{S}{S}.
\end{proof}

\subsection{The convergence theorem under \texorpdfstring{\hyp{A}{A} and \hyp{M}{M}}{(A) and (M)} with reference measure}

\begin{corollary}[Convergence with an atom, stated with a reference measure]\label{cor:mhreference}
Let $(\XX,\XXA)$ be a measurable space and $\lambda$ a $\sigma$-finite measure on
it. Let $p : \XX \to [0,\infty)$ be measurable with $\int_{\XX} p \dd\lambda = 1$, put
$\pi(A) := \int_{A} p \dd\lambda$ and $\Xp := \{p > 0\} \in \XXA$. Let
\[
  \kappa : \XX \times \XX \longrightarrow [0,\infty), \qquad (x,y) \longmapsto \kappa(y \mid x),
\]
be $\XXA\otimes\XXA$-measurable, put
$\theta(x) := \int_{\XX}\kappa(y \mid x)\,\lambda(\dd y)$, assume $\theta(x) \le 1$
for every $x \in \XX$, so that
\[
  T(A \mid x) := \int_{A} \kappa(y \mid x)\,\lambda(\dd y) + \bigl(1-\theta(x)\bigr)\,\ind_{A}(x)
  \qquad (A \in \XXA,\ x \in \XX)
\]
defines a Markov kernel $T$ on $(\XX,\XXA)$. Assume:
\begin{enumerate}[label=(\roman*),nosep,topsep=3pt]
  \item \emph{(the chain can move)} $\theta(x) > 0$ for every $x \in \XX$;
  \item \emph{(invariance)} $\displaystyle\int_{\XX} \kappa(y \mid x)\,p(x)\,\lambda(\dd x) = \theta(y)\,p(y)$ for $\lambda$-almost every $y \in \XX$;
  \item \emph{(the support is not left)} for every $x \in \XX$: $\ \kappa(y \mid x) = 0$ for $\lambda$-almost every $y \in \XX \setminus \Xp$;
  \item \emph{(positivity)} $\kappa(y \mid x) > 0$ for $(\lambda\otimes\lambda)$-almost every $(x,y) \in \Xp\times\Xp$.
\end{enumerate}
Then $\pi$ is the unique invariant probability measure of $T$, and
$\tv{T^{n}\circ\mu - \pi} \to 0$ for every $\mu \in \Prob$. If $p > 0$
$\lambda$-almost everywhere, hypothesis (iii) is vacuous.
\end{corollary}

\begin{proof}
For \hyp{A3}{A3}, Tonelli's theorem gives, for $A \in \XXA$,
\[
  (T\circ\pi)(A)
  = \int_{A}\Bigl( \int_{\XX}\kappa(y \mid x)\,p(x)\,\lambda(\dd x) \Bigr)\lambda(\dd y)
    + \int_{A} \bigl(1-\theta(y)\bigr)\,p(y)\,\lambda(\dd y),
\]
which equals $\int_{A} p \dd\lambda = \pi(A)$ for every $A$ precisely under (ii).

Define $k : \XX\times\XX \to [0,\infty)$ by $k(y \mid x) := \kappa(y \mid x)/p(y)$ for
$y \in \Xp$ and $k(y \mid x) := 0$ for $y \notin \Xp$; this map is
$\XXA\otimes\XXA$-measurable. By (iii), for every $x \in \XX$ and $A \in \XXA$,
\[
  \int_{A} k(y \mid x)\,\pi(\dd y) = \int_{A \cap \Xp}\kappa(y \mid x)\,\lambda(\dd y)
  = \int_{A} \kappa(y \mid x)\,\lambda(\dd y),
\]
so that $T(A \mid x) = \int_{A}k(y\mid x)\pi(\dd y) + (1-\theta(x))\ind_{A}(x)$ and,
taking $A := \XX$, $\int_{\XX} k(y \mid x)\,\pi(\dd y) = \theta(x)$. Thus
\hyp{M}{M}(a) holds with this $k$, and \hyp{M}{M}(b) is hypothesis (i).
Hypothesis \hyp{M}{M}(c) is (iv), because $p > 0$ exactly on $\Xp$ and therefore the
$(\pi\otimes\pi)$-null subsets of $\Xp \times \Xp$ are exactly the
$(\lambda\otimes\lambda)$-null ones. Theorem \ref{thm:mh} applies.
\end{proof}

\begin{remark}[Detailed balance as a sufficient condition]\label{rem:detailedbalance}
Hypothesis (ii) of Corollary \ref{cor:mhreference} is invariance, and nothing more is
needed. It holds in particular under \emph{detailed balance},
\[
  p(x)\,\kappa(y \mid x) = p(y)\,\kappa(x \mid y)
  \qquad\text{for } (\lambda\otimes\lambda)\text{-almost every } (x,y) \in \XX\times\XX ,
\]
since integrating this identity in $x$ against $\lambda$ gives
$\int \kappa(y\mid x)p(x)\lambda(\dd x) = p(y)\int\kappa(x \mid y)\lambda(\dd x)
= p(y)\theta(y)$. Detailed balance is the $\lambda$-form of the symmetry hypothesis
of Lemma \ref{lem:symmetry}, and it is what the Metropolis--Hastings construction
supplies; but the corollary does not require it, and chains that are invariant
without being reversible are covered as well.
\end{remark}

\subsection{The convergence theorem for the Metropolis--Hastings algorithm}

\begin{corollary}[The Metropolis--Hastings algorithm converges; Theorem \ref{thm:mhintro}]\label{cor:mh}
Let $(\XX,\XXA)$ be a measurable space, $\lambda$ a $\sigma$-finite measure on it,
and let the target be $\pi(A) = \int_{A} p \dd\lambda$ for a measurable map
$p : \XX \to (0,\infty)$ with $\int_{\XX} p \dd\lambda = 1$. Let the proposal be the
Markov kernel $Q(A \mid x) = \int_{A} q(y \mid x)\,\lambda(\dd y)$, given by an
$\XXA\otimes\XXA$-measurable map $q : \XX\times\XX \to [0,\infty)$ with
$\int_{\XX} q(y \mid x)\,\lambda(\dd y) = 1$ for every $x \in \XX$, define the
acceptance probability
\[
  a : \XX\times\XX \to [0,1], \qquad
  a(y \mid x) := \min\Bigl\{ 1,\ \frac{p(y)\,q(x \mid y)}{p(x)\,q(y \mid x)} \Bigr\}
  \quad\bigl(:= 1 \text{ when } q(y \mid x) = 0\bigr),
\]
and let $T$ be the Metropolis--Hastings kernel
\[
  T(A \mid x) := \int_{A} a(y \mid x)\,q(y \mid x)\,\lambda(\dd y) + r(x)\,\ind_{A}(x),
  \qquad
  r(x) := 1 - \int_{\XX} a(y \mid x)\,q(y \mid x)\,\lambda(\dd y) .
\]
Assume
\begin{enumerate}[label=(\roman*),nosep,topsep=3pt]
  \item $q(y \mid x) > 0$ for $(\lambda\otimes\lambda)$-almost every $(x,y) \in \XX\times\XX$;
  \item for every $x \in \XX$ the set $\{ y \in \XX : q(y \mid x) > 0 \text{ and } q(x \mid y) > 0 \} \in \XXA$ has positive $\lambda$-measure.
\end{enumerate}
Then $\pi$ is the unique invariant probability measure of $T$, and
$\tv{T^{n}\circ\mu - \pi} \to 0$ for every initial distribution $\mu \in \Prob$.

Both hypotheses hold in either of the following two cases, which are the ones met in
practice and which give Theorem \ref{thm:mhintro}:
\begin{enumerate}[label=(\alph*),nosep,topsep=3pt]
  \item $q(y \mid x) > 0$ for every $(x,y) \in \XX\times\XX$;
  \item $q$ is symmetric, $q(y \mid x) = q(x \mid y)$ for all $(x,y)$, and
        $q(y \mid x) > 0$ for $(\lambda\otimes\lambda)$-almost every $(x,y)$.
\end{enumerate}
\end{corollary}

\begin{proof}
Case (a) implies (i) trivially, and implies (ii) because $\lambda(\XX) > 0$. In case
(b), $\{ y : q(y \mid x) > 0 \text{ and } q(x \mid y) > 0 \} = \{ y : q(y \mid x) > 0 \}$,
which has positive $\lambda$-measure because $\int q(y \mid x)\lambda(\dd y) = 1$.

Apply Corollary \ref{cor:mhreference} with
\[
  \kappa(y \mid x) := a(y \mid x)\,q(y \mid x)
  = \frac{\min\bigl\{ p(x)\,q(y \mid x),\ p(y)\,q(x \mid y) \bigr\}}{p(x)} ,
\]
the second equality holding for every pair $(x,y)$, including those with
$q(y\mid x) = 0$, where both sides vanish by the convention $a := 1$. The map
$\kappa$ is $\XXA\otimes\XXA$-measurable because $p$ and $q$ are measurable and
$p > 0$ everywhere.

The kernel hypothesis of Corollary \ref{cor:mhreference} is met:
$\theta(x) = \int_{\XX}a(y\mid x)q(y\mid x)\lambda(\dd y)
\le \int_{\XX}q(y\mid x)\lambda(\dd y) = 1$ because $a \le 1$, and $\theta = 1-r$ by
the definition of $r$, so that the kernel written there is the kernel $T$ written
here.

Since $p > 0$ everywhere we have $\Xp = \XX$, so hypothesis (iii) there is vacuous.
Detailed balance holds identically, not merely almost
everywhere, because $p(x)\kappa(y \mid x) = \min\{p(x)q(y\mid x),\,p(y)q(x \mid y)\}$
is symmetric in $(x,y)$; by Remark \ref{rem:detailedbalance} this gives hypothesis
(ii) there.

For hypothesis (iv) there, note that, $p$ being strictly positive,
\[
  \kappa(y \mid x) > 0
  \quad\Longleftrightarrow\quad
  q(y \mid x) > 0 \ \text{ and } \ q(x \mid y) > 0 .
\]
Now (i) here says that the set $\{(x,y) : q(y\mid x) = 0\}$ is
$(\lambda\otimes\lambda)$-null; its image under the swap $(x,y)\mapsto(y,x)$ is
$\{(x,y) : q(x \mid y) = 0\}$, and $\lambda\otimes\lambda$ is invariant under that
swap because $\lambda$ is $\sigma$-finite, so this second set is
$(\lambda\otimes\lambda)$-null as well. Hence $\kappa > 0$
$(\lambda\otimes\lambda)$-almost everywhere, which is (iv) there.

Finally, the displayed equivalence shows that
$\theta(x) = \int_{\XX}\kappa(y \mid x)\lambda(\dd y) > 0$ if and only if the set
$\{y : q(y\mid x) > 0 \text{ and } q(x \mid y) > 0\}$ has positive $\lambda$-measure,
so (ii) here gives hypothesis (i) there, namely $\theta(x) > 0$ for every
$x \in \XX$.
\end{proof}

\begin{remark}[Hypothesis (ii) is not a technicality]\label{rem:mhsharp}
Almost everywhere positivity of the proposal density is by itself \emph{not} enough,
and condition (ii) is exactly what excludes the following object. Let $\XX := \R$
with its Borel $\sigma$-algebra, let $\lambda$ be Lebesgue measure, let
$p := \varphi$ be the standard normal density --- so that $p > 0$ everywhere and
$\pi = \varphi\lambda$ --- and put
\[
  q(y \mid x) := \varphi(y-x)\,\ind_{\{y \ne 0\}} \qquad (x,y \in \R) .
\]
Then $q$ is $\XXA\otimes\XXA$-measurable and
$\int_{\R} q(y \mid x)\,\lambda(\dd y) = 1$ for every $x \in \R$, since $\{0\}$ is
$\lambda$-null; so $Q$ is a Markov kernel of the required shape and the
Metropolis--Hastings kernel $T$ built from it is well defined. Hypothesis (i) holds,
because $\{(x,y) : q(y \mid x) = 0\} = \R\times\{0\}$ is
$(\lambda\otimes\lambda)$-null.

But $q(0 \mid y) = 0$ for \emph{every} $y \in \R$. Hence, by the identity for
$\kappa$ in the proof above,
\[
  a(y \mid 0)\,q(y \mid 0)
  = \frac{\min\bigl\{ p(0)\,q(y \mid 0),\ p(y)\,q(0 \mid y) \bigr\}}{p(0)} = 0
  \qquad (y \in \R),
\]
so that $\theta(0) = 0$, $r(0) = 1$ and $T_{0} = \delta_{0}$. The point
$0$ is absorbing; $\delta_{0}$ is therefore a second invariant probability measure,
and $T^{n}\circ\delta_{0} = \delta_{0}$ for every $n$, so that both conclusions of
Corollary \ref{cor:mh} fail from the initial distribution $\mu := \delta_{0}$.

What fails among the hypotheses is (ii), and only (ii): the set
$\{ y : q(y \mid 0) > 0 \text{ and } q(0 \mid y) > 0 \}$ is empty. In the language of
\hyp{M}{M} this is a starting point with $\theta(0) = 0$, which is what \hyp{M}{M}(b)
forbids and why it does so. The example also shows why Theorem \ref{thm:mhintro} of the
introduction asks for $q > 0$ at \emph{every} point rather than almost everywhere,
and it is ruled out by case (b) of Corollary \ref{cor:mh} as well, the map $q$ above
not being symmetric. Between them, the two cases (a) and (b) cover what occurs in
practice.
\end{remark}

\begin{remark}[Reading the hypotheses]\label{rem:mhkernel}
Condition (i) of Corollary \ref{cor:mh} is the substantive one: the proposal must be
able to move, in one step, to almost every point of the state space, and the reverse
move must be possible as well. Condition (ii) merely excludes starting points from
which every proposal is rejected with probability one. The construction goes back to
\cite{Metropolis53} and \cite{Hastings70}; its measure-theoretic formulation on a
general state space, and the standard convergence results for it, are due to Tierney
\cite{Tierney94}, and a detailed survey is given by Roberts and Rosenthal
\cite{RobertsRosenthal04}.
\end{remark}

\begin{remark}[The jump chain, and why we did not use it]\label{rem:jumpchain}
A natural alternative strategy is to remove the atom by conditioning on a move being
made, and then to appeal to Theorem \ref{thm:main} for the resulting kernel. Assume
\hyp{M}{M} and that $k$ is symmetric, and put $\bar\theta := \int_{\XX}\theta\dd\pi \in (0,1]$, a
number which is positive by \hyp{M}{M}(b). Define
\[
  \tilde T(A \mid x) := \frac{1}{\theta(x)} \int_{A} k(y \mid x)\,\pi(\dd y),
  \qquad
  \tilde\pi(A) := \frac{1}{\bar\theta}\int_{A} \theta \dd\pi
  \qquad (A \in \XXA,\ x \in \XX).
\]
Then $\tilde T$ is a Markov kernel, $\tilde\pi \in \Prob$ is invariant for
$\tilde T$ and $\tilde T$ is reversible with respect to $\tilde\pi$, because
$\tilde\pi(\dd x)\tilde T(\dd y\mid x) = \bar\theta^{-1}k(y \mid x)\,\pi(\dd x)\pi(\dd y)$ is
symmetric. Moreover $\tilde T$ has a density with respect to $\tilde\pi$, namely
\[
  \tilde t(y \mid x) = \frac{\bar\theta\,k(y\mid x)}{\theta(x)\,\theta(y)} ,
\]
which is jointly measurable and strictly positive $(\pi\otimes\pi)$-almost
everywhere, hence also $(\tilde\pi\otimes\tilde\pi)$-almost everywhere: indeed
$\tilde\pi = \bar\theta^{-1}\theta\,\pi$ has a density with respect to $\pi$ that is
strictly positive everywhere, so $\tilde\pi \sim \pi$ and the two product measures
have the same null sets. Hence Theorem \ref{thm:main}, in the almost everywhere form
of Remark \ref{rem:aepositivity} --- applied with $\tilde\pi$ in the role of $\pi$ ---
applies to the pair $(\tilde T,\tilde\pi)$ and gives
$\tv{\tilde T^{n}\circ\nu - \tilde\pi} \to 0$ for every $\nu \in \Prob$.

This is exactly the chain of accepted moves, and the reweighting is the expected one:
$\tilde\pi = \bar\theta^{-1}\theta\pi$, that is, $\pi$ is recovered from $\tilde\pi$ by
weighting with the mean holding time $1/\theta(x)$ at the point $x$. What the
construction does \emph{not} deliver for free is the transfer back. The original
chain arises from the jump chain by a random time change with state-dependent holding
times: $X_{n} = Y_{J_{n}}$, where $(Y_{j})_{j\in\Nz}$ is the jump chain and $J_{n}$
counts the accepted moves up to time $n$. The number $J_{n}$ is not independent of the
sequence $(Y_{j})$, so deducing $\tv{T^{n}\circ\mu-\pi} \to 0$ from
$\tv{\tilde T^{j}\circ\nu - \tilde\pi} \to 0$ requires a renewal argument controlling
the joint behaviour of $J_{n}$ and $Y_{J_{n}}$. That argument is longer than the
direct verification of \hyp{S}{S} in the proof of Theorem \ref{thm:mh}, which is
why we did not take this route; for chains observed at random times and the
associated renewal theory see \cite{Douc18} and \cite{MeynTweedie09}.
\end{remark}

\section{Convergence under an eventually positive transition density}\label{sec:general}

We now return to absolutely continuous kernels and weaken \hyp{D}{D} in the other
direction: instead of requiring that $T_{x}$ itself have a strictly
positive density, we require only that \emph{some} iterate
$T^{n}_{x}$ have one, where the number $n$ may depend on the point $x$.
In the general theory this is the point at which irreducibility, aperiodicity and
Harris recurrence are introduced; see \cite{Nummelin84}, \cite{MeynTweedie09} and
\cite{Douc18}. Here no aperiodicity hypothesis has to be added: the condition
improves itself, as the first subsection below shows.

\subsection{The condition improves itself}

The two ingredients were proved above: absolute continuity
propagates forwards (Lemma \ref{lem:monotone}(i)) and so does domination of $\pi$
(Lemma \ref{lem:monotone}(ii)). Together they give the following, which is the form in
which the mutual relation $\sim$ is needed here.

\begin{corollary}[Self-improvement]\label{cor:selfimprovement}
Assume \hyp{A}{A} and put
$\XX^{(n)} := \{ x \in \XX : T^{n}_{x} \sim \pi \}$ for $n \in \N$. Then
$\XX^{(n)} \subseteq \XX^{(n+1)}$ for every $n \in \N$. Hence the condition
``for every $x \in \XX$ there is a number $n \in \N$ with
$T^{n}_{x} \sim \pi$'' already implies ``for every $x \in \XX$ there is a
number $n_{0}(x) \in \N$ such that $T^{n}_{x} \sim \pi$ for \emph{every}
$n \ge n_{0}(x)$''.
\end{corollary}

\begin{proof}
Apply Lemma \ref{lem:monotone}(ii) to $\nu := T^{n}_{x}$ and use
$T \circ T^{n}_{x} = T^{n+1}_{x}$.
\end{proof}

\noindent
Corollary \ref{cor:selfimprovement} says that the property
$T^{n}_{x} \sim \pi$, once acquired, is never lost. This is what makes an
aperiodicity hypothesis unnecessary below, at a cost of two lines.

\subsection{The assumption \texorpdfstring{\hyp{E}{E}}{(E)}}

\begin{assumption}[A strictly positive density after finitely many steps]\label{ass:A4prime}
In addition to \hyp{A}{A}:
\begin{enumerate}[nosep,topsep=3pt,leftmargin=3.6em]
  \item[\hyptgt{E}{E}] There are sets $\XX_{n} \in \XXA$ $(n \in \N)$ and
        $\XXA\otimes\XXA$-measurable maps
        \[
          t_{n} : \XX \times \XX \longrightarrow [0,\infty), \qquad (x,y) \longmapsto t_{n}(y \mid x)
          \qquad (n \in \N),
        \]
        such that
        \begin{enumerate}[label=(\alph*),nosep,topsep=2pt,leftmargin=2em]
          \item $\XX_{n} \subseteq \XX_{n+1}$ for every $n \in \N$, and $\bigcup_{n\in\N}\XX_{n} = \XX$;
          \item $T^{n}(A \mid x) = \displaystyle\int_{A} t_{n}(y \mid x)\,\pi(\dd y)$ for every $n \in \N$, every $x \in \XX_{n}$ and every $A \in \XXA$;
          \item for every $n \in \N$ and every $x \in \XX_{n}$: $\ t_{n}(y \mid x) > 0$ for $\aeevery$ $y \in \XX$.
        \end{enumerate}
\end{enumerate}
\end{assumption}

\begin{remark}[Relation to the informal condition]\label{rem:measurableXn}
The mathematical content of \hyp{E}{E} is exactly the condition ``for every $x \in \XX$
there is a number $n=n(x) \in \N$ with $T^{n}_{x} \sim \pi$''; by
Corollary \ref{cor:selfimprovement} the sets $\XX^{(n)}$ associated with that
condition are automatically increasing with union $\XX$, which is (a). What \hyp{E}{E}
adds is the measurability of the sets $\XX_{n}$ together with the existence of
jointly measurable densities $t_{n}$ on them. Both are automatic under \hyp{J}{J}, the
sets $\XX_{n}$ being measurable by Tonelli's theorem (Remark \ref{rem:LviaJ}); and
\hyp{J}{J} is itself automatic when $\XXA$ is countably generated, by the measurable
Lebesgue decomposition of a kernel --- the martingale construction of
Remark \ref{rem:jointmeasurability} applied to $T^{n}$. But that construction
is genuine extra machinery, and we prefer to assume its conclusion, exactly as \hyp{D}{D}
does for $n = 1$. The construction is carried out in Proposition \ref{prop:jointdensity}, and
the informal condition itself is the hypothesis \hyp{L}{L} of
Subsection \ref{ssec:L} strengthened in three ways: from $\aeevery$ $x$ to every $x$,
from a vanishing limit to exact domination at a finite time, and from
$T^{n}_{x} \gg \pi$ to $T^{n}_{x} \sim \pi$; so
Theorem \ref{thm:asympequiv} yields Theorem \ref{thm:general} directly from the informal
condition, without the technical parts of \hyp{E}{E}, as soon as \hyp{J}{J} holds ---
which is exactly the remaining, purely measure-theoretic, part of what \hyp{E}{E}
assumes (Remark \ref{rem:Lverification}(iii)).
\end{remark}

\subsection{The convergence theorem under \texorpdfstring{\hyp{A}{A} and \hyp{E}{E}}{(A) and (E)}}

\begin{theorem}[Convergence under a strictly positive density after finitely many steps]\label{thm:general}
Assume \hyp{A}{A} and \hyp{E}{E}. Then
\[
  \lim_{n\to\infty}\ \sup_{A \in \XXA} \bigl| (T^{n}\circ\mu)(A) - \pi(A) \bigr| = 0
  \qquad\text{for every } \mu \in \Prob,
\]
and $\pi$ is the unique invariant probability measure of $T$. More precisely, the
singular mass is bounded by the mass that has not yet entered the good sets:
\[
  \sing\bigl( T^{n}\circ\mu \bigm| \pi\bigr) \;\le\; 1 - \mu(\XX_{n})
  \qquad (\mu \in \Prob,\ n \in \N) ,
\]
and the right-hand side tends to $0$ for every $\mu \in \Prob$.
\end{theorem}

\begin{proof}
Since $\XX_{n} \uparrow \XX$, continuity from below gives
\[
  \lim_{n\to\infty}\pi(\XX_{n}) = 1
  \qquad\text{and}\qquad
  \lim_{n\to\infty}\mu(\XX_{n}) = 1 \quad \text{for every } \mu \in \Prob .
\]
By Theorem \ref{thm:criterion} it suffices to verify \hyp{P}{P} and \hyp{S}{S}.

\emph{Assumption \hyp{P}{P}.} Let $\varepsilon \in (0,1)$ be a number. Choose a number
$N \in \N$ with $\pi(\XX \setminus \XX_{N}) \le \varepsilon$ and take the set
$\YY := \XX_{N}$ and the map $s := t_{N}$. Hypothesis (a) holds by the choice of
$N$; hypothesis (b) holds with equality by \hyp{E}{E}(b); hypothesis (c) holds
because, by \hyp{E}{E}(c) and Fubini's theorem applied to the
$\XXA\otimes\XXA$-measurable set
$\{(x,y) \in \XX_{N}\times\XX : t_{N}(y\mid x) = 0\}$,
\[
  0 = \int_{\XX_{N}} \pi\bigl(\{ y \in \XX : t_{N}(y \mid x) = 0 \}\bigr)\,\pi(\dd x)
    = \int_{\XX} \pi\bigl(\{ x \in \XX_{N} : t_{N}(y \mid x) = 0 \}\bigr)\,\pi(\dd y),
\]
so that the inner term on the right vanishes for $\aeevery$ $y$; taking for $\YY'$
the full-measure set of such $y$ gives \hyp{P}{P}(a) and (c), the bound in (c) holding
with $0$ in place of $\varepsilon$. Note that here $N$ depends on $\varepsilon$.

\emph{Assumption \hyp{S}{S}.} Let $\mu \in \Prob$ and let $K \in \N$ be a number, and
define the nonnegative measure
\[
  \beta_{K} : \XXA \to [0,\infty), \qquad
  \beta_{K}(A) := \int_{\XX_{K}} T^{K}(A \mid x)\,\mu(\dd x) .
\]
Then $\beta_{K} \le T^{K}\circ\mu$, and $\beta_{K} \ll \pi$ because, by
\hyp{E}{E}(b) and Tonelli's theorem,
$\beta_{K}(A) = \int_{A}\bigl(\int_{\XX_{K}} t_{K}(y \mid x)\mu(\dd x)\bigr)\pi(\dd y)$.
Since $\beta_{K}(\XX) = \mu(\XX_{K})$, Definition \ref{def:singmass} gives
\[
  \sing\bigl( T^{K}\circ\mu \bigm| \pi\bigr) \;\le\; 1 - \mu(\XX_{K})
  \qquad (K \in \N) ,
\]
and the right-hand side tends to $0$. This is \hyp{S}{S}.
\end{proof}

\begin{remark}[Theorem \ref{thm:main} is the special case of Theorem \ref{thm:general} with $\XX_{n} = \XX$]\label{rem:CK}
Assume \hyp{D}{D}. Then all iterates automatically have jointly measurable, strictly
positive densities, given recursively by the Chapman--Kolmogorov formula
\[
  t_{1} := t, \qquad
  t_{n+1}(y \mid x) := \int_{\XX} t_{n}(y \mid z)\, t(z \mid x)\, \pi(\dd z)
  \qquad (n \in \N).
\]
Indeed, if $t_{n}$ is $\XXA\otimes\XXA$-measurable then so is $t_{n+1}$, by Tonelli's
theorem applied to the $\XXA\otimes\XXA\otimes\XXA$-measurable map
$(x,y,z) \mapsto t_{n}(y \mid z)\,t(z \mid x)$; that $t_{n+1}$ is a density of
$T^{n+1}$ follows from
\[
  T^{n+1}(A \mid x)
  = \int_{\XX} T^{n}(A \mid z)\,t(z \mid x)\,\pi(\dd z)
  = \int_{A} \Bigl( \int_{\XX} t_{n}(y \mid z)\,t(z \mid x)\,\pi(\dd z) \Bigr)\pi(\dd y);
\]
and $t_{n+1} > 0$ everywhere because the integrand is strictly positive everywhere
and $\pi$ is a probability measure. Hence \hyp{D}{D} implies \hyp{E}{E} with
$\XX_{n} := \XX$ for every $n \in \N$, and Theorem \ref{thm:main} is the special case
$\XX_{1} = \XX$ of Theorem \ref{thm:general}.

The consequence for the reading of \hyp{E}{E} is this: joint measurability of the family
$(t_{n})_{n\in\N}$ is \emph{not} an additional hypothesis whenever $T$ itself has a
jointly measurable density, since it is then inherited by convolution. It is a
genuine hypothesis only in the situation \hyp{E}{E} is designed for, namely when $T$
itself has \emph{no} density at all: there is then no $t_{1}$ to convolve with, and
joint measurability of $t_{n}$ must either be assumed or produced by the measurable
Lebesgue decomposition (Remark \ref{rem:measurableXn}).
\end{remark}

\begin{remark}[On a finite state space Theorem \ref{thm:general} is the classical theorem]\label{rem:finite}
Let $\XX$ be a finite set, let $\XXA$ be its power set, let $T$ be a Markov kernel on
$(\XX,\XXA)$ --- that is, a stochastic matrix with entries
$T(\{y\} \mid x) \in [0,1]$ --- and let $\pi \in \Prob$ be invariant for $T$ with
$\pi(\{x\}) > 0$ for every $x \in \XX$, which is automatic when $T$ is irreducible.
Then \hyp{E}{E} holds if and only if $T$ is irreducible and aperiodic. Indeed, since
$\pi$ has full support, $T^{n}_{x} \sim \pi$ says exactly that the
$x$-th row of the matrix $T^{n}$ has all entries strictly positive, and the condition
that for every $x \in \XX$ some such $n \in \N$ exists is equivalent to primitivity
of $T$; that in turn is the classical characterisation of irreducibility together
with aperiodicity for a finite chain. All measurability requirements in \hyp{E}{E} are
vacuous here, since every map on a finite set is measurable. Theorem
\ref{thm:general} restricted to a finite state space is therefore precisely the
classical convergence theorem for finite irreducible aperiodic Markov chains, and the
role played in the classical proof by aperiodicity is played here by
Corollary \ref{cor:selfimprovement}. We omit the proof of the equivalence; it is
the Perron--Frobenius characterisation of primitive stochastic matrices, for which
see \cite{LevinPeres17}, where the finite theory is developed in detail.

The assumption that $\pi$ have full support is not a technicality: without it,
\hyp{E}{E} requires the chain to leave the complement of the support of $\pi$
\emph{completely} after finitely many steps, which is strictly more than the
convergence $\tv{T^{n}\circ\mu - \pi} \to 0$ demands. On $\XX := \{1,2\}$ with
$T(\{1\}\mid 1) = 1$ and $T(\{1\}\mid 2) = T(\{2\}\mid 2) = \tfrac12$ the chain
converges to $\pi = \delta_{1}$, yet $T^{n}(\{2\}\mid 2) = 2^{-n} > 0$ for every
$n \in \N$, so \hyp{E}{E} fails. The theorems of these notes give sufficient conditions,
not necessary ones; for criteria that are necessary as well, see
\cite{ScheutzowSchindler21} and Remark \ref{rem:necessary}.
\end{remark}

\begin{remark}[Necessary versus sufficient, and where \texorpdfstring{\hyp{E}{E}}{(E)} sits]\label{rem:necessary}
Assumption \hyp{E}{E} is a condition of the ``equivalence of transition
probabilities'' type in the classification of Scheutzow and Schindler
\cite{ScheutzowSchindler21}, who determine which conditions of this kind can be
sharpened into criteria that are necessary as well as sufficient for total variation
convergence of all, respectively of $\aeevery$, transition probability. Theorem
\ref{thm:general} gives a sufficient condition only; the criterion behind it,
Theorem \ref{thm:asympequiv}, is necessary as well, so under \hyp{J}{J} it must be
equivalent to their (A$_{1}$) below --- the same characterisation, read against the
target instead of against a second starting point. The necessity half is not quoted
from anywhere: it is Corollary \ref{cor:singtv}, one line from the splitting identity
\eqref{eq:singsplit}, and it needs neither \hyp{J}{J} nor invariance. Its exact place in that
classification is the following chain, in which every implication is strict:
\[
  \hyp{E}{E}
  \ \Longrightarrow\
  \text{(K)}
  \ \Longrightarrow\
  \text{(A}_{1}\text{)}
  \ \Longleftrightarrow\
  \bigl[\ \tv{T^{n}_{x}-\pi} \to 0 \ \text{ for every } x \in \XX\ \bigr] ,
\]
where (K) is the hypothesis of Kulik and Scheutzow \cite[Theorem 1]{KulikScheutzow15}
--- for every pair $x,y \in \XX$ there is a number $n = n_{x,y} \in \N$ with
$T^{n}_{x} \sim T^{n}_{y}$ --- and (A$_{1}$) is the
\emph{asymptotic equivalence} of \cite{ScheutzowSchindler21}, which asks only that
for every pair $x,y$ and every $\varepsilon > 0$ there be an $n$ and a set $A$
carrying at least $1-\varepsilon$ of both $T^{n}_{x}$ and
$T^{n}_{y}$ on which the two are equivalent.

The first implication holds because \hyp{E}{E} together with
Corollary \ref{cor:selfimprovement} supplies, for a prescribed pair $x,y$, a single
$n$ with $T^{n}_{x} \sim \pi \sim T^{n}_{y}$; it is strict
because (K) never mentions $\pi$ and holds, for instance, for a chain confined to a
$\pi$-null set on which it mixes. The second is trivial, with $A := \XX$, and is
strict by an example of \cite{ScheutzowSchindler21}. The equivalence at the right is
their Theorem 2.16. So the conclusion of Theorem \ref{thm:general} \emph{for every
starting point} is characterised by (A$_{1}$), and \hyp{E}{E} is a strictly stronger,
but far more easily checked, sufficient condition; see the second half of
Remark \ref{rem:finite} for a two-state example on which the conclusion holds and
\hyp{E}{E} fails.

Two caveats, which are the reason a self-contained proof is given here rather than a
citation. Both \cite{KulikScheutzow15} and \cite{ScheutzowSchindler21} assume
throughout that $\XXA$ is countably generated and that the diagonal of $\XX\times\XX$
is measurable; neither hypothesis is used in the proof of Theorem \ref{thm:general},
nor anywhere else in these notes. What Subsections \ref{ssec:L} and \ref{ssec:main}
assume in order to restate the criterion is \hyp{J}{J}, a jointly measurable Lebesgue
density for the iterates, which is a hypothesis on the kernel; countable generation
implies it (Proposition \ref{prop:jointdensity}) but is never itself assumed, and the
diagonal is never assumed measurable. And the conclusion characterised by
(A$_{1}$) is convergence from every \emph{point}, from which convergence from every
initial \emph{distribution} --- which is what Theorem \ref{thm:general} asserts and
what an application needs --- is obtained by integrating
$x \mapsto \tv{T^{n}_{x}-\pi}$, a map whose measurability is not
automatic without countable generation (Remark \ref{rem:tvnotmeasurable}).
\end{remark}

\subsection{The convergence theorem under \texorpdfstring{\hyp{A}{A} and \hyp{E}{E}}{(A) and (E)} with reference measure}

\begin{corollary}[Convergence after $n(x)$ steps, stated with a reference measure]\label{cor:referencegen}
Let $(\XX,\XXA)$ be a measurable space and $\lambda$ a $\sigma$-finite measure on
it. Let $p : \XX \to [0,\infty)$ be measurable with $\int_{\XX} p \dd\lambda = 1$, let
$\tau : \XX\times\XX \to [0,\infty)$ be $\XXA\otimes\XXA$-measurable with
$\int_{\XX} \tau(y \mid x)\,\lambda(\dd y) = 1$ for every $x \in \XX$, and define
\[
  \pi(A) := \int_{A} p \dd\lambda, \qquad T(A \mid x) := \int_{A}\tau(y \mid x)\,\lambda(\dd y)
  \qquad (A \in \XXA,\ x \in \XX),
\]
and $\Xp := \{p > 0\} \in \XXA$. Let $\tau_{n} : \XX\times\XX \to [0,\infty)$ be given
by the Chapman--Kolmogorov recursion
\[
  \tau_{1} := \tau, \qquad
  \tau_{n+1}(y \mid x) := \int_{\XX} \tau_{n}(y \mid z)\,\tau(z \mid x)\,\lambda(\dd z)
  \qquad (n \in \N).
\]
Assume:
\begin{enumerate}[label=(\roman*),nosep,topsep=3pt]
  \item \emph{(invariance)} $\displaystyle\int_{\XX} \tau(y \mid x)\,p(x)\,\lambda(\dd x) = p(y)$ for $\lambda$-almost every $y \in \XX$;
  \item \emph{(eventual positivity)} for every $x \in \XX$ there is a number $n \in \N$ with
        \[
          \lambda\bigl(\{ y \in \XX\setminus\Xp : \tau_{n}(y \mid x) > 0 \}\bigr) = 0
          \qquad\text{and}\qquad
          \lambda\bigl(\{ y \in \Xp : \tau_{n}(y \mid x) = 0 \}\bigr) = 0 .
        \]
\end{enumerate}
Then $\pi$ is the unique invariant probability measure of $T$, and
$\tv{T^{n}\circ\mu - \pi} \to 0$ for every $\mu \in \Prob$.
\end{corollary}

\begin{proof}
Each $\tau_{n}$ is $\XXA\otimes\XXA$-measurable by Tonelli's theorem, which applies
because $\lambda$ is $\sigma$-finite, and an induction using Tonelli again shows that
$T^{n}(A \mid x) = \int_{A}\tau_{n}(y \mid x)\lambda(\dd y)$. Invariance of $\pi$
follows from (i) as in the proof of Corollary \ref{cor:reference}, so
\hyp{A}{A} hold. Put
\[
  \XX_{n} := \Bigl\{ x \in \XX :\ \lambda\bigl(\{y \in \XX\setminus\Xp : \tau_{n}(y\mid x) > 0\}\bigr) = 0
  \ \text{ and }\ \lambda\bigl(\{y \in \Xp : \tau_{n}(y \mid x) = 0\}\bigr) = 0 \Bigr\}
\]
and $t_{n}(y\mid x) := \tau_{n}(y \mid x)/p(y)$ for $y \in \Xp$,
$t_{n}(y \mid x) := 0$ for $y \notin \Xp$. These are exactly the data required in
\hyp{E}{E}:
\begin{itemize}[nosep,topsep=3pt,leftmargin=1.4em]
  \item $\XX_{n} \in \XXA$, because $x \mapsto \lambda(\{y \in B : \tau_{n}(y\mid x) = 0\})$
        and $x \mapsto \lambda(\{y \in B : \tau_{n}(y\mid x) > 0\})$ are measurable for
        every $B \in \XXA$, once more by Tonelli;
  \item $x \in \XX_{n}$ holds precisely when $T^{n}_{x} \sim \pi$, so the
        inclusion $\XX_{n} \subseteq \XX_{n+1}$ is Corollary \ref{cor:selfimprovement},
        and $\bigcup_{n}\XX_{n} = \XX$ is hypothesis (ii);
  \item (b) and (c) of \hyp{E}{E} hold for these $\XX_{n}$ and $t_{n}$ by
        construction, exactly as in the proof of Corollary \ref{cor:reference}.
\end{itemize}
Theorem \ref{thm:general} now applies.
\end{proof}

\begin{remark}[The measurability hypotheses are free here]\label{rem:referencegen}
Corollary \ref{cor:referencegen} deserves emphasis. In the
reference-measure picture the technical parts of \hyp{E}{E} --- measurability
of the sets $\XX_{n}$, monotonicity, and the existence of jointly measurable
densities on them --- are \emph{not} additional assumptions: they were produced in
the proof from the hypotheses alone, and the measurable Lebesgue decomposition
invoked in Remark \ref{rem:measurableXn} was not needed. When moreover $p > 0$
$\lambda$-almost everywhere, hypothesis (ii) of the corollary simplifies to: for
every $x \in \XX$ there is a number $n \in \N$ with $\tau_{n}(y \mid x) > 0$ for
$\lambda$-almost every $y \in \XX$.
\end{remark}

\section{Applications}\label{sec:applications}

Sections \ref{sec:first}--\ref{sec:general} verified \hyp{P}{P} and \hyp{S}{S} under
hypotheses on the kernel that can be read off its definition. The two algorithms of
this section satisfy neither \hyp{D}{D}, nor \hyp{M}{M}, nor \hyp{E}{E} --- their
one step kernels are singular with respect to $\pi$, and so is every iterate --- and
they are covered by Theorem \ref{thm:criterion} all the same. In both cases the two
hypotheses are delivered by one and the same map, through
Lemma \ref{lem:uniformminorant}.

\subsection{The Gibbs sampler}\label{ssec:gibbs}

The Gibbs sampler \cite{GemanGeman84,GelfandSmith90} updates one coordinate of the
state at a time, replacing it by a draw from its conditional distribution under the
target. Its two standard forms differ only in how the coordinate is chosen: the
\emph{systematic scan} runs through the coordinates in a fixed order, and the
\emph{random scan} picks one uniformly at each step. The classical convergence
theory is that of Roberts and Smith \cite{RobertsSmith94} and Tierney
\cite{Tierney94}, who deduce it from Harris recurrence. As will be seen, the
systematic scan falls under Theorem \ref{thm:main}, whereas the random scan falls
under none of the three earlier theorems and under Theorem \ref{thm:criterion}.

\subsubsection*{The setting}

Let $d \in \N$ with $d \ge 2$, let $\mathcal{Y}_{1},\dots,\mathcal{Y}_{d}$ be sets
with $\sigma$-algebras $\mathcal{B}_{\mathcal{Y}_{1}},\dots,\mathcal{B}_{\mathcal{Y}_{d}}$
and $\sigma$-finite measures $\lambda_{1},\dots,\lambda_{d}$ on them, and put
\[
  \XX := \mathcal{Y}_{1}\times\cdots\times\mathcal{Y}_{d},
  \qquad
  \XXA := \mathcal{B}_{\mathcal{Y}_{1}}\otimes\cdots\otimes\mathcal{B}_{\mathcal{Y}_{d}},
  \qquad
  \lambda := \lambda_{1}\otimes\cdots\otimes\lambda_{d},
\]
so that $\lambda$ is $\sigma$-finite. Let $p : \XX \to (0,\infty)$ be measurable with
$\int_{\XX} p \dd\lambda = 1$ and let $\pi(A) := \int_{A} p \dd\lambda$ be the
target; since $p > 0$ everywhere, $\pi \sim \lambda$, so that a measure is absolutely
continuous with respect to $\pi$ if and only if it is absolutely continuous with
respect to $\lambda$.

For $x \in \XX$, a number $i \in \{1,\dots,d\}$ and $w \in \mathcal{Y}_{i}$ let
$x^{i:w} \in \XX$ denote the point obtained from $x$ by replacing its $i$-th
coordinate by $w$, and define the \emph{normalising map}
\[
  Z_{i} : \XX \to (0,\infty], \qquad
  Z_{i}(x) := \int_{\mathcal{Y}_{i}} p(x^{i:w})\,\lambda_{i}(\dd w),
\]
which is measurable by Tonelli's theorem and does not depend on the $i$-th coordinate
of $x$. It is strictly positive, because $p > 0$ and
$\lambda_{i}(\mathcal{Y}_{i}) > 0$, the latter since otherwise
$\int p \dd\lambda = 0$. We assume throughout this subsection that
\begin{equation}
  Z_{i}(x) < \infty \qquad \text{for every } i \in \{1,\dots,d\} \text{ and every } x \in \XX ,
  \label{eq:condexists}
\end{equation}
that is, that all full conditional distributions exist at every point; this is what
makes the algorithm well defined. Under \eqref{eq:condexists} the \emph{coordinate
kernels}
\[
  P_{i}(A \mid x) := \frac{1}{Z_{i}(x)}\int_{\mathcal{Y}_{i}} \ind_{A}(x^{i:w})\,p(x^{i:w})\,\lambda_{i}(\dd w)
  \qquad (A \in \XXA,\ x \in \XX,\ i \in \{1,\dots,d\})
\]
are Markov kernels on $(\XX,\XXA)$: each $P_{i,x}$ is a probability
measure by the definition of $Z_{i}$, and $x \mapsto P_{i}(A \mid x)$ is measurable
by Tonelli's theorem.

For Markov kernels $P$ and $Q$ on $(\XX,\XXA)$ let $QP$ denote the Markov kernel
$(QP)(A \mid x) := \int_{\XX} Q(A \mid z)\,P(\dd z \mid x)$, so that $QP$ is
``first $P$, then $Q$'' and $(QP)\circ\mu = Q\circ(P\circ\mu)$. The two samplers are
\[
  T_{\mathrm{sc}} := P_{d}P_{d-1}\cdots P_{1}
  \qquad\text{and}\qquad
  T_{\mathrm{rs}} := \frac{1}{d}\sum_{i=1}^{d} P_{i} .
\]

\begin{lemma}[The coordinate kernels are reversible]\label{lem:gibbsinvariance}
For every $i \in \{1,\dots,d\}$ the measure $\pi$ is invariant for $P_{i}$; indeed
$P_{i}$ is reversible with respect to $\pi$. Consequently $\pi$ is invariant for
$T_{\mathrm{sc}}$ and for $T_{\mathrm{rs}}$.
\end{lemma}

\begin{proof}
Fix $i$ and write points of $\XX$ as $x = (x_{i},x_{-i})$ with
$x_{-i} \in \mathcal{Y}_{-i} := \prod_{j \ne i}\mathcal{Y}_{j}$, and
$\lambda = \lambda_{i}\otimes\lambda_{-i}$ accordingly. Since $Z_{i}$ does not depend
on the $i$-th coordinate, write $Z_{i}(x_{-i})$ for its common value on the $i$-th
coordinate line through $x_{-i}$, and abbreviate, for $A \in \XXA$,
\[
  I_{A}(x_{-i}) := \int_{\mathcal{Y}_{i}} \ind_{A}(v,x_{-i})\,p(v,x_{-i})\,\lambda_{i}(\dd v),
  \qquad\text{so that}\qquad
  I_{\XX} = Z_{i} .
\]
By Tonelli's theorem, for $A, B \in \XXA$,
\[
  \int_{B} P_{i}(A \mid x)\,\pi(\dd x)
  = \int_{\mathcal{Y}_{-i}} \frac{I_{A}(x_{-i})\,I_{B}(x_{-i})}{Z_{i}(x_{-i})}\,
    \lambda_{-i}(\dd x_{-i}) .
\]
The right-hand side is symmetric in $A$ and $B$, which is reversibility. Taking
$B := \XX$ gives
\[
  (P_{i}\circ\pi)(A)
  = \int_{\mathcal{Y}_{-i}} I_{A}(x_{-i})\,\lambda_{-i}(\dd x_{-i})
  = \pi(A) ,
\]
which is invariance. Invariance for $T_{\mathrm{sc}}$ and $T_{\mathrm{rs}}$ follows,
since both are built from the $P_{i}$ by composition and convex combination, and both
operations preserve invariance of $\pi$.
\end{proof}

\subsubsection*{The systematic scan has a strictly positive density}

For $x,y \in \XX$ and a number $i \in \{0,1,\dots,d\}$ write
\[
  x \oplus_{i} y := (y_{1},\dots,y_{i},x_{i+1},\dots,x_{d}) \in \XX ,
\]
so that $x \oplus_{0} y = x$ and $x \oplus_{d} y = y$: the point $x \oplus_{i}y$ is
the state after the first $i$ coordinates have been updated.

\begin{lemma}[Density of the systematic scan]\label{lem:gibbsscan}
Assume \eqref{eq:condexists} and define
\[
  \tau : \XX\times\XX \to (0,\infty), \qquad
  \tau(y \mid x) := \prod_{i=1}^{d} \frac{p\bigl( x \oplus_{i} y \bigr)}{Z_{i}\bigl( x \oplus_{i-1} y \bigr)} .
\]
Then $\tau$ is $\XXA\otimes\XXA$-measurable, strictly positive at every point of
$\XX\times\XX$, and
\[
  T_{\mathrm{sc}}(A \mid x) = \int_{A} \tau(y \mid x)\,\lambda(\dd y)
  \qquad (A \in \XXA,\ x \in \XX) .
\]
\end{lemma}

\begin{proof}
Measurability is clear, and strict positivity holds because $p > 0$ everywhere and
$0 < Z_{i} < \infty$ everywhere by \eqref{eq:condexists}. For the displayed identity,
unfold the definition of $T_{\mathrm{sc}}$: writing $z^{(0)} := x$ and
$z^{(i)} := (z^{(i-1)})^{i:w_{i}}$, so that $z^{(i)} = x \oplus_{i} w$ with
$w := (w_{1},\dots,w_{d})$, Tonelli's theorem gives
\[
  T_{\mathrm{sc}}(A \mid x)
  = \int_{\mathcal{Y}_{1}}\!\cdots\!\int_{\mathcal{Y}_{d}}
    \ind_{A}\bigl(z^{(d)}\bigr) \prod_{i=1}^{d}
    \frac{p\bigl(z^{(i)}\bigr)}{Z_{i}\bigl(z^{(i-1)}\bigr)}\,
    \lambda_{d}(\dd w_{d})\cdots\lambda_{1}(\dd w_{1}) .
\]
Since $z^{(d)} = w$, the variable of the outermost integration after reordering is
$w$ itself, and the right-hand side equals $\int_{A}\tau(y \mid x)\lambda(\dd y)$.
\end{proof}

\begin{remark}[The systematic scan is covered by Theorem \ref{thm:main}]\label{rem:gibbssystematic}
By Lemmas \ref{lem:gibbsinvariance} and \ref{lem:gibbsscan} the kernel
$T_{\mathrm{sc}}$ satisfies the hypotheses of Corollary \ref{cor:reference} --- and
hence of Theorem \ref{thm:posdens} --- with the reference measure $\lambda$, the target
density $p$ and the transition density $\tau$: invariance is
Lemma \ref{lem:gibbsinvariance}, hypothesis (ii) there is vacuous because
$\Xp = \XX$, and hypothesis (iii) holds because $\tau > 0$ at \emph{every} point. So
$\tv{T_{\mathrm{sc}}^{n}\circ\mu - \pi} \to 0$ for every $\mu \in \Prob$, and no
further work is needed. It is worth noting where the positivity comes from: a single
coordinate kernel $P_{i}$ has no density with respect to $\pi$ at all, but a full
sweep composes $d$ of them and thereby moves every coordinate.
\end{remark}

\subsubsection*{The random scan satisfies none of \hyp{D}{D}, \hyp{M}{M}, \hyp{E}{E}}

\begin{lemma}[The random scan is singular]\label{lem:gibbssingular}
Assume \eqref{eq:condexists} and assume in addition that
$\{w\} \in \mathcal{B}_{\mathcal{Y}_{j}}$ and $\lambda_{j}(\{w\}) = 0$ for every
$j \in \{1,\dots,d\}$ and every $w \in \mathcal{Y}_{j}$. Then, with
$T := T_{\mathrm{rs}}$:
\begin{enumerate}[label=\textup{(\roman*)},nosep,topsep=3pt]
  \item $T_{x}$ is carried by a $\pi$-null set for every $x \in \XX$; in
        particular \hyp{D}{D} fails;
  \item \hyp{M}{M} fails;
  \item $\sing\bigl(T^{n}_{x} \bigm| \pi\bigr) \ge d^{-n} > 0$ for every $n \in \N$
        and every $x \in \XX$; in particular \hyp{E}{E} fails, for every choice
        of the sets $\XX_{n}$ and the maps $t_{n}$.
\end{enumerate}
\end{lemma}

\begin{proof}
For $x \in \XX$ and a number $i$ let
$L_{i}(x) := \{ y \in \XX : y_{j} = x_{j} \text{ for every } j \ne i \} \in \XXA$ be
the $i$-th coordinate line through $x$. It is the measurable rectangle whose $i$-th
factor is $\mathcal{Y}_{i}$ and whose $j$-th factor is the measurable set
$\{x_{j}\}$ for $j \ne i$, and $\lambda(L_{i}(x)) = 0$. To see the latter without
appealing to the convention $0\cdot\infty = 0$, fix $j \ne i$ --- possible because
$d \ge 2$ --- and exhaust $\mathcal{Y}_{i}$ by sets $E_{m} \in \mathcal{B}_{\mathcal{Y}_{i}}$
of finite $\lambda_{i}$-measure, which is possible since $\lambda_{i}$ is
$\sigma$-finite. Each of the measurable rectangles obtained from $L_{i}(x)$ by
replacing its $i$-th factor $\mathcal{Y}_{i}$ by $E_{m}$ has $\lambda$-measure
$\lambda_{i}(E_{m})\prod_{l \ne i}\lambda_{l}(\{x_{l}\}) = 0$, a product of finitely
many finite numbers one of which, the $j$-th, vanishes; and $L_{i}(x)$ is the
increasing union of these rectangles. Hence $\lambda(L_{i}(x)) = 0$ and, since
$\pi \sim \lambda$, also $\pi(L_{i}(x)) = 0$.

(i) $P_{i,x}$ is carried by $L_{i}(x)$, so $T_{x}$ is carried
by $\bigcup_{i} L_{i}(x)$, a $\pi$-null set. Since $T(\XX \mid x) = 1$, the kernel
$T$ has no density with respect to $\pi$.

(ii) Suppose \hyp{M}{M}(a) held, and put $A := \bigcup_{i} L_{i}(x)$. Then
$\pi(A) = 0$, so $1 = T(A \mid x) = \int_{A} k(y \mid x)\pi(\dd y) + r(x)\ind_{A}(x) = r(x)$,
whence $\theta(x) = 0$, contradicting \hyp{M}{M}(b).

(iii) Each $P_{i}$ is idempotent, $P_{i}P_{i} = P_{i}$: the measure
$P_{i,x}$ is carried by $L_{i}(x)$, and $P_{i,z}$ depends on
$z$ only through the coordinates $z_{j}$, $j \ne i$, which agree with those of $x$
for $z \in L_{i}(x)$. Expanding $T^{n} = d^{-n}\sum P_{i_{n}}\cdots P_{i_{1}}$ over
the $d^{n}$ sequences $(i_{1},\dots,i_{n}) \in \{1,\dots,d\}^{n}$ and keeping only
the constant sequence $(1,\dots,1)$ gives $T^{n}_{x} \ge d^{-n}P_{1,x}$
--- it is idempotence that makes the contribution of that one sequence exactly
$d^{-n}P_{1}$, and not $d^{-n}P_{1}^{n}$ with a possibly smaller mass on the line
$L_{1}(x)$; this is why the constant sequence, and not some other, is the one to
retain.
If now $\beta \le T^{n}_{x}$ with $\beta \ll \pi$, then $\beta(L_{1}(x)) = 0$
and therefore
$\beta(\XX) \le T^{n}(\XX \setminus L_{1}(x) \mid x) \le 1 - d^{-n}$, so that
$\sing(T^{n}_{x}\mid\pi) \ge d^{-n}$ by Definition \ref{def:singmass}. In
particular $T^{n}_{x} \not\ll \pi$ for every $n$ and every $x$ by the last sentence of
Proposition \ref{prop:singforms}, so no set $\XX_{n}$ as in \hyp{E}{E}(b) can be
nonempty, and \hyp{E}{E}(a) fails.
\end{proof}

\subsubsection*{The random scan satisfies \hyp{P}{P} and \hyp{S}{S}}

Assumption \hyp{P}{P} will come, as announced, from the single favourable sweep
$P_{d}\cdots P_{1}$. For \hyp{S}{S} one can use the same map through
Lemma \ref{lem:uniformminorant}, and we do; but a sharper and more transparent bound
is available, and since it is what the general weighted scan needs
(Remark \ref{rem:gibbscomment}) we record it first. The point is that a composition
of coordinate kernels is absolutely continuous as soon as \emph{every} coordinate has
been updated at least once, in whatever order and with whatever repetitions.

\begin{lemma}[Refreshing every coordinate produces a density]\label{lem:gibbsrefresh}
Assume \eqref{eq:condexists}, let $n \in \N$ and let
$(i_{1},\dots,i_{n}) \in \{1,\dots,d\}^{n}$ be a sequence with
$\{i_{1},\dots,i_{n}\} = \{1,\dots,d\}$. Then
\[
  \bigl( P_{i_{n}}\cdots P_{i_{1}} \bigr)_{x} \;\ll\; \pi
  \qquad\text{for every } x \in \XX .
\]
\end{lemma}

\begin{proof}
Fix $x \in \XX$, write $z^{(0)} := x$ and $z^{(m)} := \bigl(z^{(m-1)}\bigr)^{i_{m}:w_{m}}$
for $w = (w_{1},\dots,w_{n}) \in \mathcal{Y}_{i_{1}}\times\cdots\times\mathcal{Y}_{i_{n}}$.
Unfolding the definition of the coordinate kernels exactly as in the proof of
Lemma \ref{lem:gibbsscan} gives, for $A \in \XXA$,
\[
  \bigl( P_{i_{n}}\cdots P_{i_{1}} \bigr)(A \mid x)
  = \int \cdots \int \ind_{A}\bigl( z^{(n)} \bigr)
    \prod_{m=1}^{n} \frac{p\bigl( z^{(m)} \bigr)}{Z_{i_{m}}\bigl( z^{(m-1)} \bigr)}\,
    \lambda_{i_{n}}(\dd w_{n}) \cdots \lambda_{i_{1}}(\dd w_{1}) ,
\]
the integrand being nonnegative and jointly measurable, and the quotients being well
defined and finite because $0 < Z_{i} < \infty$ everywhere by \eqref{eq:condexists}.
For each $j \in \{1,\dots,d\}$ let
\[
  m_{j} := \max\{\, m \in \{1,\dots,n\} : i_{m} = j \,\}
\]
be the last time at which the $j$-th coordinate is updated; this is well defined
precisely because every index occurs in the sequence, and $m_{1},\dots,m_{d}$ are $d$
distinct elements of $\{1,\dots,n\}$. By construction the $j$-th coordinate of
$z^{(n)}$ is $w_{m_{j}}$, since it is set at time $m_{j}$ and never touched again.
Thus $z^{(n)}$ depends on $w$ only through $(w_{m_{1}},\dots,w_{m_{d}})$, and the map
$(w_{m_{1}},\dots,w_{m_{d}}) \mapsto z^{(n)}$ is the identification of
$\mathcal{Y}_{1}\times\cdots\times\mathcal{Y}_{d}$ with $\XX$, under which
$\lambda_{i_{m_{1}}}\otimes\cdots\otimes\lambda_{i_{m_{d}}}$ becomes
$\lambda_{1}\otimes\cdots\otimes\lambda_{d} = \lambda$.

The product measure $\bigotimes_{m=1}^{n}\lambda_{i_{m}}$ is $\sigma$-finite, so
Tonelli's theorem permits integrating in any order. Integrating first over the
variables $w_{m}$ with $m \notin \{m_{1},\dots,m_{d}\}$ leaves a nonnegative
measurable map $G(\,\cdot \mid x)$ of the remaining variables, that is of
$y := z^{(n)} \in \XX$, and
\[
  \bigl( P_{i_{n}}\cdots P_{i_{1}} \bigr)(A \mid x) = \int_{A} G(y \mid x)\,\lambda(\dd y)
  \qquad (A \in \XXA) .
\]
Hence the measure is absolutely continuous with respect to $\lambda$, and therefore
with respect to $\pi$, the two being equivalent because $p > 0$ everywhere.
\end{proof}

\noindent
Only the event that every coordinate has been refreshed is used, so the lemma applies
verbatim to a scan with unequal selection probabilities. Its consequence is recorded
in the corollary and in Remark \ref{rem:gibbscomment}.

\begin{corollary}[The Gibbs sampler with random scan converges; Theorem \ref{thm:gibbsintro}]\label{cor:gibbs}
Let $d \in \N$ with $d \ge 2$, let $\mathcal{Y}_{1},\dots,\mathcal{Y}_{d}$,
$\XX,\XXA,\lambda,p,\pi$ and $P_{1},\dots,P_{d}$ be as above, assume
\eqref{eq:condexists}, and put $T := T_{\mathrm{rs}} = \frac{1}{d}\sum_{i} P_{i}$.
Then $\pi$ is the unique invariant probability measure of $T$ and
\[
  \lim_{n\to\infty}\ \sup_{A \in \XXA}\bigl| (T^{n}\circ\mu)(A) - \pi(A) \bigr| = 0
  \qquad\text{for every } \mu \in \Prob .
\]
More precisely, for every $\mu \in \Prob$ and every $n \in \Nz$,
\[
  \sing\bigl( T^{n}\circ\mu \bigm| \pi\bigr) \;\le\; d\,\Bigl( 1-\frac{1}{d} \Bigr)^{\! n} ,
\]
a bound uniform in $\mu$.
\end{corollary}

\begin{proof}
Assumptions \hyp{A1}{A1} and \hyp{A2}{A2} hold by construction and \hyp{A3}{A3} is
Lemma \ref{lem:gibbsinvariance}. Put
\[
  N := d, \qquad \YY := \XX, \qquad
  s : \XX\times\XX \to (0,\infty), \quad s(y \mid x) := d^{-d}\,\frac{\tau(y \mid x)}{p(y)} ,
\]
with $\tau$ the map of Lemma \ref{lem:gibbsscan}. Expanding
$T^{d} = d^{-d}\sum P_{i_{d}}\cdots P_{i_{1}}$ over the $d^{d}$ sequences in
$\{1,\dots,d\}^{d}$ and retaining only the sequence $(1,2,\dots,d)$, whose
contribution is $d^{-d}T_{\mathrm{sc}}$, gives by Lemma \ref{lem:gibbsscan} and
$\pi = p\lambda$
\begin{equation}
  T^{d}(A \mid x) \;\ge\; d^{-d}\,T_{\mathrm{sc}}(A \mid x)
  = d^{-d}\int_{A}\tau(y \mid x)\,\lambda(\dd y)
  = \int_{A} s(y \mid x)\,\pi(\dd y)
  \qquad (A \in \XXA,\ x \in \XX) .
  \label{eq:gibbsminorant}
\end{equation}
Moreover, $T_{\mathrm{sc},x}$ being a probability measure,
\begin{equation}
  \int_{\XX} s(y \mid x)\,\pi(\dd y) = d^{-d}\int_{\XX}\tau(y\mid x)\,\lambda(\dd y) = d^{-d}
  \qquad (x \in \XX) .
  \label{eq:gibbsmass}
\end{equation}

\emph{Assumption \hyp{P}{P}.} Let $\varepsilon \in (0,1)$ be a number and take the
above $N$, $\YY$ and $s$ together with $\YY' := \XX$, none of which depends on
$\varepsilon$. Hypothesis (a) is trivial since
$\pi(\XX\setminus\YY) = \pi(\XX\setminus\YY') = 0$, hypothesis (b) is
\eqref{eq:gibbsminorant}, and hypothesis (c) holds because $s > 0$ at every point, so
that the set occurring in it is empty and its measure is $0 \le \varepsilon$.

\emph{Property \hyp{S}{S}.} Applying Lemma \ref{lem:uniformminorant} with the same
$N = d$, with $u := s$ and with $c := d^{-d}$ --- its hypothesis (a) being
\eqref{eq:gibbsminorant} and its hypothesis (b) being \eqref{eq:gibbsmass} --- already
gives $\sing(T^{n}\circ\mu\mid\pi) \le (1-d^{-d})^{\lfloor n/d\rfloor}$, hence \hyp{S}{S}.

For the sharper bound stated, expand $T^{n} = d^{-n}\sum P_{i_{n}}\cdots P_{i_{1}}$
over the $d^{n}$ sequences in $\{1,\dots,d\}^{n}$ and let $W$ be the set of those
sequences in which every index occurs. Put
\[
  \beta_{n} := d^{-n} \sum_{(i_{1},\dots,i_{n}) \in W}
     \bigl( P_{i_{n}}\cdots P_{i_{1}} \bigr)\circ\mu .
\]
Then $\beta_{n} \le T^{n}\circ\mu$, because the discarded sequences contribute a
nonnegative measure; and $\beta_{n} \ll \pi$, because each summand is: if
$\pi(A) = 0$ then $(P_{i_{n}}\cdots P_{i_{1}})(A \mid x) = 0$ for every $x$ by
Lemma \ref{lem:gibbsrefresh}, so the integral of that map against $\mu$ vanishes.
Now $d^{-n}|W|$ is the probability that $n$ independent uniform draws from
$\{1,\dots,d\}$ exhaust $\{1,\dots,d\}$, so by the union bound over the $d$ events
``the index $i$ is never drawn'',
\[
  \beta_{n}(\XX) = d^{-n}|W| \;\ge\; 1 - d\Bigl(1-\frac1d\Bigr)^{\! n} .
\]
Definition \ref{def:singmass} gives $\sing(T^{n}\circ\mu\mid\pi) \le 1 - \beta_{n}(\XX)
\le d\,(1-1/d)^{n}$, as claimed. This beats the bound of the previous paragraph
substantially --- for $d = 2$ it is $2\cdot 2^{-n}$ against roughly $0.866^{n}$ ---
and it is uniform in $\mu$.

Theorem \ref{thm:criterion} now applies.
\end{proof}

\begin{remark}[What the two hypotheses cost here]\label{rem:gibbscomment}
Three features of the verification are worth recording. First, both hypotheses come
from the single map $s$, that is, from the single event that the next $d$ steps happen
to update the coordinates $1,2,\dots,d$ in that order; the two are not established by
separate arguments. Second, the bound on the singular mass is geometric and
\emph{uniform in $\mu$}, so a single time $n_{1}$ serves every initial distribution
--- more than \hyp{S}{S} demands. Third, no
positivity beyond $p > 0$ is used, and the positivity that \hyp{P}{P} needs is
manufactured by the sweep rather than assumed: it is the composition of $d$
coordinate kernels, each of them singular with respect to $\pi$, that produces the
everywhere strictly positive density $\tau$. By Corollary \ref{cor:lln} the ergodic
averages along a single trajectory of the random scan Gibbs sampler converge as
well.

If the coordinate is chosen from a fixed distribution
$(\omega_{1},\dots,\omega_{d})$ with $\omega_{i} > 0$ instead of uniformly, so that
$T_{\omega} := \sum_{i=1}^{d}\omega_{i}P_{i}$, both halves survive with the constants
changed. For \hyp{P}{P}, the sweep $(1,2,\dots,d)$ now carries probability
$\prod_{i}\omega_{i}$ rather than $d^{-d}$, so \eqref{eq:gibbsminorant} and
\eqref{eq:gibbsmass} hold with $s(y \mid x) := \bigl(\prod_{i}\omega_{i}\bigr)\tau(y\mid x)/p(y)$
and $c := \prod_{i}\omega_{i}$. For \hyp{S}{S}, the argument of the previous proof
applies verbatim, the only change being that a sequence $(i_{1},\dots,i_{n})$ now
carries weight $\omega_{i_{1}}\cdots\omega_{i_{n}}$: writing $W$ for the set of
sequences in which every index occurs, the measure
$\beta_{n} := \sum_{W}\omega_{i_{1}}\cdots\omega_{i_{n}}\,(P_{i_{n}}\cdots P_{i_{1}})\circ\mu$
satisfies $\beta_{n} \le T_{\omega}^{n}\circ\mu$ and $\beta_{n}\ll\pi$ by
Lemma \ref{lem:gibbsrefresh}, while $\beta_{n}(\XX)$ is the probability that $n$
independent draws from $(\omega_{1},\dots,\omega_{d})$ exhaust $\{1,\dots,d\}$. The
union bound over the events ``the index $i$ is never drawn'' gives the
coupon-collector estimate
\[
  \sing\bigl( T_{\omega}^{n}\circ\mu \bigm| \pi\bigr) \;\le\; \sum_{i=1}^{d} (1-\omega_{i})^{n}
  \qquad (\mu \in \Prob,\ n \in \Nz) ,
\]
of which the bound of Corollary \ref{cor:gibbs} is the case
$\omega_{i} \equiv 1/d$. If some $\omega_{i}$ vanishes the $i$-th coordinate is never
updated and the conclusion is false; the proof detects this at once, both bounds
degenerating.
\end{remark}

\begin{remark}[The full conditionals must exist at \emph{every} point]\label{rem:gibbssharp}
Hypothesis \eqref{eq:condexists} cannot be weakened to the $\lambda$-almost everywhere
existence of the full conditionals, and the reason is that Theorem \ref{thm:gibbsintro} and
Corollary \ref{cor:gibbs} assert convergence from \emph{every} initial distribution.
The following object shows it. Take $d := 2$,
$\mathcal{Y}_{1} = \mathcal{Y}_{2} := \R$ with Lebesgue measure, so that
$\XX = \R^{2}$ and $\lambda$ is planar Lebesgue measure, and put
\[
  h(x,y) := |x|^{-1}\ind_{\{0<|x|<1\}}\ind_{\{|y|<|x|\}},
  \qquad
  h'(x,y) := |y-3|^{-1}\ind_{\{0<|y-3|<1\}}\ind_{\{|x-3|<|y-3|\}} ,
\]
\[
  p \;:=\; \frac{1}{9}\Bigl( \varphi(x)\varphi(y) + h + h' \Bigr) ,
\]
with $\varphi$ the standard normal density. Both $h$ and $h'$ are finite at every
point --- their indicators vanish where the factors $|x|^{-1}$, $|y-3|^{-1}$ blow up
--- so $p : \R^{2} \to (0,\infty)$ is measurable and strictly positive everywhere;
and $\int h \dd\lambda = \int h'\dd\lambda = 4$ by Tonelli, so
$\int p \dd\lambda = 1$. All the standing hypotheses of
Subsection \ref{ssec:gibbs} therefore hold \emph{except} \eqref{eq:condexists}.

Compute the two normalising maps. Integrating $h$ out over the first coordinate gives
$\int_{\R} h(w,y)\,\dd w = 2\ln(1/|y|)$ when $0 < |y| < 1$, and $0$ when
$|y| \ge 1$. At $y = 0$, by contrast, the indicator $\ind_{\{|y|<|w|\}}$ is satisfied
by every $w \ne 0$, and $\int_{0<|w|<1}|w|^{-1}\dd w = \infty$. The other two
summands of $p$ contribute finitely in either case. Hence
\[
  Z_{1} = \infty \ \text{ exactly on } \{y = 0\},
  \qquad\text{and symmetrically}\qquad
  Z_{2} = \infty \ \text{ exactly on } \{x = 3\} .
\]
Both are $\lambda$-null sets, so all full conditionals exist $\lambda$-almost
everywhere. But at the single point $(3,0)$ \emph{both} of them fail to exist.

Complete $P_{1}$ and $P_{2}$ on their bad sets in the only way available, by
``staying put'': $P_{i,x} := \delta_{x}$ when $Z_{i}(x) = \infty$. Since
the bad sets are $\pi$-null, Lemma \ref{lem:gibbsinvariance} is unaffected and $\pi$
remains invariant for $P_{1}$, $P_{2}$ and $T_{\mathrm{rs}}$. Yet
$T_{\mathrm{rs},(3,0)} = \tfrac12\delta_{(3,0)} + \tfrac12\delta_{(3,0)}
= \delta_{(3,0)}$: the point $(3,0)$ is absorbing. So $\delta_{(3,0)}$ is a second
invariant probability measure, uniqueness fails, and
$T_{\mathrm{rs}}^{n}\circ\delta_{(3,0)} = \delta_{(3,0)} \not\to \pi$.

The words ``for every $i$ and every $x$'' in \eqref{eq:condexists} are therefore
load-bearing, and not a convenience of the write-up. What goes wrong is exactly what
\hyp{P}{P} is designed to see: the minorant $s$ of Corollary \ref{cor:gibbs} is built
from $\tau$, which requires $Z_{i}$ to be finite along the whole sweep, and no set
$\YY$ of full $\pi$-measure repairs a defect at a single starting point when the
conclusion is quantified over all of them. Compare Remark \ref{rem:mhsharp}, and the
same phenomenon in the proof of Corollary \ref{cor:reference}, where hypothesis (ii)
is likewise needed for \emph{every} $x$.
\end{remark}

\subsection{Parallel tempering}\label{ssec:tempering}

Parallel tempering, also called replica exchange
\cite{SwendsenWang86,Geyer91}, runs $K$ chains side by side, the $k$-th one targeting
a flattened version $\pi_{k}$ of the distribution of interest, and occasionally
proposes to exchange the states of two of them. The flattened chains move easily
between the modes of the target and, through the exchanges, communicate that mobility
to the chain that targets $\pi_{1}$. As with the Gibbs sampler, the classical
convergence theory proceeds through Harris recurrence
\cite{Tierney94,RobertsRosenthal04}.

What puts the algorithm outside Sections \ref{sec:first}--\ref{sec:general} is the
exchange move: it is deterministic once accepted, so it contributes an atom at a
\emph{permuted} point rather than at the starting point, and \hyp{M}{M} demands the
latter. The verification of \hyp{P}{P} and \hyp{S}{S} is nevertheless one
application of Lemma \ref{lem:uniformminorant}, exactly as for the Gibbs sampler.

\subsubsection*{The setting}

Let $\mathcal{Y}$ be a set with a $\sigma$-algebra $\mathcal{B}_{\mathcal{Y}}$ and a
$\sigma$-finite measure $\lambda_{0}$ on it, let $K \in \N$ with $K \ge 2$, and let
\[
  p_{k} : \mathcal{Y} \to (0,\infty), \qquad \int_{\mathcal{Y}} p_{k}\dd\lambda_{0} = 1,
  \qquad \pi_{k}(A) := \int_{A} p_{k}\dd\lambda_{0}
  \qquad (k = 1,\dots,K)
\]
be measurable maps and the associated probability measures on
$(\mathcal{Y},\mathcal{B}_{\mathcal{Y}})$; in practice $p_{k}$ is proportional to
$p_{1}^{\beta_{k}}$ for numbers $1 = \beta_{1} > \dots > \beta_{K} > 0$, but no
relation between the $p_{k}$ is needed below. Put
\[
  \XX := \mathcal{Y}^{K}, \quad
  \XXA := \mathcal{B}_{\mathcal{Y}}^{\otimes K}, \quad
  \lambda := \lambda_{0}^{\otimes K}, \quad
  p(x) := \prod_{k=1}^{K} p_{k}(x_{k}), \quad
  \pi := \pi_{1}\otimes\cdots\otimes\pi_{K} = p\lambda ,
\]
so that $\lambda$ is $\sigma$-finite and $\pi \sim \lambda$, as in
Subsection \ref{ssec:gibbs}.

\emph{The within-chain moves.} For each $k$ let $R_{k}$ be a Markov kernel on
$(\mathcal{Y},\mathcal{B}_{\mathcal{Y}})$ which leaves $\pi_{k}$ invariant and is of
the shape \hyp{M}{M} with respect to $\pi_{k}$: there is a
$\mathcal{B}_{\mathcal{Y}}\otimes\mathcal{B}_{\mathcal{Y}}$-measurable map
$\kappa_{k} : \mathcal{Y}\times\mathcal{Y} \to [0,\infty)$ with
\[
  R_{k}(A \mid y) = \int_{A} \kappa_{k}(z \mid y)\,\pi_{k}(\dd z)
                    + \bigl( 1 - \theta_{k}(y) \bigr)\,\ind_{A}(y),
  \qquad
  \theta_{k}(y) := \int_{\mathcal{Y}} \kappa_{k}(z \mid y)\,\pi_{k}(\dd z) \le 1 ,
\]
and $\kappa_{k} > 0$ $(\pi_{k}\otimes\pi_{k})$-almost everywhere. A
Metropolis--Hastings kernel for $\pi_{k}$ is of this shape, by
Corollary \ref{cor:mh}. We assume in addition that the moving probabilities are
bounded away from $0$:
\begin{equation}
  \theta^{-} := \min_{k = 1,\dots,K}\ \inf_{y \in \mathcal{Y}} \theta_{k}(y) \;>\; 0 .
  \label{eq:temperingtheta}
\end{equation}
The \emph{parallel update kernel} is the product kernel
\[
  R(A \mid x) := \int_{\XX} \ind_{A}(y)\, R_{1}(\dd y_{1} \mid x_{1}) \cdots R_{K}(\dd y_{K} \mid x_{K})
  \qquad (A \in \XXA,\ x \in \XX) ,
\]
which updates all $K$ components independently.

\emph{The exchange moves.} Let $\mathcal{S}$ be a nonempty finite set of pairs
$\{k,l\}$ with $k \ne l$ --- all pairs, or only the neighbouring ones, as one
prefers. For $\{k,l\} \in \mathcal{S}$ let $\varsigma_{kl} : \XX \to \XX$ be the map
exchanging the $k$-th and the $l$-th coordinate, which is measurable and satisfies
$\varsigma_{kl}\circ\varsigma_{kl} = \mathrm{id}$, and put
\[
  a_{kl} : \XX \to (0,1], \qquad
  a_{kl}(x) := \min\Bigl\{ 1,\ \frac{p_{k}(x_{l})\,p_{l}(x_{k})}{p_{k}(x_{k})\,p_{l}(x_{l})} \Bigr\} ,
\]
\[
  S_{kl}(A \mid x) := a_{kl}(x)\,\ind_{A}(\varsigma_{kl}x)
                      + \bigl(1-a_{kl}(x)\bigr)\,\ind_{A}(x)
  \qquad (A \in \XXA,\ x \in \XX) .
\]
Finally fix a number $\omega \in [0,1)$, the probability of attempting an exchange,
and let
\[
  T := (1-\omega)\,R \;+\; \frac{\omega}{|\mathcal{S}|}\sum_{\{k,l\} \in \mathcal{S}} S_{kl}
\]
be the parallel tempering kernel. The value $\omega = 0$ is admitted, and gives the
pure parallel update $T = R$, in which no exchange is ever attempted; the algorithm of
interest of course has $\omega > 0$, but nothing below uses it, and the hypothesis is
stated as the proof needs it (Remark \ref{rem:temperingtheta}).

\begin{lemma}[Invariance]\label{lem:temperinginvariance}
The measure $\pi$ is invariant for $R$ and for each $S_{kl}$, hence for $T$.
\end{lemma}

\begin{proof}
For $R$: the sets $A = A_{1}\times\cdots\times A_{K}$ with
$A_{k} \in \mathcal{B}_{\mathcal{Y}}$ form a $\pi$-system generating $\XXA$, and for
such $A$, by Tonelli's theorem and the invariance of each $\pi_{k}$ for $R_{k}$,
\[
  (R\circ\pi)(A) = \prod_{k=1}^{K}\int_{\mathcal{Y}} R_{k}(A_{k}\mid x_{k})\,\pi_{k}(\dd x_{k})
                 = \prod_{k=1}^{K}\pi_{k}(A_{k}) = \pi(A) ,
\]
so $R\circ\pi = \pi$ by uniqueness of measures agreeing on a generating $\pi$-system.

For $S_{kl}$ we show reversibility. The measure
$\pi(\dd x)S_{kl}(\dd y \mid x)$ on $(\XX\times\XX,\XXA\otimes\XXA)$ is the sum of
the diagonal part $C \mapsto \int_{\XX}(1-a_{kl}(x))\ind_{C}(x,x)\pi(\dd x)$, which is
invariant under the swap $(x,y)\mapsto(y,x)$, and of the part
$C \mapsto \int_{\XX} a_{kl}(x)\,\ind_{C}(x,\varsigma_{kl}x)\,p(x)\,\lambda(\dd x)$.
Since the factors of $\lambda$ are all equal to $\lambda_{0}$, the measure $\lambda$
is invariant under $\varsigma_{kl}$; and
\[
  a_{kl}(x)\,p(x) = \min\bigl\{ p(x),\, p(\varsigma_{kl}x) \bigr\}
  \qquad (x \in \XX),
\]
because $p(\varsigma_{kl}x)/p(x) = p_{k}(x_{l})p_{l}(x_{k})/(p_{k}(x_{k})p_{l}(x_{l}))$,
the other factors of $p$ being unaffected. The right-hand side is unchanged by
$x \mapsto \varsigma_{kl}x$, so substituting $x = \varsigma_{kl}u$ turns the second
part into $C \mapsto \int_{\XX} a_{kl}(u)\ind_{C}(\varsigma_{kl}u,u)p(u)\lambda(\dd u)$,
which is its image under the swap. Hence $\pi(\dd x)S_{kl}(\dd y\mid x)$ is symmetric,
and invariance follows by evaluating it on $\XX \times A$ and on $A \times \XX$.

Invariance for $T$ follows, a convex combination of kernels leaving $\pi$ invariant
leaving $\pi$ invariant.
\end{proof}

\begin{lemma}[Parallel tempering is singular]\label{lem:temperingsingular}
Assume in addition that $\omega > 0$, that $\{y\} \in \mathcal{B}_{\mathcal{Y}}$ and
that $\lambda_{0}(\{y\}) = 0$ for every $y \in \mathcal{Y}$. Then
$\sing(T^{n}_{x}\mid\pi) > 0$ for every $n \in \N$ and every $x \in \XX$, and
none of \hyp{D}{D}, \hyp{M}{M} and \hyp{E}{E} holds.
\end{lemma}

\begin{proof}
Every singleton of $\XX$ is a measurable rectangle with
$\pi(\{z\}) = \prod_{k}\pi_{k}(\{z_{k}\}) = 0$. Fix $x \in \XX$ and
$\{k,l\} \in \mathcal{S}$, and define $z_{0} := x$ and
$z_{m+1} := \varsigma_{kl}z_{m}$, so that $z_{m} \in \{x, \varsigma_{kl}x\}$ for every
$m$. Retaining, in the expansion of $T^{n}$, only the paths that attempt and accept
the exchange $\{k,l\}$ at each of the $n$ steps gives
\[
  T^{n}(\{z_{n}\} \mid x) \;\ge\; \Bigl( \frac{\omega}{|\mathcal{S}|} \Bigr)^{\! n}
  \prod_{m=0}^{n-1} a_{kl}(z_{m}) \;>\; 0 ,
\]
the product being positive because $a_{kl}$ takes values in $(0,1]$. Since
$\pi(\{z_{n}\}) = 0$, every measure $\beta \le T^{n}_{x}$ with
$\beta \ll \pi$ satisfies $\beta(\XX) \le 1 - T^{n}(\{z_{n}\}\mid x) < 1$, so
$\sing(T^{n}_{x}\mid\pi) > 0$ by Definition \ref{def:singmass}. By the last sentence of
Proposition \ref{prop:singforms} no iterate $T^{n}_{x}$ is absolutely continuous
with respect to $\pi$; hence \hyp{D}{D} fails, and so does \hyp{E}{E}, since no set
$\XX_{n}$ as in \hyp{E}{E}(b) can be nonempty and \hyp{E}{E}(a) therefore fails.

For \hyp{M}{M}, choose $x$ with $x_{k} \ne x_{l}$ for some $\{k,l\} \in \mathcal{S}$;
such a point exists because $\mathcal{Y}$ has more than one element, $\lambda_{0}$
being atomless and not the zero measure. Then $\varsigma_{kl}x \ne x$, and
$A := \{\varsigma_{kl}x\}$ satisfies $\pi(A) = 0$ and $\ind_{A}(x) = 0$, so that
\hyp{M}{M}(a) would force
$T(A \mid x) = \int_{A} k(y\mid x)\pi(\dd y) + r(x)\ind_{A}(x) = 0$, whereas
$T(A \mid x) \ge \frac{\omega}{|\mathcal{S}|}a_{kl}(x) > 0$.
\end{proof}

\begin{corollary}[Parallel tempering converges]\label{cor:tempering}
Let $\mathcal{Y},\lambda_{0},K,p_{k},\pi_{k}$ and $\XX,\XXA,\lambda,\pi$ be as above,
let $R_{1},\dots,R_{K}$ satisfy the hypotheses stated there including
\eqref{eq:temperingtheta}, let $\mathcal{S} \ne \emptyset$ and
$\omega \in [0,1)$, and let $T$ be the parallel tempering kernel. Then $\pi$ is the
unique invariant probability measure of $T$ and
\[
  \lim_{n\to\infty}\ \sup_{A \in \XXA}\bigl| (T^{n}\circ\mu)(A) - \pi(A) \bigr| = 0
  \qquad\text{for every } \mu \in \Prob .
\]
More precisely, with $c := (1-\omega)(\theta^{-})^{K} \in (0,1]$,
\[
  \sing\bigl( T^{n}\circ\mu \bigm| \pi\bigr) \;\le\; (1-c)^{n}
  \qquad (\mu \in \Prob,\ n \in \Nz) .
\]
\end{corollary}

\begin{proof}
Assumptions \hyp{A1}{A1} and \hyp{A2}{A2} hold by construction and \hyp{A3}{A3} is
Lemma \ref{lem:temperinginvariance}. Define the
$\XXA\otimes\XXA$-measurable map
\[
  u : \XX\times\XX \to [0,\infty), \qquad
  u(y \mid x) := (1-\omega)\prod_{k=1}^{K} \kappa_{k}(y_{k} \mid x_{k}) .
\]
For each $k$ one has $R_{k,x_{k}} \ge \kappa_{k}(\,\cdot\mid x_{k})\pi_{k}$
as measures on $\mathcal{Y}$, and products of finite measures respect this order:
\begin{equation}
  \mu_{k} \ge \nu_{k} \ \ (k=1,\dots,K)
  \quad\Longrightarrow\quad
  \mu_{1}\otimes\cdots\otimes\mu_{K} \;\ge\; \nu_{1}\otimes\cdots\otimes\nu_{K} .
  \label{eq:productorder}
\end{equation}
This is proved by telescoping, and \emph{not} by a monotone class argument: the family
of sets on which one measure dominates another is closed neither under complements nor
under proper differences, so the monotone class theorem does not apply to an
inequality. Instead, put
\[
  \Pi_{j} := \nu_{1}\otimes\cdots\otimes\nu_{j-1}\otimes\mu_{j}\otimes\cdots\otimes\mu_{K}
  \qquad (j = 1,\dots,K+1),
\]
so that $\Pi_{1} = \bigotimes_{k}\mu_{k}$ and $\Pi_{K+1} = \bigotimes_{k}\nu_{k}$. For
each $j$ the finite measure
$\Sigma_{j} := \nu_{1}\otimes\cdots\otimes\nu_{j-1}\otimes(\mu_{j}-\nu_{j})\otimes\mu_{j+1}\otimes\cdots\otimes\mu_{K}$
is well defined and nonnegative, $\mu_{j}-\nu_{j}$ being a finite nonnegative measure
by hypothesis; and on a measurable rectangle $A_{1}\times\cdots\times A_{K}$,
\[
  \Pi_{j} - \Pi_{j+1}
  = \prod_{i<j}\nu_{i}(A_{i}) \cdot \bigl( \mu_{j}(A_{j}) - \nu_{j}(A_{j}) \bigr) \cdot \prod_{i>j}\mu_{i}(A_{i})
  = \Sigma_{j}(A_{1}\times\cdots\times A_{K}) .
\]
The rectangles form a $\pi$-system generating $\XXA$, and $\Pi_{j}-\Pi_{j+1}$ and
$\Sigma_{j}$ are finite measures of the same total mass, so they agree everywhere by
the uniqueness theorem for measures --- here the monotone class argument \emph{is}
legitimate, because what is being extended is an equality. Summing,
$\bigotimes_{k}\mu_{k} - \bigotimes_{k}\nu_{k} = \sum_{j=1}^{K}\Sigma_{j} \ge 0$,
which is \eqref{eq:productorder}. Hence, by Tonelli's theorem,
\begin{equation}
  T(A \mid x) \;\ge\; (1-\omega)\,R(A \mid x) \;\ge\; \int_{A} u(y \mid x)\,\pi(\dd y)
  \qquad (A \in \XXA,\ x \in \XX) ,
  \label{eq:temperingminorant}
\end{equation}
and, again by Tonelli and by \eqref{eq:temperingtheta},
\begin{equation}
  \int_{\XX} u(y \mid x)\,\pi(\dd y) = (1-\omega)\prod_{k=1}^{K}\theta_{k}(x_{k})
  \;\ge\; (1-\omega)\,(\theta^{-})^{K} = c \;>\; 0
  \qquad (x \in \XX) .
  \label{eq:temperingmass}
\end{equation}

\emph{Assumption \hyp{P}{P}.} Let $\varepsilon \in (0,1)$ be a number and take
$N := 1$, $\YY := \XX$ and $s := u$, none of which depends on $\varepsilon$.
Hypothesis (a) is trivial and hypothesis (b) is \eqref{eq:temperingminorant}. For
hypothesis (c), the set $\{(x,y) \in \XX\times\XX : u(y \mid x) = 0\}$ is the union
over $k$ of the sets $\{(x,y) : \kappa_{k}(y_{k}\mid x_{k}) = 0\}$, each of which is
$(\pi\otimes\pi)$-null because $\kappa_{k} > 0$
$(\pi_{k}\otimes\pi_{k})$-almost everywhere and $\pi\otimes\pi$ is a product measure
whose $(x_{k},y_{k})$-marginal is $\pi_{k}\otimes\pi_{k}$. Fubini's theorem then gives
$\pi(\{x \in \XX : u(y\mid x) = 0\}) = 0$ for $\aeevery$ $y$; taking for $\YY'$ the
full-measure set of such $y$ gives (a) and (c), the bound in (c) holding with $0$ in
place of $\varepsilon$.

\emph{Property \hyp{S}{S}.} Apply Lemma \ref{lem:uniformminorant} with $N := 1$,
with this $u$ and with $c$ as above: its hypothesis (a) is
\eqref{eq:temperingminorant} and its hypothesis (b) is \eqref{eq:temperingmass}.

Theorem \ref{thm:criterion} now applies.
\end{proof}

\begin{remark}[What the proof really uses]\label{rem:temperingtheta}
Three comments. First, the exchange moves play no role whatever in the proof: the
minorant $u$ comes from the parallel update alone, and the exchanges enter only
through the requirement that $\pi$ be invariant for them, which
Lemma \ref{lem:temperinginvariance} supplies. This is as it should be. The exchanges
are what make the algorithm efficient, and efficiency is a statement about rates,
which the present method does not reach (Remark \ref{rem:notclaimed}); they are not
what makes it converge.

It is better to say this openly than to leave it to be discovered, because it means
that Corollary \ref{cor:tempering} is in truth a theorem about something more general,
and one may as well state that theorem. \emph{Let $R$ be a product of $K$ kernels of
the shape \hyp{M}{M} with moving probabilities bounded below as in
\eqref{eq:temperingtheta}, let $(V_{\iota})_{\iota \in I}$ be an arbitrary family of
Markov kernels on $(\XX,\XXA)$ leaving $\pi$ invariant, let
$(\varpi_{\iota})_{\iota\in I}$ be nonnegative weights and let $\omega \in [0,1)$ with
$\sum_{\iota}\varpi_{\iota} = \omega$. Then the kernel
$T := (1-\omega)R + \sum_{\iota}\varpi_{\iota}V_{\iota}$ satisfies the conclusion of
Corollary \ref{cor:tempering}, with the same constant
$c = (1-\omega)(\theta^{-})^{K}$.} The proof is the one given, word for word: the
family $(V_{\iota})$ is used only to know that $\pi$ is invariant for $T$, and
$T \ge (1-\omega)R$ is all that \eqref{eq:temperingminorant} needs. The exchange
kernels $S_{kl}$ are one such family; so are Metropolis moves on the joint state, and
so is doing nothing. This pre-empts the reasonable objection that the theorem does not
see the algorithm --- it does not, and cannot, since what makes parallel tempering
worth running is invisible to any rate-free criterion.

Second, the hypothesis $\omega < 1$ is used, and $\omega > 0$ is not. The proof needs
$1-\omega > 0$ so that $c > 0$ in \eqref{eq:temperingmass}, and nothing else; the
nonemptiness of $\mathcal{S}$ is likewise never used, except to make the displayed
kernel well defined. Remark \ref{rem:temperingsingularscope} records where $\omega > 0$
\emph{is} needed, namely in Lemma \ref{lem:temperingsingular}, which is a statement
about the algorithm and not about its convergence.

Third, the requirement \eqref{eq:temperingtheta} that the moving probabilities be
bounded away from $0$ is a genuine restriction and not an artefact of the
bookkeeping. It holds, for instance, when each $R_{k}$ is an independence sampler
whose proposal density $q_{k}$ satisfies $\sup_{y} p_{k}(y)/q_{k}(y) < \infty$, and
more generally whenever the component samplers accept with probability bounded below.
Without it, Lemma \ref{lem:uniformminorant} is unavailable: the minorant $u$ still
exists but its mass $(1-\omega)\prod_{k}\theta_{k}(x_{k})$ may approach $0$, and
\hyp{S}{S} then has to be established by following the trajectory rather than by
a bound valid at every point. That is exactly the gap described in
Subsection \ref{ssec:directions}, Direction \ref{dir:trajectory}.
\end{remark}

\begin{remark}[Where $\omega > 0$ is used]\label{rem:temperingsingularscope}
Lemma \ref{lem:temperingsingular} --- that no iterate of $T$ is absolutely continuous
with respect to $\pi$, so that none of \hyp{D}{D}, \hyp{M}{M}, \hyp{E}{E} applies ---
does require $\omega > 0$, and visibly so: with $\omega = 0$ the kernel is the product
$R$ of $K$ kernels of the shape \hyp{M}{M}, which is itself of the shape \hyp{M}{M}
with an atom at the starting point, and Theorem \ref{thm:mh} then applies directly.
It is the exchange move, contributing an atom at a \emph{permuted} point rather than
at the starting point, that puts the algorithm outside
Sections \ref{sec:first}--\ref{sec:general}; and it is the parallel update that brings
it back inside Theorem \ref{thm:criterion}. The two halves of the subsection concern
different features of the kernel, which is why their hypotheses differ.
\end{remark}

\subsection{Further algorithms}

\begin{remark}[Other samplers of the same shape]\label{rem:otheralgorithms}
Many samplers used in practice have, like the random scan Gibbs sampler, a singular
part that is neither absent nor an atom at the starting point, and are therefore
outside Theorems \ref{thm:main}, \ref{thm:mh} and \ref{thm:general} while remaining
within reach of Theorem \ref{thm:criterion}. We indicate the shape of the two
verifications without carrying them out; in each case \hyp{P}{P} comes from one
favourable sweep and \hyp{S}{S} from the event that every component has been
refreshed.
\begin{itemize}[nosep,topsep=3pt,leftmargin=1.4em,itemsep=3pt]
  \item \emph{Metropolis within Gibbs} \cite{Tierney94,RobertsRosenthal04}, in which
        the conditional draw is replaced by a Metropolis--Hastings step: \hyp{P}{P} from
        the sweep in which every coordinate is proposed and accepted,
        \hyp{S}{S} from the bound of the singular mass by the probability that
        some coordinate has never been accepted.
  \item \emph{Ensemble samplers with affine invariance} \cite{GoodmanWeare10}, in
        which one walker at a time is moved along the line joining it to another:
        a single step is carried by a union of lines, exactly as in
        Lemma \ref{lem:gibbssingular}, and both hypotheses follow as for the random
        scan.
  \item \emph{Piecewise deterministic samplers}, such as the bouncy particle sampler
        \cite{BouchardCote18} and the zig-zag process \cite{Bierkens19}, observed at
        the times of a fixed grid: over one time step the singular part is the
        deterministic transport along which no event has occurred, whose mass decays
        with the number of steps, and \hyp{P}{P} follows once enough events have
        taken place.
  \item \emph{Particle Gibbs and conditional sequential Monte Carlo}
        \cite{AndrieuDoucetHolenstein10}: the reference trajectory is retained with
        positive probability, which contributes an atom, and partial updates
        contribute components carried by lower dimensional sets.
\end{itemize}
Two comments of a different kind. Hamiltonian Monte Carlo
\cite{Duane87,Livingstone19} does \emph{not} require this section: momentum
refreshment followed by the leapfrog map produces an absolutely continuous part, and
the rejection produces an atom at the starting point, so it is of the shape
\hyp{M}{M} and is covered by Theorem \ref{thm:mh} as soon as that absolutely
continuous part has an almost everywhere positive density, which is what a
randomised integration time supplies. Reversible jump Markov chain Monte Carlo
\cite{Green95}, and the Bayesian variable selection samplers built on it, are of the
shape \hyp{M}{M} as well, so that \hyp{S}{S} is again bounded by
$\int r^{n}\dd\mu$; but \hyp{P}{P} fails as stated, because birth and death moves
change the model index only by one and therefore cannot reach $\aeevery$ endpoint in
a bounded number of steps. That family needs the localised form of
Proposition \ref{prop:minorisation} described in Subsection \ref{ssec:directions},
Direction \ref{dir:local}.
\end{remark}

\begin{remark}[A boundary case: the preconditioned Crank--Nicolson algorithm]\label{rem:pcn}
Property \hyp{S}{S} is a genuine restriction and not a formality. Consider the
preconditioned Crank--Nicolson sampler \cite{Cotter13} on a separable Hilbert space,
with Gaussian reference measure $\pi_{0}$ and target $\pi \ll \pi_{0}$: the proposal
from $x$ is Gaussian with covariance $\beta^{2}C$, where $C$ is the covariance of
$\pi_{0}$ and $\beta \in (0,1)$. In infinite dimensions two centred Gaussian measures whose covariances are
proportional with ratio $\ne 1$ are mutually singular --- the operator entering the
Feldman--H\'ajek criterion is then a nonzero multiple of the identity, hence not
Hilbert--Schmidt --- so every proposal, and after $n$ steps every one of the finitely many
conditional laws indexed by the number of accepted moves, is carried by a
$\pi_{0}$-null set. Hence $\sing(T^{n}\circ\delta_{x}\mid\pi) = 1$ for every $n$, so
\hyp{S}{S} fails, and with it, by Lemma \ref{lem:SiffR}, does \eqref{prop:R}; and
indeed $\tv{T^{n}\circ\delta_{x} - \pi} = 1$ for every $n$, so
the conclusion of Theorem \ref{thm:criterion} fails as well. This is why the
convergence theory for that algorithm is developed in a Wasserstein distance rather
than in total variation \cite{HairerStuartVollmer14}. Both halves of the criterion fail here, and not only \hyp{S}{S}: since
$T^{n}_{x}$ is carried by a $\pi$-null set, $\sing(\pi \mid T^{n}_{x}) = 1$ as well,
so \hyp{L}{L} fails too. What survives is Proposition \ref{thm:dominated}: whenever
\eqref{prop:U} holds, the laws still converge for every initial distribution dominated
by a multiple of $\pi$, and the failure is entirely in the passage from an arbitrary
initial law to a dominated one. In Example \ref{ex:UnotR} that passage fails while
\hyp{P}{P} holds; here it fails and \hyp{P}{P} is unavailable as well, so the two are
not the same situation, though they fail at the same step.
\end{remark}

\section{The law of large numbers for ergodic Markov chains}\label{sec:ergodic}

What is used when a chain is run is not only that the law $T^{n}\circ\mu$ approaches
$\pi$, but that averages along a single trajectory converge to integrals against
$\pi$. This section records that statement. Unlike the rest of the note it is not
self-contained: it quotes, without proof, two classical results --- the canonical
construction of the chain as a stochastic process, and Birkhoff's pointwise ergodic
theorem. What the preceding sections contribute is precisely the hypothesis that
Birkhoff's theorem needs and that does not come for free, namely ergodicity.

\subsection{The Markov chain as a stochastic process}

\begin{definition}[Path space and shift]\label{def:pathspace}
Let $\XX^{\Nz}$ be the set of all sequences $\omega = (\omega_{0},\omega_{1},\dots)$
with $\omega_{n} \in \XX$, let
\[
  X_{n} : \XX^{\Nz} \longrightarrow \XX, \qquad X_{n}(\omega) := \omega_{n}
  \qquad (n \in \Nz),
\]
be the coordinate maps, and let $\XXA^{\otimes\Nz}$ be the smallest $\sigma$-algebra
on $\XX^{\Nz}$ making every $X_{n}$ measurable. Sets of the form
$\{X_{0} \in A_{0},\dots,X_{m} \in A_{m}\}$ with $m \in \Nz$ and
$A_{0},\dots,A_{m} \in \XXA$ are called \emph{rectangles}. They form a $\pi$-system
--- the intersection of two of them is again one --- and they generate
$\XXA^{\otimes\Nz}$; this is what is used to identify a measure on
$\XXA^{\otimes\Nz}$ by its values on them. They do \emph{not} form an algebra: the
complement of $\{X_{0}\in A_{0}, X_{1}\in A_{1}\}$ is in general not a rectangle. The
family
\[
  \mathcal{A} := \bigcup_{m \in \Nz} \sigma(X_{0},\dots,X_{m})
  \;\subseteq\; \XXA^{\otimes\Nz}
\]
\emph{is} an algebra --- it is an increasing union of $\sigma$-algebras --- and it
generates $\XXA^{\otimes\Nz}$, since it contains every rectangle. It is $\mathcal{A}$
that is used in the approximation step of Proposition \ref{prop:ergodic}.
The \emph{shift} is the map
\[
  \vartheta : \XX^{\Nz} \longrightarrow \XX^{\Nz}, \qquad
  (\vartheta\omega)_{n} := \omega_{n+1} ,
\]
which is measurable and satisfies $X_{n} \circ \vartheta^{k} = X_{n+k}$.
\end{definition}

\begin{theorem}[Canonical Markov chain]\label{thm:canonical}
Assume \hyp{A}{A}. For every $\mu \in \Prob$ there is a unique
probability measure $\Prb_{\mu}$ on $(\XX^{\Nz},\XXA^{\otimes\Nz})$ with
\[
  \Prb_{\mu}\bigl( X_{0} \in A_{0},\dots,X_{m} \in A_{m} \bigr)
  = \int_{A_{0}}\int_{A_{1}} \cdots \int_{A_{m}}
      T(\dd x_{m} \mid x_{m-1}) \cdots T(\dd x_{1} \mid x_{0})\, \mu(\dd x_{0})
\]
for all $m \in \Nz$ and $A_{0},\dots,A_{m} \in \XXA$. Here the differentials appear,
as usual, in the order opposite to that of the integral signs, so that
$\int_{A_{0}}$ goes with $\mu(\dd x_{0})$ and $\int_{A_{m}}$ with
$T(\dd x_{m} \mid x_{m-1})$. Read from the right, the measures then occur in the
order in which they act, matching the composition $T^{m}\circ\mu$ of
\ref{conv:kernelnotation}. Writing
$\Prb_{x} := \Prb_{\delta_{x}}$ and denoting by $\Ex_{\mu}$ the corresponding expectation,
one has moreover:
\begin{enumerate}[label=(\roman*),nosep,topsep=3pt]
  \item the map $x \mapsto \Prb_{x}(B)$ is measurable and
        $\Prb_{\mu}(B) = \int_{\XX} \Prb_{x}(B)\,\mu(\dd x)$, for every $B \in \XXA^{\otimes\Nz}$;
  \item \emph{(Markov property)} for every bounded measurable
        $F : \XX^{\Nz} \to \R$ and every $m \in \Nz$,
        \[
          \Ex_{\mu}\bigl[ F \circ \vartheta^{m} \mid \sigma(X_{0},\dots,X_{m}) \bigr]
          = \Ex_{X_{m}}[F] \qquad \Prb_{\mu}\text{-almost surely};
        \]
  \item the law of $X_{n}$ under $\Prb_{\mu}$ is $T^{n}\circ\mu$;
  \item the shift $\vartheta$ preserves $\Prb_{\pi}$, that is $\Prb_{\pi}\circ\vartheta^{-1} = \Prb_{\pi}$.
\end{enumerate}
\end{theorem}

\noindent
Existence and uniqueness are the theorem of Ionescu-Tulcea; see
\cite[Chapter~8]{Kallenberg21}, or \cite{MeynTweedie09} and \cite{Douc18} for the
same construction in the language of Markov chains. Items (i)--(iii) are part of that
construction, and (iv) follows from it, since by $T\circ\pi = \pi$ both $\Prb_{\pi}$ and
$\Prb_{\pi}\circ\vartheta^{-1}$ assign the displayed value to every cylinder. We use
these facts without further comment.

\subsection{Birkhoff's ergodic theorem}

\begin{definition}[Invariant sets, ergodicity]\label{def:ergodic}
Let $(\mathcal{Z},\mathcal{B}_{\mathcal{Z}},\Prb)$ be a probability space and let
$\vartheta : \mathcal{Z} \to \mathcal{Z}$ be measurable. One says that $\vartheta$
\emph{preserves} $\Prb$ if $\Prb \circ \vartheta^{-1} = \Prb$. The \emph{invariant
$\sigma$-algebra} is
\[
  \mathcal{I} := \bigl\{ B \in \mathcal{B}_{\mathcal{Z}} : \vartheta^{-1}B = B \bigr\}
  \subseteq \mathcal{B}_{\mathcal{Z}} ,
\]
and $\vartheta$ is called \emph{ergodic} for $\Prb$ if $\Prb(B) \in \{0,1\}$ for every
$B \in \mathcal{I}$.
\end{definition}

\begin{theorem}[Birkhoff's pointwise ergodic theorem]\label{thm:birkhoff}
Let $(\mathcal{Z},\mathcal{B}_{\mathcal{Z}},\Prb)$ be a probability space and let
$\vartheta : \mathcal{Z} \to \mathcal{Z}$ be measurable and preserve $\Prb$. Then for
every $F \in L^{1}(\Prb)$,
\[
  \frac{1}{n}\sum_{k=0}^{n-1} F \circ \vartheta^{k}
  \;\longrightarrow\; \Ex_{\Prb}\bigl[ F \mid \mathcal{I} \bigr]
  \qquad \Prb\text{-almost surely and in } L^{1}(\Prb) .
\]
If $\vartheta$ is ergodic for $\Prb$, the limit is the constant $\Ex_{\Prb}[F]$.

If moreover $F \in L^{r}(\Prb)$ for some number $r \in [1,\infty)$, then the
convergence holds in $L^{r}(\Prb)$ as well.
\end{theorem}

\noindent
The almost sure statement is the theorem of Birkhoff \cite{Birkhoff31}; the $L^{2}$
convergence is the mean ergodic theorem of von Neumann \cite{vonNeumann32},
published almost simultaneously. For the formulation above, with the conditional
expectation as limit and with $L^{1}$ convergence included, and for the approximation
of a set in a generated $\sigma$-algebra by sets of the generating algebra used
below, see \cite[Chapters~1 and~10]{Kallenberg21}; for the $L^{r}$ statement for
general $r \in [1,\infty)$, which follows from the almost sure convergence together
with the uniform integrability of the averages of $|F|^{r}$, see
\cite[Chapter~1]{Krengel85}. We quote all of this without proof; it is the only
input to this note that is not proved here.

\subsection{Ergodicity of the stationary chain under \texorpdfstring{\hyp{A}{A} and \hyp{P}{P} and \hyp{S}{S}}{(A), (P) and (S)}}

\begin{proposition}[Ergodicity]\label{prop:ergodic}
Assume \hyp{A}{A}, \hyp{P}{P} and \hyp{S}{S}.
Then the shift $\vartheta$ is ergodic for $\Prb_{\pi}$.
\end{proposition}

\begin{proof}
\emph{Step 1.} Let $A \in \mathcal{A}$, say $A \in \sigma(X_{0},\dots,X_{m})$ for some
$m \in \Nz$, and let $B \in \XXA^{\otimes\Nz}$ be arbitrary. We claim that
\[
  \Prb_{\pi}\bigl( A \cap \vartheta^{-n}B \bigr) \longrightarrow \Prb_{\pi}(A)\,\Prb_{\pi}(B)
  \qquad (n \to \infty).
\]
Put $g(x) := \Prb_{x}(B)$, a measurable map $\XX \to [0,1]$ by
Theorem \ref{thm:canonical}(i), and let $n \ge m$. Applying
Theorem \ref{thm:canonical}(ii) at time $m$ to $F := \ind_{B}\circ\vartheta^{\,n-m}$
and then again at time $0$ gives
\[
  \Prb_{\pi}\bigl( A \cap \vartheta^{-n}B \bigr)
  = \Ex_{\pi}\bigl[ \ind_{A}\, \Prb_{X_{m}}\bigl( \vartheta^{-(n-m)}B \bigr) \bigr]
  = \Ex_{\pi}\bigl[ \ind_{A}\, h_{n-m}(X_{m}) \bigr],
  \qquad
  h_{k}(x) := \int_{\XX} g \dd(T^{k}\circ\delta_{x}) .
\]
Since $0 \le g \le 1$, Definition \ref{def:tv} gives
\[
  \Bigl| h_{k}(x) - \int_{\XX} g \dd\pi \Bigr|
  = \Bigl| \int_{\XX} g \dd\bigl( T^{k}\circ\delta_{x} - \pi \bigr) \Bigr|
  \;\le\; 2\,\tv{T^{k}\circ\delta_{x} - \pi} \;\longrightarrow\; 0
\]
for every $x \in \XX$, by Theorem \ref{thm:criterion}, whose hypotheses
\eqref{prop:U} and \eqref{prop:R} hold by Lemmas \ref{lem:PimpliesU} and
\ref{lem:SiffR}. Moreover
$\int_{\XX} g \dd\pi = \int_{\XX} \Prb_{x}(B)\,\pi(\dd x) = \Prb_{\pi}(B)$ by
Theorem \ref{thm:canonical}(i). As $|h_{k}| \le 1$, bounded convergence yields
$\Ex_{\pi}[\ind_{A} h_{n-m}(X_{m})] \to \Prb_{\pi}(A) \Prb_{\pi}(B)$, which is the claim.

\emph{Step 2.} Let $B \in \mathcal{I}$, so that $\vartheta^{-n}B = B$ for every
$n \in \Nz$, and let $\varepsilon \in (0,1)$ be a number. Since $\mathcal{A}$ is an
algebra generating $\XXA^{\otimes\Nz}$ and $\Prb_{\pi}$ is a finite measure, there is
a set $A \in \mathcal{A}$ with
$\Prb_{\pi}(A \,\triangle\, B) \le \varepsilon$ (see \cite[Chapter~1]{Kallenberg21}).
Then
\[
  \Prb_{\pi}(B) = \Prb_{\pi}\bigl( B \cap \vartheta^{-n}B \bigr)
  \quad\text{and}\quad
  \bigl| \Prb_{\pi}\bigl( B \cap \vartheta^{-n}B \bigr) - \Prb_{\pi}\bigl( A \cap \vartheta^{-n}B \bigr) \bigr|
  \le \Prb_{\pi}(A \,\triangle\, B) \le \varepsilon .
\]
Letting $n \to \infty$ and using Step 1 together with
$|\Prb_{\pi}(A) - \Prb_{\pi}(B)| \le \varepsilon$ gives
$\bigl| \Prb_{\pi}(B) - \Prb_{\pi}(B)^{2} \bigr| \le 2\varepsilon$. As
$\varepsilon \in (0,1)$ was arbitrary, $\Prb_{\pi}(B) = \Prb_{\pi}(B)^{2}$, so
$\Prb_{\pi}(B) \in \{0,1\}$.
\end{proof}

\begin{remark}[The hypotheses enter only through the convergence theorem]\label{rem:ergodichyp}
Assumptions \hyp{P}{P} and \hyp{S}{S} are used in the proof of
Proposition \ref{prop:ergodic}, and hence in everything that follows in this section,
only through the single conclusion
$\tv{T^{k}\circ\delta_{x}-\pi} \to 0$ for every $x \in \XX$; the same is true of
the two other places where Theorem \ref{thm:criterion} is invoked below, in the proof
of Proposition \ref{prop:weaklaw}, and of the appeal to
Remark \ref{rem:aposteriori} in the proof of Corollary \ref{cor:lln}, since that
remark is itself derived from that conclusion by Lemma \ref{lem:pointwise}. Any
hypothesis delivering it therefore serves equally well. In particular the whole of Section \ref{sec:ergodic} holds
verbatim under \hyp{A}{A}, \hyp{J}{J}, \hyp{L}{L} and \hyp{S}{S}, by
Theorem \ref{thm:asympequiv} in place of Theorem \ref{thm:criterion}.
\end{remark}

\subsection{The strong law of large numbers under \texorpdfstring{\hyp{A}{A}, \hyp{P}{P} and \hyp{S}{S}}{(A), (P) and (S)}, with Birkhoff}

The proof below needs one small fact about path space, which has nothing to do with
Markov chains and is separated out because it is used twice and is worth having by
itself: the set on which the Ces\`aro averages of an observable converge to a
prescribed value is not merely almost invariant under the shift, but invariant as a
set, which is what the definition of $\mathcal{I}$ in Definition \ref{def:ergodic}
demands.

\begin{lemma}[Ces\`aro convergence sets are strictly shift-invariant]\label{lem:cesaroinvariant}
Assume \hyp{A1}{A1} and let $f : \XX \to \R$ be measurable and $L \in \R$ a number.
Put
\[
  \Gamma := \Bigl\{ \omega \in \XX^{\Nz} :\
     \lim_{n\to\infty} \frac{1}{n}\sum_{k=0}^{n-1} f\bigl( X_{k}(\omega) \bigr) = L \Bigr\} .
\]
Then $\Gamma \in \XXA^{\otimes\Nz}$ and $\vartheta^{-1}\Gamma = \Gamma$, with equality
of sets and not merely up to a null set.
\end{lemma}

\begin{proof}
\emph{Measurability.} Each map $S_{n} := \sum_{k=0}^{n-1} f \circ X_{k}$ is measurable,
the $X_{k}$ being measurable by Definition \ref{def:pathspace}; and
\[
  \Gamma = \bigcap_{r \in \N}\ \bigcup_{N \in \N}\ \bigcap_{n \ge N}
     \Bigl\{ \bigl| \tfrac{1}{n}S_{n} - L \bigr| \le \tfrac1r \Bigr\}
\]
is a countable combination of measurable sets.

\emph{Invariance.} Write $S_{n}'(\omega) := \sum_{k=1}^{n} f(X_{k}(\omega)) = S_{n}(\vartheta\omega)$.
\begin{itemize}[nosep,topsep=3pt,leftmargin=1.4em]
  \item If $\omega \in \Gamma$, i.e.\ $S_{n}(\omega)/n \to L$, then also
        $S_{n+1}(\omega)/(n+1) \to L$, so that
        \[
          \frac{f(X_{n}(\omega))}{n}
          = \frac{S_{n+1}(\omega)-S_{n}(\omega)}{n}
          = \frac{n+1}{n}\cdot\frac{S_{n+1}(\omega)}{n+1} - \frac{S_{n}(\omega)}{n}
          \longrightarrow L - L = 0 ;
        \]
        hence
        $S_{n}'(\omega)/n = S_{n}(\omega)/n + \bigl(f(X_{n}(\omega))-f(X_{0}(\omega))\bigr)/n
        \to L$, i.e.\ $\vartheta\omega \in \Gamma$.
  \item Conversely, if $\vartheta\omega \in \Gamma$, i.e.\ $S_{n}'(\omega)/n \to L$,
        then
        $S_{n}(\omega)/n = f(X_{0}(\omega))/n + \frac{n-1}{n}\cdot S_{n-1}'(\omega)/(n-1)
        \to L$, i.e.\ $\omega \in \Gamma$.
\end{itemize}
The first item is the step that is not merely bookkeeping: the vanishing of
$f(X_{n})/n$ has to be \emph{deduced} from the convergence of the averages, not
assumed. Together the two give $\vartheta^{-1}\Gamma = \Gamma$.
\end{proof}

\begin{corollary}[Convergence of ergodic averages]\label{cor:lln}
Assume \hyp{A}{A}, \hyp{P}{P} and \hyp{S}{S},
and let $\mu \in \Prob$. Then:
\begin{enumerate}[label=(\roman*),nosep,topsep=3pt]
  \item for every $f \in L^{1}(\pi)$,
        \[
          \frac{1}{n}\sum_{k=0}^{n-1} f(X_{k}) \;\longrightarrow\; \int_{\XX} f \dd\pi
          \qquad \Prb_{\mu}\text{-almost surely};
        \]
  \item if $f$ is bounded and measurable, the convergence in (i) holds in addition in
        $L^{r}(\Prb_{\mu})$ for every number $r \in [1,\infty)$;
  \item if $\mu \le M\pi$ for some number $M \in [1,\infty)$, then for every number
        $r \in [1,\infty)$ and every $f \in L^{r}(\pi)$ the convergence in (i) holds
        in addition in $L^{r}(\Prb_{\mu})$.
\end{enumerate}
\end{corollary}

\begin{proof}
Fix $f \in L^{1}(\pi)$ and put $F := f \circ X_{0}$, so that
$F \circ \vartheta^{k} = f \circ X_{k}$ and, by
Theorem \ref{thm:canonical}(iii), $\Ex_{\pi}|F| = \int_{\XX}|f|\dd\pi < \infty$. By
Theorem \ref{thm:canonical}(iv) the shift preserves $\Prb_{\pi}$, and by
Proposition \ref{prop:ergodic} it is ergodic for $\Prb_{\pi}$; so
Theorem \ref{thm:birkhoff} gives the assertion for $\mu = \pi$, both almost surely
and in $L^{1}(\Prb_{\pi})$.

Let $\Gamma$ be the set of those $\omega \in \XX^{\Nz}$ for which
$n^{-1}\sum_{k<n} f(X_{k}(\omega))$ converges to $L := \int f \dd\pi$. By
Lemma \ref{lem:cesaroinvariant}, $\Gamma$ is a measurable subset of $\XX^{\Nz}$ and
$\vartheta^{-1}\Gamma = \Gamma$ as sets, so $\Gamma \in \mathcal{I}$; and
$\Prb_{\pi}(\Gamma) = 1$ by the previous paragraph.
By Theorem \ref{thm:canonical}(i), $\int_{\XX} \Prb_{x}(\Gamma)\,\pi(\dd x) = 1$,
whence
\[
  \Prb_{x}(\Gamma) = 1 \qquad\text{for } \aeevery\ x \in \XX .
\]

Now let $\mu \in \Prob$ be arbitrary and let $\varepsilon \in (0,1)$ be a number.
Hypotheses \eqref{prop:R} and \hyp{S}{S} are imposed at the points of $\XX$ only, but
by Remark \ref{rem:aposteriori} they hold for every initial law; so there are
$n_{1} \in \Nz$, $M \in [1,\infty)$ and $\rho \in \Prob$
with $\rho \le M\pi$ and $\tv{T^{n_{1}}\circ\mu - \rho} \le \varepsilon$. Write
$\alpha := T^{n_{1}}\circ\mu$ and let $\alpha - \rho = (\alpha-\rho)^{+} -
(\alpha-\rho)^{-}$ be the Jordan decomposition; then
\[
  \sigma := \alpha - (\alpha - \rho)^{+}
\]
is a nonnegative measure: if $E$ is a Hahn set for $\alpha-\rho$, then
$\sigma(A) = \alpha(A) - (\alpha-\rho)(A\cap E) \ge \alpha(A) - \alpha(A \cap E) \ge 0$.
Moreover $\sigma \le \alpha$, and
$\sigma = \rho - (\alpha-\rho)^{-} \le \rho \le M\pi$, so
$\sigma \ll \pi$, and $\sigma(\XX) = 1 - (\alpha-\rho)^{+}(\XX) = 1 -
\tv{\alpha-\rho} \ge 1-\varepsilon$ by Definition \ref{def:tv}. Using
$\vartheta^{-n_{1}}\Gamma = \Gamma$, then Theorem \ref{thm:canonical}(ii) and (iii),
and finally $\Prb_{x}(\Gamma) = 1$ for $\sigma$-almost every $x$,
\[
  \Prb_{\mu}(\Gamma) = \Prb_{\mu}\bigl( \vartheta^{-n_{1}}\Gamma \bigr)
  = \Ex_{\mu}\bigl[ \Prb_{X_{n_{1}}}(\Gamma) \bigr]
  = \int_{\XX} \Prb_{x}(\Gamma)\, \alpha(\dd x)
  \;\ge\; \int_{\XX} \Prb_{x}(\Gamma)\, \sigma(\dd x)
  = \sigma(\XX) \;\ge\; 1-\varepsilon .
\]
As $\varepsilon \in (0,1)$ was arbitrary, $\Prb_{\mu}(\Gamma) = 1$. This proves (i).

(ii) The averages $n^{-1}\sum_{k<n}f(X_{k})$ are bounded in absolute value by
$\sup_{x}|f(x)|$, so (i) and dominated convergence give the assertion.

(iii) If $\mu \le M\pi$ then $\Prb_{\mu} \le M\,\Prb_{\pi}$, because
$\Prb_{\mu}(B) = \int \Prb_{x}(B)\mu(\dd x) \le M \int \Prb_{x}(B)\pi(\dd x)
= M\,\Prb_{\pi}(B)$ for every $B \in \XXA^{\otimes\Nz}$, by
Theorem \ref{thm:canonical}(i); hence
$\lVert \cdot \rVert_{L^{r}(\Prb_{\mu})}^{r} \le M \lVert \cdot \rVert_{L^{r}(\Prb_{\pi})}^{r}$.
For $f \in L^{r}(\pi)$ one has $F = f \circ X_{0} \in L^{r}(\Prb_{\pi})$ by
Theorem \ref{thm:canonical}(iii), so the last part of Theorem \ref{thm:birkhoff}
gives convergence in $L^{r}(\Prb_{\pi})$, and therefore in $L^{r}(\Prb_{\mu})$.
\end{proof}

\begin{remark}[The restriction in (iii) is necessary]\label{rem:lprestriction}
Part (iii) cannot be extended to arbitrary initial distributions. Throughout this
section $f$ is a genuine map $\XX \to \R$, so that $f(X_{k})$ is defined at every
point of path space; but a set of $\pi$-measure zero, invisible to the norm of
$L^{r}(\pi)$, is not invisible to $\Prb_{\delta_{x}}$. Concretely, for
$\mu = \delta_{x}$ and an unbounded $f \in L^{r}(\pi)$ the value $|f(x)|$ may be
arbitrarily large, so that
$\lVert n^{-1}\sum_{k<n}f(X_{k})\rVert_{L^{r}(\Prb_{x})} \ge |f(x)|/n$ is not
controlled by $\lVert f \rVert_{L^{r}(\pi)}$ at all, uniformly in $x$. The almost sure
statement (i) is unaffected by this, since it concerns the limit only.
\end{remark}

\subsection{The weak law of large numbers, without Birkhoff}

The hypotheses of this note are considerably stronger than measure preservation, and
it is natural to ask whether they permit a short self-contained substitute for
Theorem \ref{thm:birkhoff}. They do, but only for the weak law.

\begin{proposition}[Weak law of large numbers]\label{prop:weaklaw}
Assume \hyp{A}{A}, \hyp{P}{P} and \hyp{S}{S}.
Then for every bounded measurable map $f : \XX \to \R$ and every $\mu \in \Prob$,
\[
  \frac{1}{n}\sum_{k=0}^{n-1} f(X_{k}) \;\longrightarrow\; \int_{\XX} f \dd\pi
  \qquad\text{in } L^{2}(\Prb_{\mu}), \text{ hence in } \Prb_{\mu}\text{-probability.}
\]
More precisely, if $\int_{\XX} f \dd\pi = 0$ and $\sup_{x \in \XX}|f(x)| \le 1$ ---
to which the general case reduces by an affine change of $f$ --- then, with the
numbers
\[
  u_{m} := \int_{\XX} \bigl| (T^{m}f)(x) \bigr|\,\pi(\dd x),
  \qquad
  v_{j} := \tv{T^{j}\circ\mu - \pi}
  \qquad (j,m \in \Nz),
\]
where $(T^{m}f)(x) := \int_{\XX} f \dd(T^{m}\circ\delta_{x})$, one has the explicit
variance bound
\begin{equation}
  \Ex_{\mu}\Bigl[ \Bigl( \frac{1}{n}\sum_{k=0}^{n-1} f(X_{k}) \Bigr)^{\!2} \Bigr]
  \;\le\; \frac{1}{n} + \frac{2}{n}\sum_{m=1}^{n-1} u_{m} + \frac{4}{n}\sum_{j=0}^{n-1} v_{j}
  \qquad (n \in \N),
  \label{eq:weaklawbound}
\end{equation}
in which $u_{m} \to 0$ and $v_{j} \to 0$, so that the right-hand side tends to $0$.
\end{proposition}

\begin{proof}
Put $\bar f := f - \int_{\XX} f \dd\pi$ and $c := \sup_{x \in \XX}|\bar f(x)|$. If
$c = 0$ then $f$ is constant and the assertion is trivial. Otherwise replace $f$ by
$\bar f/c$; this is an affine change of $f$, so it changes both sides of the assertion
by the same affine map, and after it $\int_{\XX} f \dd\pi = 0$ and
$\sup_{x}|f(x)| \le 1$. (Note that dividing the \emph{uncentred} $f$ by
$\sup_{x}|f(x)|$ would not suffice: centring can enlarge the supremum norm.)

Write $(T^{m}f)(x) := \int_{\XX} f \dd(T^{m}\circ\delta_{x})$; this is a bounded
measurable map $\XX \to \R$, since $T^{m}$ is a Markov kernel and $f$ is bounded and
measurable, and $\sup_{x}|(T^{m}f)(x)| \le \sup_{x}|f(x)| \le 1$. Put
\[
  u_{m} := \int_{\XX} \bigl| (T^{m}f)(x) \bigr|\,\pi(\dd x),
  \qquad
  v_{j} := \tv{T^{j}\circ\mu - \pi}
  \qquad (j,m \in \Nz),
\]
two sequences of numbers in $[0,1]$. Both tend to $0$. For $v_{j}$ this is
Theorem \ref{thm:criterion}. For $u_{m}$, note first that, since $\int f \dd\pi = 0$,
\begin{equation}
  \bigl| (T^{m}f)(x) \bigr|
  = \Bigl| \int_{\XX} f \dd\bigl(T^{m}\circ\delta_{x} - \pi\bigr) \Bigr|
  \;\le\; \sup_{x'}|f(x')| \cdot \Bigl( h^{+}(\XX) + h^{-}(\XX) \Bigr)
  \;\le\; 2\,\tv{T^{m}\circ\delta_{x} - \pi}
  \label{eq:Tmfbound}
\end{equation}
for every $x \in \XX$, where $h := T^{m}\circ\delta_{x} - \pi$ satisfies $h(\XX)=0$,
so that $h^{\pm}(\XX) = \tv{h}$ by Definition \ref{def:tv}. By
Theorem \ref{thm:criterion} the right-hand side of \eqref{eq:Tmfbound} tends to $0$ for
every fixed $x \in \XX$; since $|T^{m}f| \le 1$, dominated convergence gives
$u_{m} \to 0$.

It is essential here that the \emph{integrand} is $|T^{m}f|$ and not the map
$x \mapsto \tv{T^{m}\circ\delta_{x}-\pi}$; see Remark \ref{rem:tvnotmeasurable}
below.

By the Markov property, Theorem \ref{thm:canonical}(ii) applied to
$F := f \circ X_{m}$, and by Theorem \ref{thm:canonical}(iii),
\[
  \Ex_{\mu}\bigl[ f(X_{j})\,f(X_{j+m}) \bigr]
  = \Ex_{\mu}\bigl[ f(X_{j})\,(T^{m}f)(X_{j}) \bigr]
  = \int_{\XX} f\cdot(T^{m}f) \dd(T^{j}\circ\mu) .
\]
Since $\sup_{x}|f(x)| \le 1$ we get
\[
  \Bigl| \int_{\XX} f\,(T^{m}f) \dd\pi \Bigr| \;\le\; \int_{\XX}|T^{m}f| \dd\pi \;=\; u_{m} ;
\]
and since $\sup_{x}|f(x)(T^{m}f)(x)| \le 1$, the same computation as in
\eqref{eq:Tmfbound},
applied to the signed measure $T^{j}\circ\mu - \pi$ of total mass $0$, gives
\[
  \Bigl| \int_{\XX} f\,(T^{m}f) \dd(T^{j}\circ\mu) - \int_{\XX} f\,(T^{m}f) \dd\pi \Bigr|
  \;\le\; 2\,v_{j} .
\]
Hence $\bigl| \Ex_{\mu}[ f(X_{j})f(X_{j+m}) ] \bigr| \le u_{m} + 2 v_{j}$ for
$m \ge 1$, and it is at most $1$ for $m = 0$. Splitting the double sum into its
diagonal, which contributes $n$ terms bounded by $1$, and its off-diagonal part,
which consists of the pairs $\{j,j+m\}$ with $m \ge 1$ counted twice,
\[
  \Ex_{\mu}\Bigl[ \Bigl( \frac{1}{n}\sum_{k=0}^{n-1} f(X_{k}) \Bigr)^{2} \Bigr]
  = \frac{1}{n^{2}}\sum_{j,k=0}^{n-1} \Ex_{\mu}\bigl[ f(X_{j})f(X_{k}) \bigr]
  \;\le\; \frac{1}{n} + \frac{2}{n}\sum_{m=1}^{n-1}u_{m}
        + \frac{4}{n}\sum_{j=0}^{n-1}v_{j} ,
\]
which is \eqref{eq:weaklawbound}; and all three terms tend to $0$, the last two
because the Ces\`aro means of a null sequence are null.
\end{proof}

\begin{remark}[Why $|T^{m}f|$ and not $\tv{T^{m}\circ\delta_{x}-\pi}$]\label{rem:tvnotmeasurable}
The proof above integrates $|T^{m}f|$ against $\pi$. The seemingly more natural
quantity $\int_{\XX}\tv{T^{m}\circ\delta_{x}-\pi}\,\pi(\dd x)$ is not available: by
Definition \ref{def:tv} the integrand is
$\sup_{A \in \XXA}|T^{m}(A \mid x) - \pi(A)|$, a supremum of measurable maps of $x$
over the \emph{uncountable} index set $\XXA$, and on a general measurable space
nothing in Remark \ref{rem:nostructure} makes it measurable. It \emph{is} measurable
under either of two additional hypotheses, both of which are avoided here: if $\XXA$ is
countably generated, that is under \hyp{C}{C}, because then $\sup_{A \in \XXA}|h(A)| = \sup_{A \in \mathcal{A}_{0}}|h(A)|$ for every
$h \in \Sign$ and every countable generating algebra $\mathcal{A}_{0}$, by
approximation in $|h|$; or if $T^{m}$ has a jointly measurable density $t_{m}$ with
respect to $\pi$, because then
$\tv{T^{m}\circ\delta_{x}-\pi} = \tfrac12\int_{\XX}|t_{m}(y\mid x)-1|\pi(\dd y)$,
which is measurable in $x$ by Tonelli. The map $T^{m}f$, by contrast, is measurable
under \hyp{A}{A} alone, by Definition \ref{def:kernel}, and the
pointwise bound \eqref{eq:Tmfbound} is all the proof needs.
\end{remark}

\begin{remark}[Why the strong law appears to need more]\label{rem:whybirkhoff}
The proof just given uses only Theorem \ref{thm:criterion} and the Markov property,
and is three times shorter than any proof of Theorem \ref{thm:birkhoff}. It yields
convergence in probability, and the gap to almost sure convergence does not close by
the same means. Passing from $L^{2}$ convergence to almost sure convergence along the
full sequence would require the variances above to be summable along a subsequence,
that is, a \emph{rate} in Theorem \ref{thm:abstract}; and no rate is available under
\hyp{P}{P} and \hyp{S}{S} alone, nor under \hyp{D}{D}
(Remark \ref{rem:notclaimed}). When a rate is assumed the strong law does have short
proofs, for instance through the Poisson equation $g - Tg = f - \int f \dd\pi$, whose
bounded solution exists as soon as $\sum_{n}\sup_{x}|(T^{n}f)(x)|$ converges, and the
martingale strong law; this is the route taken under geometric ergodicity, for which
see \cite{Gallegos24} and \cite{MeynTweedie09}.

Under the present hypotheses we know of no route to the strong law that avoids both a
rate and Theorem \ref{thm:birkhoff}. The reason is structural: the difficulty of
Birkhoff's theorem lies in the almost sure convergence of
$n^{-1}\sum_{k<n} F \circ \vartheta^{k}$ under measure preservation alone, and
ergodicity --- which is exactly what the hypotheses of this note supply, by
Proposition \ref{prop:ergodic} --- serves only to identify the limit. Strengthening
the mixing hypothesis therefore does not simplify the hard part. The other classical
route, through Harris recurrence and regeneration \cite{AsmussenGlynn11}, is short
but reintroduces the irreducibility machinery that these notes set out to avoid.
\end{remark}

\noindent
Corollary \ref{cor:lln} applies in each of the five settings of this note, since
\hyp{P}{P} and \hyp{S}{S} were verified under \hyp{D}{D} in the proof of
Theorem \ref{thm:main}, under \hyp{M}{M} in the proof of Theorem \ref{thm:mh}, under
\hyp{E}{E} in the proof of Theorem \ref{thm:general}, and directly in
Corollaries \ref{cor:gibbs} and \ref{cor:tempering}. Together with
Corollaries \ref{cor:reference}, \ref{cor:mh} and \ref{cor:referencegen} this proves
Corollary \ref{cor:llnintro} of the introduction.

\section{Discussion}\label{sec:discussion}

\subsection{Summary}

One convergence theorem has been proved and then verified five times. Its
hypotheses are \hyp{P}{P}, a strictly positive minorant of the transition density
after finitely many steps, and \hyp{S}{S}, the vanishing of the singular mass of the
law of the chain started at $x$, for every starting point $x$; from these,
Theorem \ref{thm:criterion} gives convergence in total
variation from every initial distribution, and uniqueness of the invariant measure.
Two observations sharpen the statement without changing the proof. Convergence from
every starting \emph{point} already implies convergence from every initial
\emph{law} (Lemma \ref{lem:pointwise}), by an argument that fixes one Hahn set per
time step and so needs no measurability of $x \mapsto \tv{T^{n}_{x}-\pi}$;
consequently \eqref{prop:R} and \hyp{S}{S} are imposed at the points of $\XX$ alone,
which is also where they are checked; that they then hold for every initial law is a
consequence and not a hypothesis (Remark \ref{rem:aposteriori}).

Under one further hypothesis on the kernel, the criterion can then be stated with no
density in it at all. Assumption \hyp{L}{L} asks only that
$\sing(\pi \mid T^{n}_{x}) \to 0$ for $\aeevery$ starting point $x$ --- that the part
of the target which the law of the chain does not see becomes negligible --- with no
density required to be jointly measurable, or indeed to be exhibited at all, and with
no exact domination demanded at any finite time. It is implied by \hyp{P}{P}
(Lemma \ref{lem:PimpliesL}); each of the five verifications establishes it in the
stronger form $\pi \ll T^{n}_{x}$ (Remark \ref{rem:Lverification}), but the gap
between the two is real and is exactly the case of a local proposal on an unbounded
state space (Remark \ref{rem:whylimit}); and it is equivalent to \hyp{P}{P} as soon as
the absolutely continuous part of each iterate has a jointly measurable density, which
is the hypothesis \hyp{J}{J} (Proposition \ref{prop:LimpliesP}), of the same kind as
the measurable transition densities assumed in \hyp{D}{D} and \hyp{E}{E}; it holds
whenever $\XXA$ is countably generated (Proposition \ref{prop:jointdensity}), by a
martingale construction whose classical one-measure form is the theorem of Andersen and
Jessen.

Combining this with \hyp{S}{S} gives Theorem \ref{thm:asympequiv}, stated as
Theorem \ref{thm:crit} in the introduction. Both hypotheses are then statements about
the Lebesgue decomposition of $T^{n}_{x}$ at a single starting point, and nothing else
appears; and they are the two halves of a single relation, $T^{n}_{x} \sim \pi$, each
asked in the limit rather than at a finite time. Read that way the criterion is not
only sufficient but \emph{necessary}: the conclusion bounds both singular masses by
$\tv{T^{n}_{x}-\pi}$, so asymptotic equivalence with the target characterises
convergence. What that form costs, and all it costs, is \hyp{J}{J}; the proofs of
Sections \ref{sec:first}--\ref{sec:applications} are kept through \hyp{P}{P}, which
needs no measurable selection at all, and which in three of the five settings is met by
a strict minorant rather than by a density (Remark \ref{rem:PvsJ}). No convergence
theorem in these notes assumes anything of the state space.

The verifications are: \hyp{D}{D}, a strictly positive transition density with
respect to $\pi$ (Theorem \ref{thm:main}); \hyp{M}{M}, an absolutely continuous part
with strictly positive density together with an atom at the starting point, which
brings the Metropolis--Hastings algorithm within reach (Theorem \ref{thm:mh});
\hyp{E}{E}, a strictly positive density only after a number of steps depending on the
starting point (Theorem \ref{thm:general}); and the two algorithms of
Section \ref{sec:applications}, whose kernels are singular with respect to $\pi$ at
every step, so that none of the three settings applies to them
(Corollaries \ref{cor:gibbs} and \ref{cor:tempering}). Each of the three settings has
a counterpart phrased with densities against a $\sigma$-finite reference measure
(Corollaries \ref{cor:reference}, \ref{cor:mhreference} and
\ref{cor:referencegen}), and all of them yield the convergence of ergodic averages
along a trajectory (Corollary \ref{cor:lln}).

The mechanism divides in two. The convergence theorem itself
(Theorem \ref{thm:abstract}) rests on two lemmas only: the distance to $\pi$
contracts (Lemma \ref{lem:contraction}) and domination by $M\pi$ is preserved because
$\pi$ is invariant (Lemma \ref{lem:domination}); the renormalisation in
Proposition \ref{thm:dominated} then converts a stalled distance into a
contradiction. Its two conditions \eqref{prop:U} and \eqref{prop:R} isolate what has
to be true of dominated laws and of the chain started at a point respectively. Producing them from
\hyp{P}{P} and \hyp{S}{S} rests on four further lemmas: a common minorant bounds the
distance (Lemma \ref{lem:minorant}), positivity yields a quantitative lower bound on
a set of substantial measure (Lemma \ref{lem:positivity}), a density may be truncated
(Lemma \ref{lem:truncation}), and absolute continuity is never lost
(Lemma \ref{lem:monotone}(i), in quantitative form Lemma \ref{lem:monotone}(iv)).
Neither hypothesis is implied by the other, and neither may be dropped
(Examples \ref{ex:RnotU} and \ref{ex:UnotR}).

\begin{remark}[The main hypotheses compared]\label{rem:comparison}
The logical relations between \hyp{D}{D}, \hyp{M}{M}, \hyp{E}{E}, \hyp{P}{P} and \hyp{S}{S} are the
following.
\begin{enumerate}[label=\textup{(\roman*)},nosep,topsep=3pt,leftmargin=2.4em]
  \item \hyp{D}{D} implies \hyp{M}{M}, with $k := t$ and $r := 0$: then
        $\theta \equiv 1$, so \hyp{M}{M}(a) and \hyp{M}{M}(b) hold, and \hyp{M}{M}(c)
        holds because $t > 0$ everywhere.
  \item \hyp{D}{D} implies \hyp{E}{E}, with $\XX_{n} := \XX$ for every $n$
        (Remark \ref{rem:CK}).
  \item \hyp{M}{M} does \emph{not} imply \hyp{E}{E}. If $\pi$ is atomless and
        $r(x) > 0$ for every $x$, then $T^{n}_{x} \ge r(x)^{n}\delta_{x}$
        has an atom at $x$ of positive mass while $\pi(\{x\}) = 0$, so no iterate is
        absolutely continuous with respect to $\pi$ and \hyp{E}{E}(b) fails for
        every $n$ and every $x$.

        The hypothesis $r > 0$ everywhere is met by the plainest example there is.
        Take $\XX := \R$ with Lebesgue measure, the standard normal target
        $p := \varphi$, and the Gaussian random walk proposal
        $q(y \mid x) := \varphi(y-x)$, which is symmetric, so that
        $a(y \mid x) = \min\{1, \varphi(y)/\varphi(x)\}$. Then $a(y \mid x) < 1$
        exactly on $\{ y : \varphi(y) < \varphi(x) \} = \{ y : |y| > |x| \}$, a set of
        infinite Lebesgue measure on which $q(\,\cdot\mid x)$ is strictly positive, so
        that
        \[
          r(x) = \int_{\R} \bigl( 1 - a(y \mid x) \bigr)\,q(y \mid x)\,\lambda(\dd y)
          \;>\; 0 \qquad\text{for every } x \in \R .
        \]
        Every point therefore carries an atom at every time. This is the generic case
        for a Metropolis--Hastings kernel on a continuous state space, and it is why
        Section \ref{sec:mh} cannot be dispensed with in favour of
        Section \ref{sec:general}.
  \item \hyp{E}{E} does \emph{not} imply \hyp{M}{M}. Let $\XX := \{1,2,3\}$ with
        $\pi$ the uniform distribution and let $T$ be the stochastic matrix with rows
        $T_{1} = (0,\tfrac12,\tfrac12)$,
        $T_{2} = (\tfrac12,0,\tfrac12)$,
        $T_{3} = (\tfrac12,\tfrac12,0)$. Then $\pi$ is invariant (the
        matrix is doubly stochastic), $T^{2}$ has all entries strictly positive, so
        \hyp{E}{E} holds with $\XX_{1} := \emptyset$ and $\XX_{n} := \XX$ for
        $n \ge 2$, whereas \hyp{D}{D} fails. And \hyp{M}{M} fails too: a
        representation as in \hyp{M}{M}(a) forces
        $k(y \mid x) = 3\,T(\{y\}\mid x)$ for $y \ne x$, hence
        $\theta(x) = \tfrac13\sum_{y \ne x} k(y\mid x) + \tfrac13 k(x\mid x)
        = 1 + \tfrac13 k(x \mid x)$, so $\theta(x) \le 1$ forces $k(x\mid x) = 0$,
        contradicting \hyp{M}{M}(c) on the finite space $\XX$, where every point has
        positive $\pi$-measure.
\end{enumerate}
Thus \hyp{M}{M} and \hyp{E}{E} are two incomparable weakenings of \hyp{D}{D},
which is why Sections \ref{sec:mh} and \ref{sec:general} are independent of one
another and why both are needed. Each of the three implies \hyp{P}{P} together with
\hyp{S}{S}, and the implication is strict: the random scan Gibbs kernel and the
parallel tempering kernel satisfy the latter pair and none of the former three
(Lemmas \ref{lem:gibbssingular} and \ref{lem:temperingsingular}). The hypotheses are
therefore ordered as
\[
  \hyp{D}{D} \implies \hyp{M}{M} \implies \hyp{P}{P} \wedge \hyp{S}{S}
  \implies \hyp{L}{L} \wedge \hyp{S}{S},
  \qquad
  \hyp{D}{D} \implies \hyp{E}{E} \implies \hyp{P}{P} \wedge \hyp{S}{S}
  \implies \hyp{L}{L} \wedge \hyp{S}{S},
\]
with \hyp{M}{M} and \hyp{E}{E} incomparable and all the implications strict except
possibly the last, which is an equivalence under \hyp{J}{J}. The
last implication is Lemma \ref{lem:PimpliesL}; the reverse one, under \hyp{J}{J}, is
Proposition \ref{prop:LimpliesP}, and whether it can fail without \hyp{J}{J} we do not
know. The individual verifications are
collected in Remarks \ref{rem:roadmap} and \ref{rem:Lverification}.
\end{remark}

\subsection{Two directions not pursued}\label{ssec:directions}

Two limitations of the criterion were met in the course of the applications, and are
recorded here because each is a definite gap with a definite shape, and because
Remarks \ref{rem:temperingtheta} and \ref{rem:otheralgorithms} point to this
subsection. Neither is pursued.

\begin{enumerate}[label=\textup{(\alph*)},ref=\textup{(\alph*)},leftmargin=2.6em,itemsep=6pt,topsep=4pt]

\item\label{dir:local} \emph{A localised minorisation.}
Proposition \ref{prop:minorisation}, and with it \hyp{P}{P}, asks for a minorant $s$
that carries mass from the set $\YY$ of starting points to all but an
arbitrarily small $\pi$-proportion of the endpoints of $\XX$. Some algorithms supply instead a minorant available only \emph{between a pair
of sets}: from $\YY$ into some third set $\YY''$, of possibly small $\pi$-measure, and
the passage from $\YY''$ to the rest of the space achieved only after further steps whose
number is not bounded uniformly. Reversible jump Markov chain Monte Carlo
\cite{Green95} is of this kind: birth and death moves change the model index by one,
so that from a model of dimension $j$ the chain cannot reach $\aeevery$ point of a
model of dimension $j'$ in a number of steps bounded independently of $|j - j'|$, and
\hyp{P}{P} fails as stated even though the algorithm converges.

The weakened form of \hyp{L}{L} adopted here removes part of this obstacle, and it is
worth saying which part. What \hyp{L}{L} asks of such a chain is only that the
$\pi$-mass of the models it cannot reach in $n$ steps tend to $0$, which for a prior
spread over countably many models is exactly what birth and death moves deliver; the
difficulty described above was created by demanding a \emph{single} $N$ serving all
starting points, and that demand is gone. What is not gone is \hyp{S}{S}, and the
within-model absolute continuity that \hyp{L}{L} still needs at each fixed dimension.
So the remaining question is narrower than it was.

What is needed for the \hyp{P}{P} route itself is a version of
Proposition \ref{prop:minorisation} in which the
overlap constant $\gamma$ is produced from a \emph{chain} of local minorisations
rather than from a single global one. The obstacle is not the construction of the
chain but the constant: each link costs a factor, and the number of links is not
bounded over the class of $M$-dominated laws, so the product may vanish. Recovering a
positive $\gamma$ would require some form of control on how far the class can spread
out, which is a hypothesis of a different type from the ones used here.

\item\label{dir:trajectory} \emph{(S) along the trajectory.}
Every verification of \hyp{S}{S} in these notes is by a bound on the mass of the
singular part that holds \emph{pointwise in the starting point}: $\int r^{n}\dd\mu$
under \hyp{M}{M}, $1-\mu(\XX_{n})$ under \hyp{E}{E}, and the geometric bounds of
Lemma \ref{lem:uniformminorant} in the two applications. Such a bound is unavailable
as soon as the mass of the minorant is positive at every point but not bounded away
from $0$. The instance is parallel tempering without hypothesis
\eqref{eq:temperingtheta}, that is with $\inf_{y}\theta_{k}(y) = 0$ for some
component: the minorant $u$ of Corollary \ref{cor:tempering} still exists, but its
mass $(1-\omega)\prod_{k}\theta_{k}(x_{k})$ may be arbitrarily small, and
Lemma \ref{lem:uniformminorant} does not apply.

The conclusion is nevertheless to be expected, because the chain need not linger where
$\theta$ is small; but establishing it means following the trajectory --- showing that
the time spent in the region $\{\theta < \epsilon\}$ is almost surely finite, or at
least that the mass which never moves vanishes in the limit --- rather than bounding
the singular mass step by step. That is an argument about the process and not about
the kernel, and it is the one place in these notes where the distinction bites.
Example \ref{ex:UnotR} shows that some such argument is genuinely required: there the
minorant has mass $a_{i} > 0$ at every point, the masses are not bounded below, and
\hyp{S}{S} \emph{fails}. So no purely pointwise weakening of
Lemma \ref{lem:uniformminorant} can succeed, and the trajectory has to be looked at.

\end{enumerate}

\subsection{What is lost relative to the general theory}

The theory of $\psi$-irreducible aperiodic Harris chains
\cite{Nummelin84,MeynTweedie09,Douc18} covers all of the above and much besides. Its
central device, the splitting construction of Nummelin and of Athreya and Ney, turns
a minorisation on a small set into genuine regeneration times, and from those one
obtains not only convergence but rates, central limit theorems, and the strong law
for every starting point by way of the renewal theorem. The hypotheses of this note
buy the conclusion of that theory in the total variation metric while dispensing with
the construction; what they do not buy is the construction itself, and hence none of
its quantitative consequences. Two points deserve emphasis. First, the reason the
present route can skip recurrence is not an accident of the proof: a transition
density upgrades irreducibility to positive Harris recurrence automatically
\cite{AsmussenGlynn11}, so under \hyp{D}{D} recurrence is not an additional
hypothesis but a consequence. Second, aperiodicity is likewise not assumed but
derived, in Corollary \ref{cor:selfimprovement}, at a cost of two lines.

\phantomsection
\addcontentsline{toc}{section}{Acknowledgements}
\section*{Acknowledgements}

This note was prepared with the assistance of Claude Opus 5, a large language model developed by Anthropic. Based on the author's preliminary notes and input, the model was used to draft and restructure the exposition, to prepare the typescript, to search for and cross-check references, to work out and check the arguments, and to provide feedback, which the author used to refine further inputs. All references, statements and proofs have been verified by the author, who takes full responsibility for the content, including any remaining errors.

\phantomsection

\end{document}